\documentclass[12pt,letterpaper]{article}
\usepackage{setspace}
\usepackage[margin=1in]{geometry}
\usepackage{amsmath}
\usepackage{url}
\usepackage[normalem]{ulem}
\usepackage{caption}
\usepackage{subcaption}
\usepackage{booktabs}
\usepackage{hyperref}
\usepackage{amssymb}
\usepackage{fourier}
\usepackage[utf8]{inputenc}
\usepackage{framed}
\usepackage{color}
\usepackage{amsthm}
\usepackage{graphicx}
\usepackage{tikz}
\usepackage{tikz-cd}
\usepackage{enumitem}
\usepackage{mathtools}
\usepackage{etoolbox}
\usepackage{extpfeil}
\usepackage{footnote}
\usepackage{authblk}
\makesavenoteenv{tabular}
\makesavenoteenv{table}
\usepackage{makecell}
\usepackage{algorithm}
\usepackage[noend]{algpseudocode}
\makeatletter
\def\BState{\State\hskip-\ALG@thistlm}
\makeatother

\usetikzlibrary{matrix,arrows,decorations.pathmorphing,patterns,shapes.geometric}

\makeatletter
\newtheorem*{rep@theorem}{\rep@title}
\newcommand{\newreptheorem}[2]{%
\newenvironment{rep#1}[1]{%
 \def\rep@title{#2 \ref{##1}}%
 \begin{rep@theorem}}%
 {\end{rep@theorem}}}
\makeatother

\theoremstyle{definition}
\newtheorem{theorem}{Theorem}[section] 
\newreptheorem{theorem}{Theorem}
\newtheorem{lemma}[theorem]{Lemma}
\newtheorem{corollary}[theorem]{Corollary}
\newtheorem{proposition}[theorem]{Proposition}

\newtheorem{example}[theorem]{Example}
\newtheorem{remark}[theorem]{Remark}
\newtheorem{definition}[theorem]{Definition}
\definecolor{darkblue}{rgb}{0.0, 0.0, 0.8}
\definecolor{darkred}{rgb}{0.8, 0.0, 0.0}
\definecolor{darkgreen}{rgb}{0.0, 0.5, 0.0}
\definecolor{orange}{rgb}{1, 0.5, 0.0}

\makeatletter
\newcommand*{\rom}[1]{\expandafter\@slowromancap\romannumeral #1@}
\makeatother

\newcommand{\dmatch}{d_{\mathrm{match}}}
\newcommand{\dm}{\mathrm{dm}}
\newcommand{\B}{\mathcal{B}}
\newcommand{\ob}{\mathrm{ob}}
\newcommand{\EU}{\overline{\mathbf{U}}}
\newcommand{\ER}{\overline{\mathbf{R}}}
\newcommand{\Pb}{\mathbf{P}}
\newcommand{\Q}{\mathbf{Q}}
\newcommand{\vect}{\mathbf{Vec}}
\newcommand{\dhaus}{d_\mathrm{H}}
\newcommand{\gammax}{\gamma_X=(X,d_X(\cdot))}
\newcommand{\gammay}{\gamma_Y=(Y,d_Y(\cdot))}

\newcommand{\dom}{\mathrm{dom}}

\newcommand{\dyndis}{\mathrm{dis}^\mathrm{dyn}}

\newcommand{\Sets}{\mathbf{Sets}}
\newcommand{\Part}{\mathbf{Part}}

\newcommand{\simp}{\mathbf{Simp}}

\newcommand{\dgh}{d_\mathrm{GH}}
\newcommand{\dghdyn}{d_\mathrm{GH}^{\mathrm{dyn}}}
\newcommand{\dl}{d_{\lambda}^{\bullet}}
\newcommand{\F}{\mathbb{F}}

\newcommand{\im}{\mathrm{im}}

\newcommand{\dis}{\mathrm{dis}}

\newcommand{\Zop}{\Z_+^{\mathrm{op}}}
\newcommand{\Rop}{\R^{\mathrm{op}}}
\newcommand{\Qop}{\Q^{\mathrm{op}}}
\newcommand{\Pop}{\Pb^{\mathrm{op}}}
\newcommand{\dintone}{d_{\mathrm{I},1}}
\newcommand{\dinttwo}{d_{\mathrm{I},2}}
\newcommand{\dintthree}{d_{\mathrm{I},3}}
\newcommand{\dintsix}{d_{\mathrm{I},6}}
\newcommand{\bdint}{\mathbf{d}_\mathrm{I}}
\newcommand{\bdintd}{\mathbf{d}_{\mathrm{I},d}}
\newcommand{\dintd}{d_{\mathrm{I},d}}
\newcommand{\dinttwod}{d_{\mathrm{I},2d}}
\newcommand{\dintl}{d_{\mathtt{dyn}}}

\newcommand{\bott}{d_\mathrm{B}}

\newcommand{\ripss}{\mathcal{R}^{\mathrm{lev}}}
\newcommand{\ripsss}{\mathcal{R}_{\bullet}}
\newcommand{\sing}{\mathcal{H}^\mathrm{SL}}

\newcommand{\ba}{\mathbf{a}}
\newcommand{\bb}{\mathbf{b}}
\newcommand{\bc}{\mathbf{c}}
\newcommand{\bt}{[t,t]}

\newcommand{\bu}{I}
\newcommand{\buu}{\mathbf{u}}
\newcommand{\bv}{\mathbf{v}}
\newcommand{\Hrm}{\mathrm{H}}
\newcommand{\T}{\mathbf{R}}
\newcommand{\Z}{\mathbf{Z}}

\newcommand{\R}{\mathbf{R}}
\newcommand{\U}{\mathbf{Int}}
\newcommand{\UU}{\mathbf{U}}
\newcommand{\N}{\mathbf{N}}
\newcommand{\C}{\mathcal{C}}
\newcommand{\eps}{\varepsilon}
\newcommand{\dint}{d_{\mathrm{I}}}
\newcommand{\dero}{d_{\mathrm{E}}}
\newcommand{\dgm}{\mathrm{dgm}}
\newcommand{\rips}{\mathcal{R}_\delta}
\newcommand{\sets}{\mathbf{Sets}}
\newcommand{\lmulti}{\{\!\!\{}
\newcommand{\rmulti}{\}\!\!\}}
\newcommand{\abs}[1]{\left\lvert#1\right\rvert}
\newcommand{\norm}[1]{\left\lVert#1\right\rVert}
\newcommand{\tripod}{R:\ X \xtwoheadleftarrow{\varphi_X} Z \xtwoheadrightarrow{\varphi_Y} Y}

\newcommand{\rk}{\mathrm{rk}}

\title{Spatio-temporal Persistent Homology for Dynamic Metric Spaces}
\author[1]{Woojin Kim}
\author[2]{Facundo M\'emoli}
\affil[1]{Department of Mathematics,
 	The Ohio State University\\
 	\texttt{kim.5235@osu.edu}}
 \affil[2]{Department of Mathematics and Department of Computer Science and Engineering,
 	The Ohio State University\\ 
 	\texttt{memoli@math.osu.edu}}

\begin{document}
\maketitle

\begin{abstract}
Characterizing the dynamics of time-evolving data within the framework of topological data analysis (TDA) has been attracting increasingly more attention.
Popular instances of time-evolving data include flocking/swarming behaviors in animals and social networks in the human sphere. A natural mathematical model
for such collective behaviors is a dynamic point cloud, or more generally a dynamic metric space (DMS). 

In this paper we  extend the Rips filtration stability result for (static) metric spaces to
the setting of DMSs. We do this by devising a certain  three-parameter "spatiotemporal" filtration of a DMS. Applying the homology functor to this filtration gives rise to 
multidimensional persistence module derived from the DMS. We show that this multidimensional module  enjoys stability under a suitable generalization of the Gromov-Hausdorff distance which permits metrizing the collection of all DMSs.

On the other hand, it is recognized that, in general, comparing two multidimensional persistence modules leads to  intractable computational problems. For the purpose of practical comparison of DMSs, we focus on both the rank invariant or the dimension function of the multidimensional persistence module that is derived from a DMS.
We specifically propose to utilize a certain metric $d$ for comparing these invariants: In our work this $d$ is either (1) a certain generalization 
of the erosion distance by Patel, or  (2) a specialized version of the well known interleaving distance. In either case, the metric $d$ can be computed in polynomial time.

\end{abstract}

\tableofcontents


\section{Introduction}

\paragraph{Stability and tractability of TDA for studying metric spaces.} Finite point clouds or  finite metric spaces are amongst the most common  data representations considered in topological data analysis (TDA) \cite{Carl09,edelsbrunner2008persistent,ghrist2008barcodes}. In particular, the stability of the Single Linkage Hierarchical Clustering (SLHC) method \cite{clustum} or the stability of the  persistent homology of filtered Rips complexes built on metric spaces \cite{dghrips,chazal2014persistence} motivates adopting these constructions when studying metric spaces arising in applications.

Whereas there has been extensive applications of TDA to static metric data (thanks to the aforementioned theoretical underpinnings), there is not much study of \emph{dynamic metric} data \emph{from the TDA perspective}. Our motivation for considering dynamic metric data stems from the study and characterization
of flocking/swarming behaviors of animals \cite{benkert2008reporting,gudmundsson2006computing,gudmundsson2007efficient,huang2008modeling,li2010swarm,parrish1997animal,sumpter-collective,vieira2009line}, convoys \cite{jeung2008discovery}, moving clusters \cite{kalnis2005discovering}, or mobile groups \cite{hwang2005mining,wang2008efficient}. In this paper,  by extending ideas from \cite{clustum,dghrips,chazal2014persistence,kim2017stable,kim2018CCCG}, we aim at establishing a TDA framework for the study of dynamic metric spaces (DMSs) which comes together with stability theorems. We begin by describing and comparing relevant work with ours.

\paragraph{Lack of  an adequate metric for DMSs.} In \cite{munch2013applications}, Munch considers  \emph{vineyards} --- a certain notion of time-varying persistence diagrams introduced by  Cohen-Steiner et al. \cite{CEM06} ---  as signatures for dynamic point clouds. 
Munch, in particular, shows that vineyards are stable\footnote{Under a certain  notion of distance arising from in the integration over time of the bottleneck distance between the instantaneous persistence diagrams.} \cite{cohen2007stability} under perturbations of the input dynamic point cloud \cite[Theorem 17]{munch2013applications}. 
However, we will observe below that, for the purpose of comparing two DMSs (which we regard as models of flocking behaviors), the metrics that directly arise as the integration of the Hausdorff or Gromov-Haussdorff distance can sometimes fail to be discriminative enough (see Example \ref{ex:weak isomorphic} and Remark \ref{rem:sensitivity}).

In \cite{topaz}, Halverson, Topaz and Ziegelmeier study aggregation
models for biological systems by adopting ideas from TDA. They show that topological analysis of aggregation reveals dynamical events which are not captured by classical analysis methods. Specifically, in order to extract insights
about the global behavior of dynamic point clouds obtained by simulating aggregation models, they employ the so-called \emph{CROCKER\footnote{Contour Realization Of Computed $k$-dimensional hole Evolution in the Rips
complex} plot}. This plot represents the evolution of Betti numbers of Rips complexes over the plane of \emph{time} and \emph{scale} parameters. In \cite{ulmer2018assessing}, Topaz, Ulmer and Ziegelmeier discretize CROCKER plots as matrices and make use of Frobenius norm for comparing any two such matrices. In \cite{topaz,ulmer2018assessing}, the authors do not provide stability results for CROCKER plots derived from biological aggregation models.

\begin{figure}
    \centering
    \includegraphics[width=0.9\textwidth]{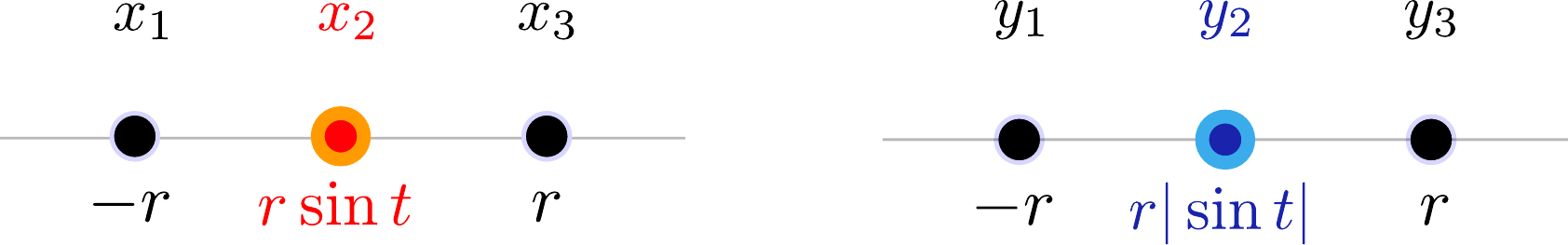}
    \caption{Fix $r>0$. The two figures above stand for two dynamic point clouds $X_r(\cdot)$ and  $Y_r(\cdot)$ in the real line each consisting of $3$ points $x_1,x_2,x_3$ and $y_1,y_2,y_3$,  respectively. Each of $X_r(\cdot)$ and $Y_r(\cdot)$ contains (1) two static points located at $-r$ and $r$ respectively ($x_1,x_3$ and $y_1,y_3$), and (2) one dynamic point with the time-dependent coordinate either $r\sin t$ or $r\abs{\sin t}$, $t\in \R$ ($x_2$ and $y_2$). Observe that in $X_r(\cdot)$ the unique dynamic point $x_2$ meets both of $x_1$ and $x_2$ periodically. On the contrary, in $Y_r(\cdot)$, the unique dynamic point $y_2$ meets only $y_3$ periodically.}
    \label{fig:intro}
\end{figure}

\paragraph{Motivation for introducing a new metric for DMSs.}  Consider the two dynamic point clouds $X_r(\cdot)$ and  $Y_r(\cdot)$ illustrated as in Figure \ref{fig:intro}. Let us regard them as instances of DMS with the time-dependent metrics obtained by restricting the Euclidean metric on $\R^2$ at each time $t\in\R$. 

Observe that for each time $t\in\R$, the metric spaces $X_r(t)$ and $Y_r(t)$ are \emph{isometric} and hence the Gromov-Hausdorff distance \cite[Ch.7]{burago} $\dgh\left( X_r(t),Y_r(t) \right)$ is \emph{zero}. This in turn implies that the integral $\int_{t\in\R}\ \dgh\left( X_r(t),Y_r(t) \right)\ dt$ is also \emph{zero}, implying that $X_r(\cdot)$ and $Y_r(\cdot)$ are not distinguished from each other by the \emph{integrated} Gromov-Hausdorff distance. \footnote{In \cite{munch2013applications}, in order to compare two dynamic point clouds, Munch considered the \emph{integrated} Hausdorff distance $\int \dhaus$ over time. Since the metric $\int \dhaus$ takes account of relative position of two dynamic point clouds inside an ambient metric space, we do not consider utilizing $\int \dhaus$ for the purpose of comparing \emph{intrinsic behaviors} of two dynamic metric data.  
Also, Munch considered the \emph{integrated} bottleneck distance $\int \bott$ by computing the Rips filtrations of dynamic point clouds at each time. However, by \cite[Theorem 3.1]{dghrips}, the metric  $\int \bott$ is upper-bounded by (twice) the integrated Gromov-Hausdorff distance, which in this case vanishes.  Therefore, $\int \bott$ does not discriminate the two dynamic point clouds given as in Figure \ref{fig:intro}.} See Remark \ref{rem:weak gh}.

However, regarding $X_r(\cdot)$ and $Y_r(\cdot)$ as models of \emph{collective behaviors} of animals,vehicles or people, $X_r(\cdot)$ and $Y_r(\cdot)$ are clearly distinct from each other. This motivates us to seek an adequate metric that measures the difference between the \emph{dynamics} underlying any two given DMSs. In particular, this metric should \emph{not} be a mere sum of instantaneous differences of the given DMSs over time.  

In this paper, we adopt $\dintl$, called the \emph{$\lambda$-slack interleaving distance} with $\lambda=2$ (Definition \ref{def:lambda dist}, originally introduced in \cite{kim2017stable}), as a measure of the behavioral difference between DMSs.  In Section \ref{sec:overview}, we specifically show that the metric $\dintl$ returns a positive value for the pair of DMSs $X_r(\cdot)$ and $Y_r(\cdot)$ in Figure \ref{fig:intro}, demonstrating its sensitivity.

\paragraph{About stability and tractability of $\dintl$.} Even though the metric $\dintl$ is able to differentiate subtly different DMSs (Theorem \ref{thm:lambda slack distance}), computing $\dintl$ is \emph{not} tractable in general (Remark \ref{rem:dintl computational complexity}). This hinders us from utilizing $\dintl$ in practice. Therefore, as a pragmatic approach, we adopt the comparison of \emph{invariants} of DMSs, rather than directly comparing DMSs . To this end, 

\begin{itemize}
    \item[(a)] the invariants \emph{must} be stable under perturbations of the input DMS, and 
    \item[(b)]  the metric for comparing two invariants extracted from two DMSs \emph{must} be efficiently computable.
\end{itemize}
\paragraph{Contributions.} In this work, we achieve both items (a) and (b) above,  described as follows. 

With regard to (a), we first extract invariants from a given DMS, where these invariants are in the form of \emph{3-dimensional} persistence modules of sets or vector spaces. These are obtained from a  blend of ideas related to the Rips filtration \cite{cohen2007stability,dghrips,comptopo-herbert}, the \emph{single linkage hierarchical clustering (SLHC)} method \cite{clustum}, and the interlevel set persistence/categorified Reeb graphs \cite{bendich2013homology,botnan2018algebraic,carlsson2009zigzag,de2016categorified}.

We are able to prove the stability of these invariants (Theorems \ref{thm:main2} and \ref{thm:stability of spatiotemporal dendrogram}) by adapting ideas from \cite{clustum,dghrips,chazal2014persistence}. We specifically emphasize that our stability results are a generalization of the well known stability theorems for the SLHC method \cite{clustum} and the Rips filtration of a metric space \cite{dghrips,chazal2014persistence}: Indeed, we show that by restricting ourselves to the class of \emph{constant} DMSs, our results reduce to the standard stability theorems for static metric spaces in \cite{clustum,dghrips,chazal2014persistence}.

Next, in regard to item (b) above, we address the issue of computability of the metric between invariants of DMSs. In \cite{bjerkevik2018computing,bjerkevik2017computational}, Bjerkevik and Botnan show that computing the interleaving distance $\dint$ \cite{lesnick} between multidimensional persistence modules can in general be NP-hard. Also, since we are not  guaranteed to have  interval decomposability \cite{botnan2018algebraic,carlsson2009theory} of the  $3$-dimensional modules considered in this paper, we are not in a position to utilize the bottleneck distance and relevant algorithms developed by Dey and Xin  \cite{dey2018computing} instead of $\dint$. 

This motivates us to further simplify our invariant $M_{X}$ associated to a DMS $(X,d_X(\cdot))$, which is in the form of $3$-dimensional persistence module. We focus on both the \emph{dimension function} and the \emph{rank function}. 
The dimension function $\dm(M_{X})$ of a persistence module $M_{X}$ has been studied in various contexts and with various names such as Betti curve, feature counting function, etc,  \cite{babichev2018robust,dey2018computing,giusti2016two,giusti2015clique,kahle2013limit,scolamiero2017multidimensional}. The rank function $\rk(M_{X})$ of $M_{X}$ has also been extensively considered  \cite{carlsson2009theory,cerri2013betti,landi2018rank,patel2018generalized,puuska2017erosion}. 
We observe that both of these functions  (1) can themselves be computed in polynomial time, (2) can be compared to each other via the interleaving distance $\dint^{\Z}$ for \emph{integer-valued functions} (see Section \ref{subsec:interleaving}) and (3) are stable to perturbations of $(X,d_X(\cdot))$ under $\dintl$ (Theorems \ref{thm:rank k stability} and \ref{thm:betti-0 stability}). We also propose a simple algorithm for computing  $\dint^{\Z}$ in poly-time (Section \ref{sec:algorithm}). Therefore, we can bound the distance $\dintl$ in poly-time by computing $\dint^{\Z}$ and either $\dm(\cdot)$ or $\rk(\cdot)$.

We in particular emphasize that our method for computing $\dint^\Z$ provides a  poly-time algorithm for bounding from below the interleaving distance between \newline $d$-dimensional persistence modules $M$ of vector spaces  \emph{without any restriction on  $d$ or on the structure of $M$} (even if $M$ is not derived from a DMS).

\paragraph{Other related work.} Aiming at analyzing/summarizing trajectory data such as the movement of animals,
vehicles, and people, Buchin and et al. introduce the notion of \emph{trajectory grouping structure} \cite{buchin2013trajectory}. This is a summarization, in the form of a labeled Reeb graph, of a set of points having piecewise linear trajectories with time-stamped
vertices in Euclidean space $\R^d$. This work was subsequently enriched in \cite{kostitsyna2015trajectory,van2016grouping,van2015central,van2016refined}.

In \cite{kim2018CCCG,kim2017stable}, the thread of ideas in \cite{buchin2013trajectory} is blended with ideas in zigzag persistence theory \cite{zigzag}. Specifically, particular cases of  trajectory grouping structure in \cite{buchin2013trajectory}, are named \emph{formigrams}. By clarifying the zigzag persistence structure of formigrams, formigrams are further summarized into \emph{barcodes}. Regarding the barcode as a signature of a set of trajectory data, the authors of \cite{kim2018CCCG,kim2017stable} utilize these barcodes for carrying out the classification task of a family of synthetic flocking behaviors \cite{zane}.   

The central results in \cite{kim2018CCCG,kim2017stable} show that  barcodes or formigrams from a trajectory data are stable to perturbations of the input data \cite[Theorem 5]{kim2018CCCG},\cite[Theorem 9.21]{kim2017stable}. This  work is a sequel to  \cite{kim2018CCCG,kim2017stable}. Namely, by considering Rips-like filtrations parametrized  both by \emph{time intervals} and \emph{spatial scale}, we obtain novel stability results in every homological dimension.

Other work utilizing TDA-like ideas in the analysis of dynamic data includes: a study of time-varying merge trees or time-varying Reeb graphs \cite{edelsbrunner2008time,oesterling2015computing}. Also, ideas of persistent homology are utilized in the study of time-varying graphs \cite{hajij2018visual}, discretely
sampled dynamical systems \cite{bauer2017persistence,edelsbrunner2015persistent} or in the study of combinatorial dynamical systems \cite{dey2018persistent}.

\paragraph{Organization.} In Section \ref{sec:dms} we review the notion of DMSs and the metric $\dintl$ on DMSs. In Section \ref{sec:interleaving for various categories} we review the interleaving distance. In Section \ref{sec:overview} we provide an overview of our new stability results about persistent homology features of DMSs. In Section \ref{sec:algorithm} we propose and study an algorithm for computing the interleaving distance between integer-valued functions. Section 6 contains proofs of statements (theorems, etc.) from Section 4.

In Section \ref{sec:discretization of a DMS} we describe how to analyze and compare discrete time series of metric data. In Section \ref{sec:Crocker} we clarify the relationship between the rank invariants of DMSs and the CROCKER-plots of DMSs. In Section \ref{sec:other metrics} we compare the interleaving distance between integer-valued functions with other relevant metrics. In Section \ref{sec:static metric spaces} we review the stability of the single linkage hierarchical clustering (SLHC) method for static metric spaces; results in this section are generalized to those in Section \ref{sec:details about SC}. 

\paragraph{Acknowledgement.} FM thanks Justin Curry and Amit Patel for beneficial discussions. This work was partially supported by NSF grants IIS-1422400, CCF-1526513, DMS-1723003, and CCF-1740761.


\section{Dynamic metric spaces (DMSs)}\label{sec:dms}

Throughout this paper, we fix a certain field $\F$ and only consider vector spaces over $\F$ whenever they arise. Any simplicial homology has coefficients in $\F$.  By $\Z_+$ and $\R_+$, we denote the set of non-negative integers and the set of non-negative reals, respectively. 

\subsection{Definition of DMSs}

\paragraph{DMSs.} A DMS $\gammax$ stands for a pair of finite set $X$ with $\R$-parametrized metric $d_X(\cdot):\R\times X\times X \rightarrow \R_+$: for each $t\in \R$, a certain (pseudo-)metric $d_X(t):X\times X \rightarrow \R_+$ is obtained:

\begin{definition}[Dynamic metric spaces {\cite{kim2017stable}}]\label{def:dms} A \emph{dynamic metric space}  is a  pair $\gamma_X = (X,d_X(\cdot))$ where $X$ is a non-empty finite set and $d_X(\cdot):\T\times X\times X\rightarrow \R_+$ satisfies:\begin{enumerate}[label=(\roman*)]
		\item For every $t\in\T$, $\gamma_X(t)=(X,d_X(t))$ is a pseudo-metric space.\label{item:dynamic1}		
		\item There exists $t_0\in\T$ such that $\gamma_X(t_0)$ is a metric space.\footnote{This condition is assumed since otherwise one could substitute the  DMSs $\gamma_X$ by  another DMSs $\gamma_{X'}$ over a set $X'$ which satisfies $|X'|<|X|$, and such that $\gamma_{X'}$ is point-wisely equivalent to $\gamma_X$.}\label{item:dynamic2}	
		\item For fixed $x,x'\in X$,	$d_X(\cdot)(x,x'):\T\rightarrow \R_+$  is continuous. \label{item:dynamic3}	
	\end{enumerate}
	We refer to $t$ as the \emph{time} parameter.
\end{definition}

Let $(\mathcal{M},\dgh)$ be the collection of all finite (pseudo-)metric spaces equipped with the Gromov-Hausdorff distance (Definition \ref{def:the GH}). Any DMS $\gammax$ can be seen as a continuous curve from $\R$ to $(\mathcal{M},\dgh)$. 

\begin{example}[{\cite{kim2017stable}}]\label{ex:constant} 
Examples of DMSs include: 
	\begin{enumerate}[label=(\roman*)]
		\item (Constant DMSs) Given a finite metric space $(X,d_X)$, define the DMS $\gamma_X=(X,d_X'(\cdot))$ by declaring that for all $t\in \R$, $d_X'(t)=d_X$ as a function $X\times X\rightarrow \R_+$. We refer to such $\gamma_X$ as a \emph{constant} DMS and simply write $\gamma_X\equiv (X,d_X)$.\label{item:constant1}
\item (Dynamic point clouds) A family of examples is given by $n$ points moving continuously inside an ambient metric space $(Z,d_Z)$ where particles are allowed to coalesce. 
		If the $n$ trajectories are $x_1(t),\ldots,x_n(t)\in Z$, then let $X:=\{1,\ldots,n\}$ and define the DMS $\gamma_X := (X,d_X(\cdot))$ as follows: for $t\in \T$ and $i,j\in\{1,\ldots,n\}$, let $d_X(t)(i,j):=d_
		Z(x_i(t),x_j(t)).$  We call $\gamma_X$ a \emph{dynamic point cloud} in $Z$ and simply write  $X(\cdot)=\{x_i(\cdot)\}_{i=1}^n$ or $X(\cdot)$. \label{item:dynamic point cloud}
	\end{enumerate}
\end{example}

\paragraph{Weak and strong isomorphism between DMSs.} We introduce two different notions of  \emph{isomorphism} between DMSs.                                                                                                                                                                 
\begin{definition}[Isomorphism between DMSs]\label{def:isomorphism} Let $\gammax, \gammay$ be any two DMSs. 
	\begin{enumerate}[label=(\roman*)]
		\item $\gamma_X$ and $\gamma_Y$ are \emph{strongly isomorphic} if there exists a bijection $\varphi:X\rightarrow Y$ such that $\varphi$ is an isometry between $\gamma_X(t)=(X,d_X(t))$ and $\gamma_Y(t)=(Y,d_Y(t))$  for all $t\in \R$.\label{item:strong}
		
		\item $\gamma_X$ and $\gamma_Y$ are \emph{weakly isomorphic} if for each $t\in \T$, $\gamma_X(t)=(X,d_X(t))$ is isometric to $\gamma_Y(t)=(Y,d_Y(t))$.\label{item:weak}
	\end{enumerate}
\end{definition}
Any two strongly isomorphic DMSs are weakly isomorphic, but the converse is not true: 

\begin{example}[Weakly isomorphic DMSs] \label{ex:weak isomorphic} 
The dynamic point clouds $X_r(\cdot)$ and $Y_r(\cdot)$ described in Figure \ref{fig:intro} are weakly isomorphic, but not strongly isormorphic: Indeed, there is no bijection between $\{x_1,x_2,x_3\}$ and $\{y_1,y_2,y_3\}$ which serves as an isometry for all $t\in\R$.

\end{example}  

\subsection{The $\lambda$-slack interleaving distance between DMSs}
We review the extended metric $\dintl$ for DMSs, which was introduced in \cite[Definition 9.13]{kim2017stable} under the name of \emph{$\lambda$-slack interleaving distance}, for each $\lambda\in [0,\infty)$.

\begin{definition}Let $\eps\geq 0$. Given any map $d:X\times X\rightarrow \R$, by $d+\eps$ we denote the map $X\times X\rightarrow \R$ defined as $(d+\eps)(x,x')=d(x,x')+\eps$ for all $(x,x')\in X\times X.$ 
\end{definition}

In order to compare any two DMSs, we will utilize the notion of \emph{tripod}: 
\begin{definition}[Tripod]\label{def:tripod} Let $X$ and $Y$ be any two non-empty sets. For another set $Z$, any pair of \emph{surjective} maps $\tripod$ is called a \emph{tripod} between $X$ and $Y$.  
\end{definition}

Given any map $d:X\times X \rightarrow \R$, let $Z$ be any set and let $\varphi:Z\rightarrow X$ be any map. Then, we define $\left(\varphi^*d\right):Z\times Z\rightarrow \R$ as
\[\left(\varphi^*d\right)(z,z'):= d\left(\varphi(z),\varphi(z')\right), \ \ \mbox{ $(z,z')\in Z\times Z.$}
\]

\begin{definition}[Comparison of metrics via tripods] Consider any two maps $d_1:X \times X\rightarrow \R$ and $d_2:Y\times Y \rightarrow \R$. Given a tripod $\tripod$ between $X$ and $Y$, by \[d_1\leq_R d_2,\] we mean $\left(\varphi_X^*d_1\right)(z,z')\leq \left(\varphi_Y^*d_2\right)(z,z')$ for all $(z,z')\in Z\times Z$. 
\end{definition}

Let $\U$ be the collection of all finite \emph{closed} intervals of $\R$.  See Figure \ref{fig:intervals}. 

\begin{definition}[Time-interlevel analysis of a DMS]\label{def:bigvee}Given a DMS $\gammax$, define the function $\bigvee d_X:\U\times X\times X\rightarrow \R_+$ as \[\left(I,x,x'\right)\mapsto \bigvee_{I} d_X(x,x'):= \min_{s\in I}d_X(s)(x,x').\]
\end{definition} 

In words, $\bigvee_{I} d_X(x,x')$ stands for the minimum distance between $x$ and $x'$ within the time interval $I$. Observe that if $I\subset I'$ are both in $\U$, then $\bigvee_{I'}d_X(x,x')\leq \bigvee_{I}d_X(x,x')$ for all $x,x'\in X$.

For any $t\in \T$, let $[t]^\eps:=[t-\eps,t+\eps]\in \U$. 

\begin{figure}
    \centering
    \includegraphics[width=0.3\textwidth]{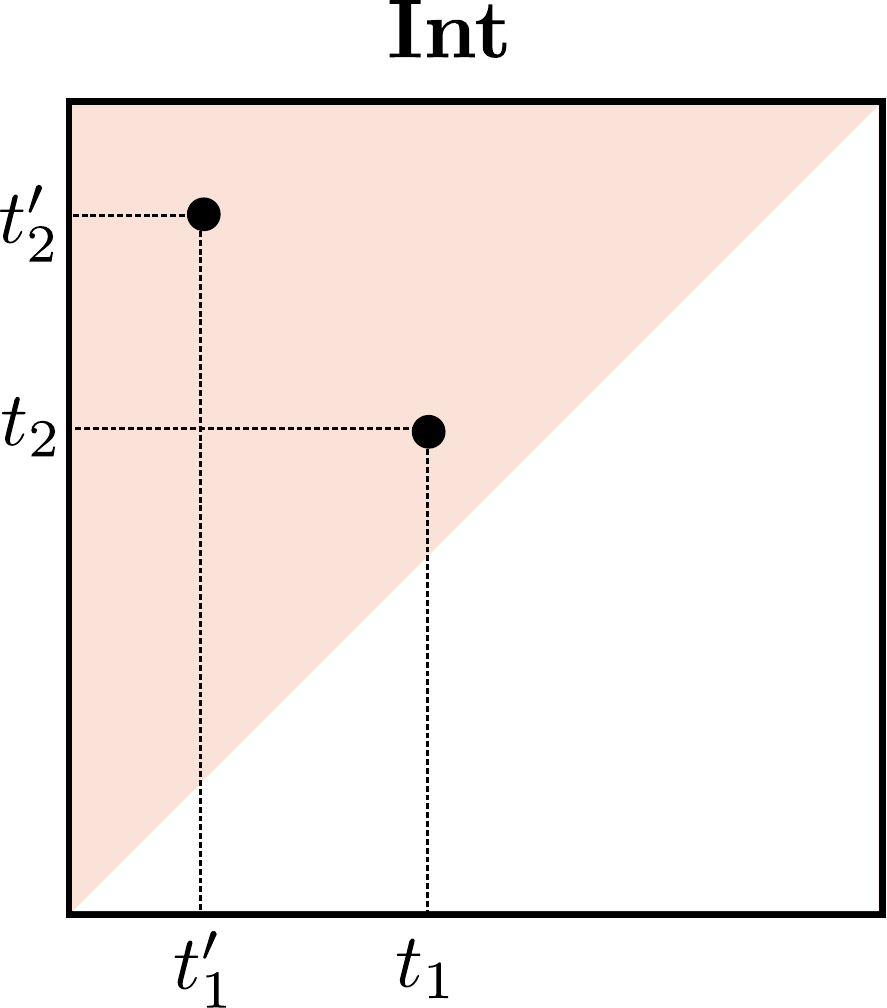}
    \caption{The collection $\U$ can be graphically represented as the upper-half plane $\{(t_1,t_2)\in\R^2: t_1 \leq t_2 \}$: Any closed interval $[t_1,t_2]$ of $\R$ is identified with the point $(t_1,t_2)$ on $\R^2$. Observe that if $[t_1,t_2]\subset [t_1',t_2']$, then the point $(t_1',t_2')$ is located at upper-left of the point $(t_1,t_2)$ in the plane.}
    \label{fig:intervals}
\end{figure}

\begin{definition}[Distortion of a tripod]\label{def:distortion}
Let $\gammax$ and \hbox{$\gammay$} be any two DMSs. Let $\tripod$ be a tripod between $X$ and $Y$ such that 
	\begin{equation}\label{eq:distor}
	\mbox{for all  $t\in \T$},\ \ \bigvee_{[t]^\eps}d_X\leq_R d_Y(t)+2\eps\ \  \mbox{and}\ \ \bigvee_{[t]^\eps}d_Y\leq_R d_X(t)+2\eps.	
	\end{equation}
	We call any such $R$ an \emph{$\eps$-tripod} between $\gamma_X$ and $\gamma_Y$. Define the \emph{distortion} $\dyndis(R)$ of $R$ to be the infimum of $\eps\geq 0$ for which $R$ is an $\eps$-tripod.
\end{definition}
In Definition \ref{def:distortion}, if $R$ is a $\eps$-tripod, then $R$ is also a $\eps'$-tripod for any $\eps'\geq\eps$.

\begin{definition}[The distance $\dintl$ between DMSs]\label{def:lambda dist}Given any two DMSs \newline $\gammax$ and $\gammay$, we define
	\[\dintl(\gamma_X,\gamma_Y):=\min_R\dyndis(R), \]
	where the minimum ranges over all tripods between $X$ and $Y$. 
\end{definition}

We remark that $\dintl$ is a hybrid between the Gromov-Hausdorff distance (Definition \ref{def:the GH}) and the interleaving distance \cite{bubenik2014categorification,CCG09} for Reeb graphs \cite{de2016categorified}. We also remark that, in \cite{kim2017stable}, $\dintl$ is introduced under the name of \emph{$\lambda$-slack interleaving distance} for $\lambda=2$. We use $\lambda=2$ in this paper for ease of notation.  This choice is not significant because different choices of $\lambda>0$ yield bilipschitz equivalent metrics for DMSs \cite[Proposition 11.29]{kim2017stable}. 

Any DMS $\gammax$ is said to be \emph{bounded} if there exists $r\in[0,\infty)$ such that for all $x,x'\in X$ and all $t\in \R,$ $d_X(t)(x,x')\leq r.$ For example, both DMSs given in Figure \ref{fig:intro} are bounded. 

\begin{theorem}[{\cite[Theorem 9.14]{kim2017stable}}]\label{thm:lambda slack distance}
$\dintl$ is an \emph{extended} metric between DMSs modulo strong isomorphism (Definition \ref{def:isomorphism} \ref{item:strong}). In particular,
 $\dintl$ is a metric between bounded DMSs modulo strong isomorphism.
\end{theorem}

\begin{remark}[$\dintl$ generalizes the Gromov-Hausdorff distance {\cite[Remark 11.28]{kim2017stable}}]\label{rem:gromov generalization}  Given any two constant DMSs $\gamma_X\equiv (X,d_X)$ and $\gamma_Y\equiv (Y,d_Y)$, the metric $\dintl$ recovers the Gromov-Hausdorff distance between $(X,d_X)$ and $(Y,d_Y)$. Indeed, for any tripod $R$ between $X$ and $Y$, condition  (\ref{eq:distor}) reduces to 
	\[\abs{d_X(x,x')-d_Y(y,y')}\leq 2\eps \ \ \mbox{for all $(x,y),(x',y')\in R$}.\] Therefore, \[\dgh((X,d_X),(Y,d_Y))=\dintl(\gamma_X,\gamma_Y).\]	
\end{remark}

\begin{remark}\label{rem:dintl computational complexity}
From Remark \ref{rem:gromov generalization}, we conclude that the computation of $\dintl$ is in general not tractable: On the class of constant DMSs the metric $\dintl$ reduces to the Gromov-Hausdorff distance, which leads to NP-hard problems \cite{agarwal2015computing,schmiedl14shape,schmiedl2017computational}.
\end{remark}

\subsection{Variants of $\dintl$}

Recall that $\dintl$ is the $\lambda$-slack interleaving distance for $\lambda=2$. Here we discuss a variant of the $\lambda$-slack interleaving distance which arises from a slightly different way of incorporating the $\lambda$ parameter:

\begin{definition}[Multiplicative $\lambda$-slack interleaving distance]\label{def:dl} For $\lambda\in (0,\infty)$, we  define the \emph{multiplicative} $\lambda$-slack interleaving distance $\dl(\gamma_X,\gamma_Y)$ between two DMSs $\gammax$ and $\gammay$ as the infimum $\eps$  for which there exists a tripod $R$ between $X$ and $Y$ such that\footnote{In \cite{kim2017stable}, the original $\lambda$-slack interleaving distance $d_{\lambda}(\gamma_X,\gamma_Y)$, $\lambda\in [0,\infty)$ is defined as the infimum amount of time $\eps$ for which there exists a tripod $R$ between $X$ and $Y$ such that \[	\mbox{for all  $t\in \T$},\ \ \bigvee_{[t]^{\eps}}d_X\leq_R d_Y(t)+\lambda\eps \  \mbox{and}\ \ \bigvee_{[t]^\eps}d_Y\leq_R d_X(t)+\lambda\eps.\] In this original definition, the units of $\lambda$ is (distance units)/(time unit), whereas the units of $\lambda$ for $\dl$ is (time units)/(distance units).}	
	\begin{equation}\label{eq:distor2}
	\mbox{for all  $t\in \T$},\ \ \bigvee_{[t]^\frac{\eps}{\lambda}}d_X\leq_R d_Y(t)+\eps\ \  \mbox{and}\ \ \bigvee_{[t]^\frac{\eps}{\lambda}}d_Y\leq_R d_X(t)+\eps.	
	\end{equation}
\end{definition}

\begin{definition}[dyn-Gromov-Hausdorff distance between DMSs and its relation to $\dl$]  Let $\gamma_X$ and $\gamma_Y$ be any two DMSs and fix a tripod $R$ between $X$ and $Y$. For each $t\in \T$, let 
\[\dis(R)(t):=\inf\{\eps\in \R_+: d_X(t)\leq_R d_Y(t)+\eps\ \mbox{and}\  d_Y(t)\leq_R d_X(t)+\eps \}. \]
Define
$$\displaystyle \dghdyn(\gamma_X,\gamma_Y):=\min_R\sup_{t\in \T}\ \dis(R)(t),$$ where the minimum is taken over all tripods $R$ between $X$ and $Y$. We call this distance the \emph{dyn-Gromov-Hausdorff distance between $\gamma_X$ and $\gamma_Y$}.

\end{definition}
Note that, for the multiplicative interleaving distance $\dl$ in Definition \ref{def:dl}, we have \[\lim_{\lambda \rightarrow \infty}\dl(\gamma_X,\gamma_Y)= \dghdyn(\gamma_X,\gamma_Y).\] 
  Also, note that $\dghdyn$ between constant DMSs $\gamma_X\equiv (X,d_X)$ and $\gamma_Y\equiv (Y,d_Y)$ reduces to twice the Gromov-Hausdorff distance between $(X,d_X)$ and $(Y,d_Y)$. 
We remark that $\dghdyn$ is in general \emph{not} the supremum of the  Gromov-Hausdorff distances $\dgh(\gamma_X(t),\gamma_Y(t))$ over all times  $t\in\T$. Specifically, we have the following inequality:
$$
    L_{\mathrm{GH}}^{(\infty)}(\gamma_X,\gamma_Y):=\sup_{t\in \T}\ \dgh(\gamma_X(t),\gamma_Y(t))=\frac{1}{2}\cdot\sup_{t\in \T}\min_R\ \dis(R)(t)\stackrel{(*)}{\leq} \frac{1}{2}\cdot\min_R\sup_{t\in \T}\ \dis(R)(t)=\frac{1}{2}\cdot\dghdyn(\gamma_X,\gamma_Y).
$$
The inequality denoted by $(\ast)$ above is often strict, as it is to be expected as a result of swapping the $\sup \min$ implicit in $L^{(\infty)}_{\mathrm{GH}}$ for the $\min \sup$ in the definition of $\dghdyn$.\footnote{The quantity in the LHS is allows for picking a different correspondence for each time $t$ whereas the RHS demands that a single correspondence is adequate for all times.} For instance, for any pair $\gamma_X,\gamma_Y$ of weakly isomorphic but not strongly isomorphic DMSs (cf. Example \ref{ex:weak isomorphic}), one has that (1) $\dgh(\gamma_X(t),\gamma_Y(t))=0$ for every $t\in \T$ and in turn $\sup_{t\in \T}\ \dgh(\gamma_X(t),\gamma_Y(t))=0$; but in contrast (2)  $\dghdyn(\gamma_X,\gamma_Y)$ is strictly positive.

It is possible to give rise to a whole family of pseudo-distances of which $L_{\mathrm{GH}}^{(\infty)}$ is a particular example.  

This construction is analogous to the \emph{integrated Hausdorff distance} between dynamic point clouds considered in \cite{munch2013applications}.
\begin{remark}[Weak-$L^p$-Gromov-Hausdorff distance]\label{rem:weak gh}Fix any two DMSs $\gamma_X$ and $\gamma_Y$. For any fully supportted probability measure 
$\zeta $on $\R$ and $p\in[1,\infty)$, define
\[ L^{(p)}_{\mathrm{GH},\zeta}(\gamma_X,\gamma_Y):=\left(\int_{t\in\R}\big(\dgh(\gamma_X(t),\gamma_Y(t))\big)^p d\zeta\right)^\frac{1}{p}.
\]

It is clear that $L^{(p)}_{\mathrm{GH},\zeta}(\gamma_X,\gamma_Y)$ vanishes whenever $\gamma_X$ and $\gamma_Y$ are weakly isomorphic.
\end{remark}

\subsection{Persistent homology features of a DMS}\label{sec:analysis of dmss}
We extend ideas from persistent homology/single linkage hierarchical clustering method for metric spaces (Section \ref{sec:static metric spaces}) to the setting of \emph{dynamic} metric spaces (DMSs).

\paragraph{Posets and their opposite.} Given any poset $\Pb=(\Pb,\leq)$, we regard $\Pb$ as the category: Objects are the elements of $\Pb$. Also, for any $p,q\in \Pb$, there exists the unique morphism $p\rightarrow q$ if and only if $p\leq q$. Since there exists at most one morphism between any two elements of $\Pb$, the category $\Pb$ is called \emph{thin} and, \emph{any} closed diagram in $\Pb$ must commute. We sometimes consider the \emph{opposite} category of $\Pb$, which will be denoted by $\Pop$. In the category $\Pop$, for $p,q\in \Pb$, there exists the unique morphism $p\rightarrow q$ if and only if $p\geq q$.

\begin{example}[$\U$] Recall the collection $\U$ of all finite closed intervals of $\R$. We regard $\U$ as poset, where the order $\leq$ is the  inclusion $\subseteq$. Hence, $\U$ can be seen as the category of finite closed real intervals whose morphisms are inclusions.
\end{example}

\paragraph{Product of posets.} Given any two posets $\Pb$ and $\Q$, we assume by default that their product $\Pb\times \Q$ is equipped with the partial order $\leq$ defined as $(p,q)\leq (p'.q')$ if and only if $p\leq p'$ in $\Pb$ and $q\leq q'$ in $\Q$.

\begin{remark}\label{rem:subposet} In the poset $\U\times \R_+$, we have $(I,\delta)\leq (I',\delta')$ if and only if $I\subset I'$ and $\delta\leq \delta'$. See Figure \ref{fig:3d filtration}. We will regard $\U\times \R_+$ as a subposet of the product poset $\R_\times^3:=\Rop\times \R\times \R$ via the identification $([t_1,t_2],\delta) \leftrightarrow (t_1,t_2,\delta)$. Indeed,
\[([t_1,t_2],\delta)\leq ([t_1',t_2'],\delta') \ \mbox{in $\U\times \R_+$\ \ \ if and only if}\ \ \  (t_1,t_2,\delta)\leq (t_1',t_2',\delta')\ \mbox{in $\R_\times^3$}.\]
\end{remark}

\paragraph{Spatiotemporal Rips filtration of a DMS.} 
Let $\simp$ be the cateogry of abstract simplicial complexes with simplicial maps. By a \emph{(simplicial) filtration} we mean a functor from a poset to $\simp$. In order to encode multiscale topological features of \emph{DMSs} into a \emph{single} filtration, we define the \emph{spatiotemporal Rips filtration} of a DMS. Let us begin by recalling the \emph{Rips complex}:

\begin{definition}[The Rips complex]\label{def:Rips complex}
Let $(X,d_X)$ be a metric space. For each $\delta\in\R$, by $\rips(X,d_X)$ we mean the abstract simplicial complex on the set $X$ where a subset $\sigma\subset X$ belongs to $\rips(X,d_X)$ if and only if $d_X(x,x')\leq \delta$ for all $x,x'\in \sigma$. Note that if $\delta<0$, then $\rips(X,d_X)$ is empty.
\end{definition}

\begin{definition}[The Rips filtration]\label{def:Rips} Let $(X,d_X)$ be a metric space. The \emph{Rips filtration of a finite metric space $(X,d_X)$} is the functor $\ripsss(X,d_X):\R\rightarrow \simp$ described as follows: To each $\delta\in\R$, the simplicial complex $\rips(X,d_X)$ is assigned. Also, to any pair $\delta\leq \delta'$ in $\R$, the inclusion map $\rips(X,d_X)\hookrightarrow \mathcal{R}_{\delta'}(X,d_X)$ is assigned. 
\end{definition}

\begin{definition}[The spatiotemporal Rips filtration of a DMS]\label{def:spatiotemporal Rips} Given any DMS $\gammax$, the simplicial filtration $\ripss(\gamma_X):\U\times \R_+\rightarrow \simp$ defined as in Figure \ref{fig:3d filtration} is called the \emph{(spatiotemporal) Rips filtration} of $\gamma_X$.
\end{definition}

Definition \ref{def:spatiotemporal Rips} is based on a blend of ideas related to the Rips filtration \cite{cohen2007stability,dghrips,comptopo-herbert} and the interlevel set persistence/categorified Reeb graphs \cite{bendich2013homology,botnan2018algebraic,carlsson2009zigzag,de2016categorified}. The super-index ``lev" in $\ripss(\gamma_X)$ is meant to emphasize the  connection to ``inter\underline{lev}elset persistence".

\begin{figure}
\begin{center}
 \includegraphics[width=0.6\textwidth]{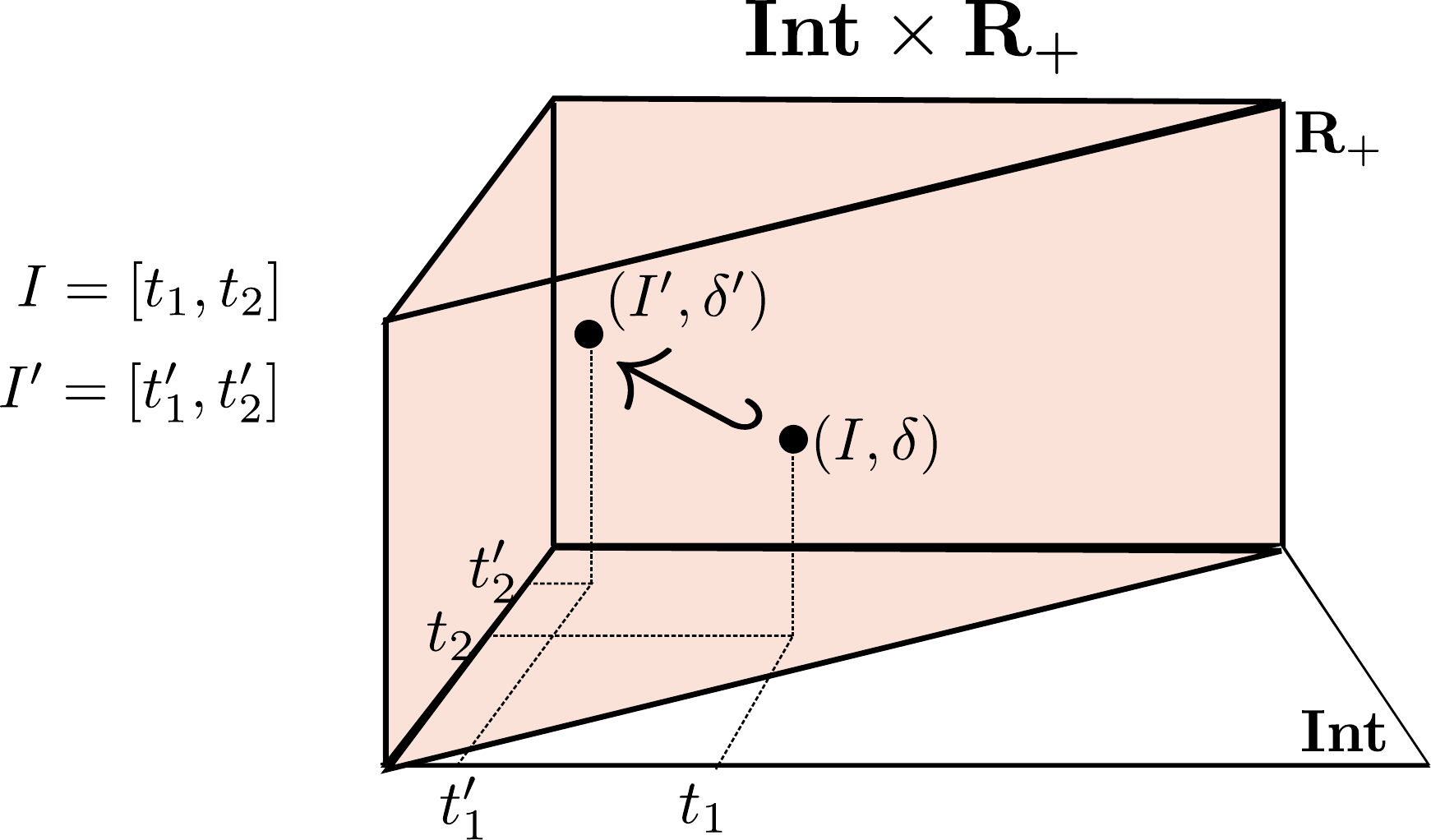}
 \caption{\label{fig:3d filtration} To each $(I,\delta)\in \U\times\R_+$, we associate the \emph{Rips complex} $\rips(X,\bigvee_{I}d_X)$ on the metric space* $(X,\bigvee_{I}d_X)$. Provided another interval $I'\in \U$ and scale $\delta'\in\R_+$ with $I\subset I'$ and $\delta\leq \delta'$, we obtain the inclusion $\rips(X,\bigvee_{I}d_X)\hookrightarrow \mathcal{R}_{\delta'}(X,\bigvee_{I'}d_X)$. This construction gives rise to a $3$-dimensional simplicial filtration $\ripss(\gamma_X)$ indexed by $\U\times \R_+$.\\ * \small{In fact, $\bigvee_{I}d_X:X\times X\rightarrow \R_+$ does not necessarily satisfy the triangle inequality. However, it does not prevent us from defining the Rips complex on the semi-metric space $(X,\bigvee_{I}d_X)$.}} 
\end{center}
\end{figure}

\begin{remark}[Comprehensiveness of Definition \ref{def:spatiotemporal Rips}]\label{rem:basic2}We remark the following:
\begin{enumerate}[label=(\roman*)]
    \item Consider the constant DMS $\gamma_X\equiv (X,d_X)$ as in Example \ref{ex:constant} \ref{item:constant1}. Then, the spatiotemporal Rips filtration of $\gamma_X$ amounts to the Rips filtration of $(X,d_X)$: for all $\bu\in \U$ and $\delta\in \R_+$, \[\ripss(\gamma_X)_{(\bu,\delta)}=\rips(X,d_X).\] 
    \item Let $\gammax$ be a DMS. For each $t\in \R$, we have the Rips filtration \newline $\ripsss(X,d_X(t)):\R_+\rightarrow \simp$ of the metric space $(X,d_X(t))$. All those filtrations are incorporated by $\ripss(\gamma_X)$ in the following sense:
    \[\ripss(\gamma_X)_{(\bt,\delta)}=\rips(X,d_X(t)), \ \ t\in \R, \ \delta\in \R_+.\]
\end{enumerate}
\end{remark}

By functoriality of the simplicial homology functor, we can define, for each $k\in \Z_+$, the persistence module $\Hrm_k\left(\ripss(\gamma_X)\right):=\U\times\R_+\rightarrow \vect$.

\paragraph{The rank invariant and the Betti-$0$ function of a DMS.}
We consider the rank invariant \cite{carlsson2009theory} of this multidimensional persistence module $\Hrm_k(\ripss(\gamma_X))$. Let \begin{equation}\label{eq:R6 subposet}
    \underline{\R^6}:=\{(t_1,t_2,\delta,t_1',t_2',\delta')\in \R^6: [t_1,t_2]\subset [t_1',t_2']\ \mbox{and}\ \delta\leq \delta'\}.
\end{equation}

\begin{definition}[The rank invariant of a DMS]\label{def:rank invariant_brief} Let $\gamma_X$ be any DMS. For each non-negative integer $k$, the $k$-th rank invariant of $\gamma_X$ is a function $\rk_k(\gamma_X):\underline{\R^6}\rightarrow \Z_+$ defined as
\[\rk_k(\gamma_X)\left(t_1,t_2,\delta,t_1',t_2',\delta'\right):=\mathrm{rank} \left(\Hrm_k\left(\rips\left(X,\bigvee_{[t_1,t_2]}d_X\right)\hookrightarrow \mathcal{R}_{\delta'}\left(X,\bigvee_{[t_1',t_2']}d_X\right)\right)\right).
\]
See Figure \ref{fig:3d filtration}.
\end{definition}

In Section \ref{sec:Crocker} we compare the rank invariant of a DMS with the \emph{CROCKER-plots} introduced in \cite{topaz}.

\begin{definition}[The Betti-0 function of a DMS]\label{def:betti0}Let $\gammax$ be a DMS. We define the \emph{Betti-0 function} $\beta_0^{\gamma_X}:\U\times\R_+\rightarrow \Z_+$ of $\gamma_X$ by sending each $(I,\delta)\in \U\times \R_+$ to the dimension of  $\Hrm_0\left(\rips\left(X,\bigvee_{I}d_X \right)\right)$.
\end{definition}

 \begin{example}\label{ex:betti-0} Consider the DMSs $\gamma_X$ and $\gamma_Y$ given as the dynamic point clouds $X_r(\cdot)$ and $Y_r(\cdot)$ in Figure \ref{fig:intro} respectively. The Betti $0$-functions of $\gamma_X$ and $\gamma_Y$ are illustrated in Figure \ref{fig:betti_zero}.
\end{example}

It is not difficult to check that if $I\subset I'$ in $\U$ and $\delta\leq\delta'$ in $\R_+$, then $\beta_0^{\gamma_X}(I,\delta)\geq \beta_0^{\gamma_X}(I',\delta')$. This monotonicity is a special feature of Betti-$0$ functions, which is not shared by \emph{other Betti-$k$ functions} for $k\geq 1$. We will exploit this monotonicity property to metrize the collection of Betti-$0$ functions and in turn to obtain a tight lower bound for $\dintl$ or $\dgh$. Also, we remark that when $\gamma_X$ is a constant DMS (Example \ref{ex:constant} \ref{item:constant1}), $\beta_0^{\gamma_X}$ is constant with respect to the first factor.

\begin{figure}
    \begin{center}
    \includegraphics[width=\textwidth]{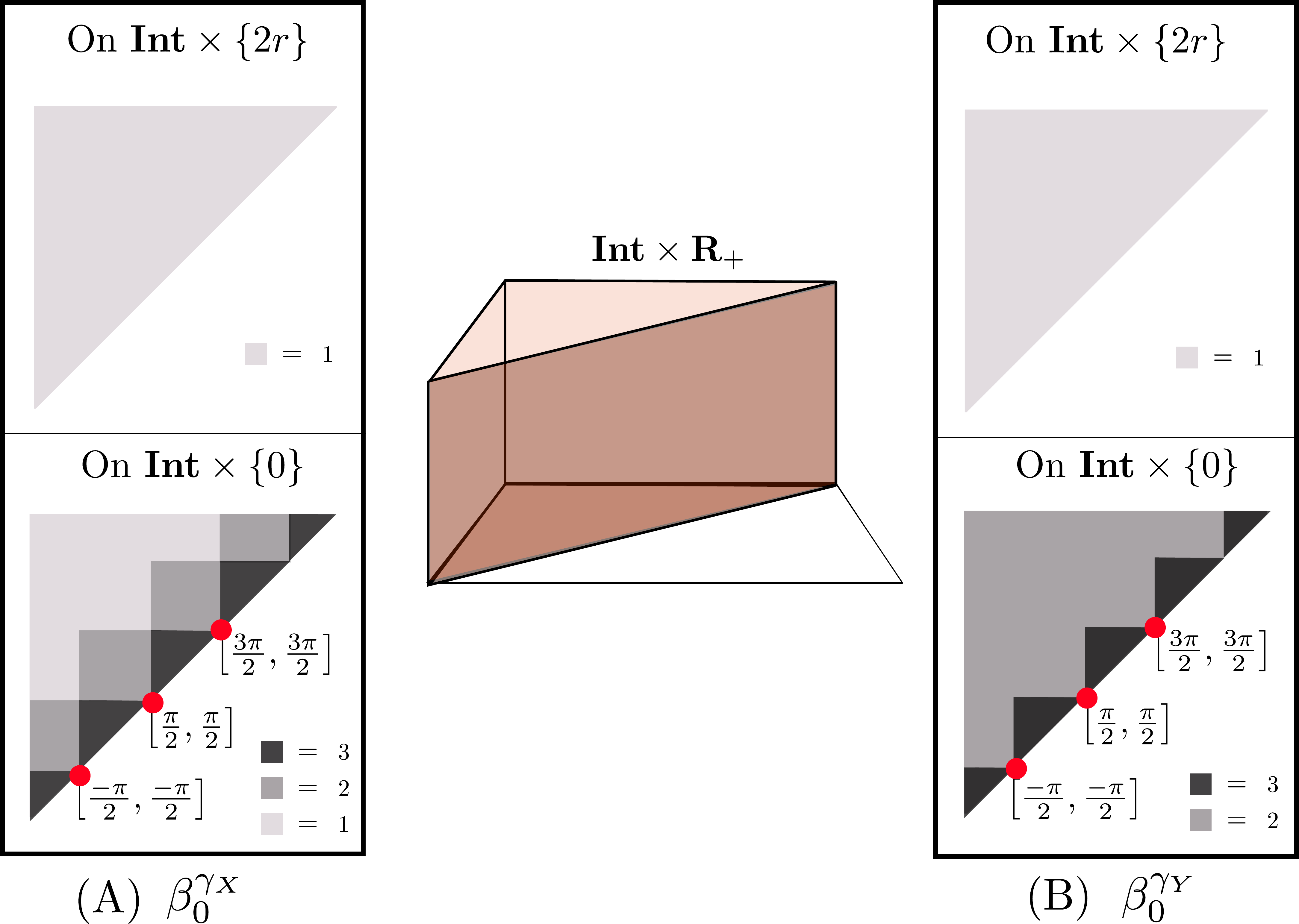}
    \caption{(The Betti-$0$ functions $\beta_0^{\gamma_X},\beta_0^{\gamma_Y}$ of the DMSs in Figure \ref{fig:intro}) 
     The middle figure represents the domain  $\U\times\R_+$ (Figure \ref{fig:3d filtration}) of $\beta_0^{\gamma_X}$ and $\beta_0^{\gamma_Y}$. (A) and (B) illustrate the value of $\beta_0^{\gamma_X}$ and $\beta_0^{\gamma_Y}$ respectively on the horizontal half-planes $\U\times \{0\}$ (bottom) and $\U\times \{2r\}$  (top).
     In particular, if $\delta\in [2r,\infty)$, $\beta_0^{\gamma_X}(I,\delta)= 1$ for all $I\in \U$. The same properties hold for $\beta_0^{\gamma_Y}$.   }
    \label{fig:betti_zero}
    \end{center}
\end{figure}

\section{Interleaving distance}\label{sec:interleaving for various categories}
In this section we review the \emph{interleaving distance} for $\R^d$-indexed functors \cite{botnan2018algebraic,CCG09,lesnick}. In particular, the interleaving distance \emph{between integer-valued functions} (Section \ref{subsec:interleaving}) will be utilized for obtaining a computationally tractable lower bound for $\dintl$.
\subsection{Interleaving distance}\label{sec:details about the interleaving}

\paragraph{Natural transformations.} We recall the notion of \emph{natural transformations} from category theory \cite{mac2013categories}: Let $\mathcal{C}$ and $\mathcal{D}$ be any categories and let $F,G:\mathcal{C}\rightarrow \mathcal{D}$ be any two functors. A natural transformation $\varphi:F\Rightarrow G$ is a collection of morphisms $\varphi_c: F_c\rightarrow G_c$ in $\mathcal{D}$  for all objects $c\in\mathcal{C}$ such that for any morphism $f:c\rightarrow c'$ in $\mathcal{C}$, the following diagram commutes:

\begin{center}
\begin{tikzcd}
	F_c \arrow{r}{F(f)} \arrow{d}{\varphi_{c}}
	&F_{c'} \arrow{d}{\varphi_{c'}}\\
	G_{c} \arrow{r}{G(f)} &G_{c'}.
\end{tikzcd}
\end{center}
Natural transformations $\varphi:F\rightarrow G$ are considered as morphisms in the category $\mathcal{D}^\mathcal{C}$ of all functors from $\mathcal{C}$ to $\mathcal{D}.$

\paragraph{The interleaving distance between $\R^d$-indexed functors.} In what follows, for any $\eps\in [0,\infty)$, we will denote the vector $\eps(1,\ldots,1)\in \R^d$ by $\vec{\eps}$. The dimension $d$ will be clearly specified in context.

\begin{definition}[$\mathbf{v}$-shift functor]\label{def:eps-shifting} Let $\mathcal{C}$ be any category.  For each $\mathbf{v}\in [0,\infty)^n$,  the $\bv$-shift functor $(-)(\bv):\C^{\R^d}\rightarrow \C^{\R^d}$ is defined as follows: 
\begin{enumerate}[label=(\roman*)]
    \item (On objects) Let $F:\R^d\rightarrow \mathcal{C}$ be any functor. Then the functor $F(\mathbf{v}):\R^d\rightarrow \mathcal{C}$ is defined as follows: For any $\ba\in \R^d$, 
	\[F(\mathbf{v})_\ba:=F_{\ba+\bv}.\]
	Also, for another $\ba'\in \R^d$ such that $\ba\leq \ba'$ we define
	\[F(\mathbf{v})(\ba\leq\ba'):= F\left(\ba+\bv\leq \ba'+\bv\right).\]
In particular, if $\mathbf{v}=\vec{\eps}\in[0,\infty)^d$, then we simply write $F(\eps)$ in lieu of $F(\vec{\eps})$.	

    \item (On morphisms) Given any natural transformation $\varphi:F\Rightarrow G$, the natural transformation $\varphi(\bv):F(\bv)\Rightarrow G(\bv)$ is defined as $\varphi(\bv)_\ba=\varphi_{\ba+\bv}:F(\bv)_\ba\rightarrow G(\bv)_\ba$ for each $\ba\in \R^d$.
\end{enumerate}
 
\end{definition}

For any $\bv\in [0,\infty)^d$, let $\varphi_F^\bv: F \Rightarrow F(\bv)$ be the natural transformation whose restriction to each $F_\ba$ is the morphism $F(\ba\leq \ba+\bv)$ in $\C$. When $\bv=\vec{\eps}$, we denote $\varphi_F^\bv$ simply by $\varphi_F^\eps$.  

\begin{definition}[$\bv$-interleaving between $\R^d$-indexed functors]\label{def:interleaving_general} Let $\mathcal{C}$ be any category. Given any two functors $F,G:\R^d\rightarrow \mathcal{C},$ we say that they are \emph{$\bv$-interleaved} if there are natural transformations $f:F\Rightarrow G(\bv)$ and $g:G\Rightarrow F(\bv)$  such that 
	\begin{enumerate}[label=(\roman*)]
		\item $g(\bv)\circ f = \varphi_F^{2\bv}$,
		\item $f(\bv)\circ g = \varphi_G^{2\bv}$.
	\end{enumerate} 

In this case, we call $(f,g)$ a \emph{$\bv$-interleaving pair}. When $\bv=\eps(1,\ldots,1)$, we simply call $(f,g)$ \emph{$\eps$-interleaving pair}. The interleaving distance between $F$ and $G$ is defined as
\begin{equation}\label{eq:interleaving0}
    \dintd^\C(F,G):=\inf\{\eps\in [0,\infty):F,G\ \mbox{are $\vec{\eps}$-interleaved}\},
\end{equation}
	where we set $\dintd^\C(F,G)=\infty$ if there is no $\eps$-interleaving pair between $F$ and $G$ for any $\eps\in [0,\infty)$. Then $\dintd^\C$ is an extended pseudo-metric for $\C$-valued  $\R^d$-indexed functors. We drop the subscript $d$ from $\dintd^{\C}$  when confusion is unlikely.
\end{definition}

\begin{remark}\label{rem:free subposet}
\begin{enumerate}[label=(\roman*)]
    \item  Let $\R^{\Diamond}$ denote the poset either of $\R$ or $\R^{\mathrm{op}}$. The interleaving distance $\dint^\C$ is also defined in the similar way for $\R^d$-indexed modules, where the poset $\R^d$ is equipped with the product partial order $\R^{\Diamond}\times\R^{\Diamond}\times\ldots\times\R^{\Diamond}.$ \label{item:free subposet1}
    \item Let $\Pb$ be any non-empty \emph{upper set} of $\R^d$: For every $p\in\Pb$, $U(p):=\{q\in \R^d: q\geq p\}$ is contained in $\Pb$. Then, we can define the interleaving distance between $\Pb$-indexed modules in the manner described by Definition \ref{def:interleaving_general}. \label{item:upper set}
    
\end{enumerate}

\paragraph{Full interleaving.} By $\sets$, we mean the category of sets with set maps as morphisms. Also, by $\vect$, we mean the category of vector spaces over a fixed field $\F$, with linear maps as morphisms. 

Let $\C$ be either $\Sets$ or $\vect$.  Given any $F,G:\R^d\rightarrow \C$, suppose that $(f,g)$ is an $\eps$-interleaving pair between $F$ and $G$. For \emph{each} $\ba\in \R^d$, if $f_\ba:F_\ba\rightarrow G_{\ba+\vec{\eps}}$ and $g_\ba:G_\ba\rightarrow F_{\ba+\vec{\eps}}$ are surjective, then we call $(f,g)$ a \emph{surjective} $\eps$-interleaving pair. If there exists a surjective $\eps$-interleaving between $F$ and $G$, we say that $F$ and $G$ are \emph{fully} $\eps$-interleaved. We define
\[\bdintd^\C(F,G):=\inf\left\{\eps\in [0,\infty):\ \mbox{$F,G$ are \emph{fully} $\vec{\eps}$-interleaved} \right\}.\]

We drop the subscript $d$ from $\bdintd^\C$  when confusion is unlikely. By definition, for any $F,G:\R^d\rightarrow \C$, it is clear that $\dintd^\C(F,G)\leq \bdintd^\C(F,G)$.  Also, it is not difficult to check that $\bdintd^\C$ is an extended pseudometric on $\ob(\C^{\R^d})$. 

By utilizing the full interleaving distance $\bdint^\C$, we obtain a lower bound for $\dintl$ as well as a new lower bound for the Gromov-Hausdorff distance (Theorem \ref{thm:betti-0 stability}, Remark \ref{thm:tighter} and Theorem \ref{thm:better}). 
    
\end{remark}
\subsection{Interleaving distance between integer-valued functions}\label{subsec:interleaving}

In this section we consider the interleaving distance between monotonic integer-valued functions by regarding them as functors.

\paragraph{Poset-valued maps.} Let $\Pb$ and $\Q$ be any two posets. Suppose that $f:\Pb\rightarrow \Q$ is any (monotonically) \emph{increasing} map, i.e. for any $p\leq q$  in $\Pb$, $f(p)\leq f(q)$. Then, by regarding $\Pb,\Q$ as categories, $f$  can be regarded as a \emph{functor}. On the other hand, suppose that $g:\Pb\rightarrow \Q$ is any (monotonically) \emph{decreasing} map, i.e. for any $p\leq q$ in $\Pb$, $f(p)\geq f(q)$. Then, $g:\Pb\rightarrow \Qop$ can also be called a functor.

\paragraph{The interleaving distance between integer-valued functions.} Let $d$ be a positive integer. Let $\R^d$ be the poset, where $\ba=(a_1,\ldots,a_d)\leq\bb=(b_1,\ldots,b_d)$ in $\R^d$ if and only if $a_i\leq b_i$ for each $i=1,\ldots,d$. For any $\eps>0$, let $\vec{\eps}:=\eps(1,\ldots,1)\in \R^d$. Consider any non-increasing integer-valued function $F:\R^d\rightarrow \Z_+$. Note that $F$ can be regarded as a \emph{functor} from the poset cateogory $\R^d$ to the other poset category $\Zop$. Since $\Zop$ is a \emph{thin} category, given another functor $G:\R^d\rightarrow \Zop$, the interleaving distance (Definition \ref{def:interleaving_general}) between $F$ and $G$ can be written as
\[\dintd^{\Zop}(F,G)=\inf\{\eps\in[0,\infty):\ \forall \ba\in \R^d, F_{\ba}\geq G_{\ba+\vec{\eps}},\ \mbox{and}\ G_{\ba}\geq F_{\ba+\vec{\eps}} \}.\]
The computational complexity for $\dintd^{\Zop}$ is provided in Theorem \ref{thm:cost2}. We will use $\dintd$, or even more simply $\dint$ in place of $\dintd^{\Zop}$  when confusion is unlikely.

\begin{remark} The metric $\dint$ is closely related to the erosion distance \cite{patel2018generalized}. See Remark \ref{rem:erosion}.
\end{remark}

\section{Stability theorems for persistent homology features of DMSs}\label{sec:overview}

In this section we establish the main results of this paper: namely, stability of the rank invariant and Betti-$0$ function of DMSs (Section \ref{subsec:stability}). We interpret these stability theorems as a generalization of the standard stability results for (static) metric spaces (Section \ref{sec:standard}).

\subsection{Stability theorems}\label{subsec:stability}

Recall the spatiotemporal Rips filtration $\U\times \R_+\rightarrow \simp$ of a DMS (Definition \ref{def:spatiotemporal Rips}). The poset $\U\times \R_+$ can be regarded as an upper set of $\R_\times^3$ (Remarks \ref{rem:subposet} and \ref{rem:free subposet} \ref{item:upper set}) and thus we can utilize $\dint^{\mathbf{Vec}}$ for comparing $(\U\times \R_+)$-indexed persistence modules.

\begin{theorem}[Stability of spatiotemporal persistence modules induced by DMSs]\label{thm:main2}
Let $\gammax$ and $\gammay$ be any two DMSs. 
Then for any $k\in \Z+$,
	\begin{equation}\label{eq:main2}
	    \dint^\mathbf{Vec}\left(\Hrm_{k}(\ripss(\gamma_X)),\Hrm_{k}(\ripss(\gamma_Y))\right)  \leq 2 \cdot \dintl(\gamma_X,\gamma_Y).
	\end{equation}
In particular, when $k=0$, the $\dint^\vect$ in the LHS of the above inequality can be promoted to the full interleaving $\bdint^\vect$.
\end{theorem}

We remark that the promotion of $\dint^\vect$ to $\bdint^\vect$ for $k=0$ is crucial for proving Theorem \ref{thm:betti-0 stability} below. See Section \ref{sec:details about PH} for the proof of Theorem \ref{thm:main2}. This stability implies that $\dint^{\vect}$ between 3-dimensional persistence modules serves as a lower bound for $\dintl$.  Since computing $\dint^\vect$ between 3-dimensional persistence modules is prohibitive \cite{botnan2018algebraic}, we make use of the \emph{rank invariants/Betti-$0$ functions} of DMSs (Definitions \ref{def:rank invariant_brief} and \ref{def:betti0}) and the interleaving distance $\dint$ between integer-valued functions (Section \ref{subsec:interleaving}) to obtain a lower bound for $\dintl$ as below.

\paragraph{Adapted rank invariant of a DMS.} The set $\underline{\R^6}$ in (\ref{eq:R6 subposet}) is \emph{not} an upper set (Remark \ref{rem:free subposet} \ref{item:upper set}) of the poset 
\begin{equation}\label{eq:R6 poset}
    \R_\times^6 :=\R\times\Rop\times\Rop\times\Rop\times \R\times\R,
\end{equation} into which $(\U\times\R_+)^\mathrm{op}\times (\U\times\R_+)$ can be embedded.  In order to ensure that we are in a position to utilize the metric $\dint$ for comparing rank invariants of DMSs, we extend the domain of the rank invariant of a DMS to the poset $\R_\times^6$. Given any $(v_1,v_2,v_3)\in\R^3$, we write $(v_1,v_2,v_3)\in \U\times \R_+$ if $v_1\leq v_2$ and $v_3\in \R_+$. 

Any element $\ba=(a_1,a_2,a_3,a_4,a_5,a_6)\in \R^6$, is  called  \emph{admissible}, if $\ba$ is obtained by concatenating a comparable pair of elements in $\U\times \R_+$, i.e.
both $(a_1,a_2,a_3)$ and $(a_4,a_5,a_6)$ belong to $\U\times\R_+$ and $(a_1,a_2,a_3)\leq (a_4,a_5,a_6)$ in  $\U\times\R_+$. Otherwise, $\ba$ is called \emph{non-admissible}. In particular, $\ba$ is called \emph{trivially} non-admissible, if there is no admissible $\bb\in \R^6$ such that $\bb<\ba$ in the poset $\R_\times^6$.

\begin{definition}[Adapted rank invariant of a DMS]\label{def:the rank invariant of a DMS}Let $\gammax$ be any DMS and let $k\in \Z_+$. We define the map $\rk_k(\gamma_X):\R^6\rightarrow \Z_+\cup\{\infty\}$, called the \emph{$k$-th rank invariant of $\gamma_X$}, as follows: For $\ba=(a_1,\ldots,a_6)\in \R^6$,

\[\rk_k(\gamma_X)(\ba):=\begin{cases} \mathrm{rank}\left( \Hrm_k\left( \rips(\bigvee_{I}d_X)\hookrightarrow \mathcal{R}_{\delta'}(\bigvee_{I'}d_X)\right) \right),&\mbox{$\ba$ is admissible,}\\ \infty,&\mbox{$\ba$ is trivially non-admissible,}\\
0,&\mbox{otherwise.}
\end{cases}\]
where $I=[a_1,a_2]$, $I'=[a_4,a_5]$, $\delta=a_3$ , and $\delta'=a_6$.
\end{definition}

Note that when $\ba\in\R^6$ is a concatenation of a \emph{repeated} pair $([t_0,t_0],\delta_0),([t_0,t_0],\delta_0)\in \U\times\R_+$, i.e. $\ba=(t_0,t_0,\delta_0,t_0,t_0,\delta_0)$, then \[\rk_0(\gamma_X)(\ba)=\dim\left(\Hrm_0\left(\mathcal{R}_{\delta_0}(X,d_X(t_0))\right)\right)=\beta_0^{\gamma_X}(t_0,t_0,\delta_0)\ \ \mbox{(Definition \ref{def:betti0}).}\] 

We can regard $\rk_k(\gamma_X)$ as a \emph{functor} $\R^6_\times \rightarrow  (\Z_+\cup\{\infty\})^{\mathrm{op}}$:

\begin{proposition}\label{prop:decreasing property}Let $\gamma_X$ be any DMS. For any $\ba,\bb \in \R_\times ^6$ with $\ba\leq \bb$, 
\[\rk_k(\gamma_X)(\ba)\geq \rk_k(\gamma_X)(\bb)\ \ \mbox{in $\Z_+\cup\{\infty\}$.}\]
\end{proposition}

See Section \ref{proof:decreasing property} for the proof. 
By virtue of Proposition \ref{prop:decreasing property},  $\dint$ can serve as a metric on the collection of all (adapted) rank invariants of DMSs. 

By combining Theorem \ref{thm:main2} with standard stability results for the rank invariant (Theorem \ref{prop:stability of the rank invariant}) we arrive at:

\begin{theorem}[Stability of the rank invariant of DMSs]\label{thm:rank k stability} 
Let $\gammax$ and $\gammay$ be any two DMSs. For any $k\in \Z_+$, 
\begin{equation}\label{eq:rank k stability}
    \dint\left(\rk_k(\gamma_X),\rk_k(\gamma_Y)  \right)\leq  2 \cdot \dintl(\gamma_X,\gamma_Y).
\end{equation}
\end{theorem}

\paragraph{Improvement for $k=0$.} By restricting ourselves to  clustering information (i.e.  $0$-th homology) of DMSs, we obtain a stronger lower bound for the metric $\dintl$. 
Namely, by regarding the Betti-$0$ function of a DMS (Definition \ref{def:betti0}) as a functor $\U\times\R_+\rightarrow \Z_+^{\mathrm{op}}$, we can compare any two Betti-0 functions of DMSs via the interleaving distance $\dint$ and we have:

\begin{theorem}[Stability of the Betti-0 function]\label{thm:betti-0 stability}
Let $\gamma_X$ and $\gamma_Y$ be any two DMSs. Then, 
\begin{equation}\label{eq:betti-0 stability}
	\dint\left(\beta_0^{\gamma_X},\beta_0^{\gamma_Y}\right)\leq 2\cdot \dintl(\gamma_X,\gamma_Y).
\end{equation}
\end{theorem}

We prove Theorem \ref{thm:betti-0 stability} in Section \ref{sec:details about SC}. Also, we remark that the LHSs of inequalities in (\ref{eq:rank k stability}) and (\ref{eq:betti-0 stability}) are computable in poly-time (Theorem \ref{thm:cost2}) using the well-known \emph{binary search} algorithm.

\begin{remark}[Sensitivity of the LHS in (\ref{eq:betti-0 stability})]\label{rem:sensitivity} Consider the DMSs $\gamma_X$ and $\gamma_Y$ given as in Example \ref{ex:betti-0}.  The value $\dint\left(\beta_0^{\gamma_X},\beta_0^{\gamma_
Y}\right)$ is at least $r$, as we will see below. This in turn implies that the metric $\dint$ is sensitive enough to discriminate (the Betti-$0$ functions of) $\gamma_X$ and $\gamma_Y$.
\end{remark}

\begin{proof}[Details about Remark \ref{rem:sensitivity}]\label{proof:sensitivity} 
Observe that
\[\bigvee_{\left[\frac{\pi}{2},\frac{3\pi}{2}\right]}d_X (x_i,x_j)=\begin{cases}2,& i=1,j=3 \\ 0,&\mbox{otherwise,} 
\end{cases}\hspace{15mm}\bigvee_{\left[\frac{\pi}{2},\frac{3\pi}{2}\right]}d_Y (y_i,y_j)=\begin{cases}1,&i=1,j=2\\ 2,& i=1,j=3 \\ 0,&\mbox{otherwise.} 
\end{cases}\]
Hence, the geometric realization of Rips complexes $\mathcal{R}_0(X,\bigvee_{\left[\frac{\pi}{2},\frac{3\pi}{2}\right]}d_X)$ and \newline $\mathcal{R}_0(Y,\bigvee_{\left[\frac{\pi}{2},\frac{3\pi}{2}\right]}d_Y)$ are illustrated in Figure \ref{fig:Rips complexes}.
\begin{figure}
    \centering
    \includegraphics[width=0.8\textwidth]{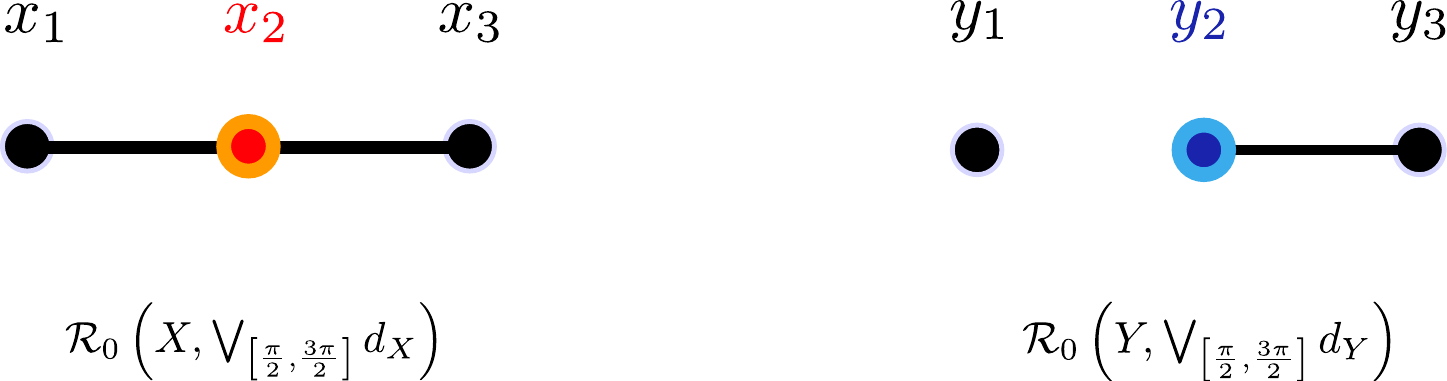}
    \caption{The geometric realization of $\mathcal{R}_0(X,\bigvee_{\left[\frac{\pi}{2},\frac{3\pi}{2}\right]}d_X)$ and \newline $\mathcal{R}_0(Y,\bigvee_{\left[\frac{\pi}{2},\frac{3\pi}{2}\right]}d_Y)$ for the DMSs $\gamma_X$ and $\gamma_Y$ in Example \ref{ex:betti-0}.}
    \label{fig:Rips complexes}
\end{figure}
By counting the number of connected components of these complexes, we have $\beta_0^{\gamma_X}\left(\left[\frac{\pi}{2},\frac{3\pi}{2}\right],0 \right)=1$ and  $\beta_0^{\gamma_Y}\left(\left[\frac{\pi}{2},\frac{3\pi}{2}\right],0 \right)=2$. Also, it is not difficult to check that for any $\eps\in [0,r)$, $\mathcal{R}_\eps(Y,\bigvee_{\left[\frac{\pi}{2}-\eps,\frac{3\pi}{2}+\eps\right]}d_Y)=\mathcal{R}_0(Y,\bigvee_{\left[\frac{\pi}{2},\frac{3\pi}{2}\right]}d_Y)$, so that \[\beta_0^{\gamma_X}\left(\left[\frac{\pi}{2},\frac{3\pi}{2}\right],0 \right)=1<2=\beta_0^{\gamma_Y}\left(\left[\frac{\pi}{2}-\eps,\frac{3\pi}{2}+\eps\right],\eps \right).\] 
By the definition of $\dint$, this inequality implies that $\dint\left(\beta_0^{\gamma_X},\beta_0^{\gamma_Y}\right)$ is at least $r$.

Next, we show that $\dint\left(\beta_0^{\gamma_X},\beta_0^{\gamma_Y}\right)\leq 2r$. For any $\eps\in[2r,\infty)$ and  any $I\in \U$, \[\beta_0^{\gamma_X}\left(I,\eps \right)=\beta_0^{\gamma_Y}\left(I,\eps \right)=1,\] 
which is illustrated in Figure \ref{fig:betti_zero}. Therefore, for any $([t_1,t_2],\delta)\in \U\times \R_+$, 
\[\beta_0^{\gamma_X}\left([t_1,t_2],\delta\right)\geq \beta_0^{\gamma_Y}\left([t_1-2r,t_2+2r],\delta+2r\right)=1,\]
\[\beta_0^{\gamma_Y}\left([t_1,t_2],\delta\right)\geq \beta_0^{\gamma_X}\left([t_1-2r,t_2+2r],\delta+2r\right)=1.\]
Therefore, we have $\dint\left(\beta_0^{\gamma_X},\beta_0^{\gamma_Y}\right)\leq 2r$.
\end{proof}

 In order to obtain a lower bound for $\dintl$ between two DMSs, computing the distance between the Betti-$0$ functions of the DMSs (the LHS of the inequaliy in  (\ref{eq:betti-0 stability})) is better than computing the distance between their $0$-th rank invariants (the LHS of the inequaliy in (\ref{eq:rank k stability})) : 

\begin{proposition}\label{prop:better-dynamic} For any two DMSs $\gammax$ and $\gammay$,
\begin{equation}\label{eq:inequality}
    \dintsix\left(\rk_0(\gamma_X),\rk_0(\gamma_Y)  \right)\leq  \dintthree(\beta_0^{\gamma_X},\beta_0^{\gamma_Y}).
\end{equation}
\end{proposition}

Proposition \ref{prop:better-dynamic} is a corollary of Proposition \ref{prop:dim vs rk}. The proof relies on the fact that all inner morphisms of the persistence modules $\Hrm_0\left(\ripss(\gamma_X) \right)$ and $\Hrm_0\left(\ripss(\gamma_Y) \right)$ are \emph{surjective}. 
 In Example \ref{ex:demonstration2}, we consider a concrete example of the bound provided in Proposition \ref{prop:better-dynamic}.

\subsection{Relationship with standard stability theorems}\label{sec:standard} 

The main goal of this section is to explain, when restricting ourselves to the class of constant DMSs (Example \ref{ex:constant} \ref{item:constant1}), how Theorems \ref{thm:main2}, \ref{thm:rank k stability} and \ref{thm:betti-0 stability} boil down to the well-known stability theorems for (static) metric spaces. Along the way, we also identify a new lower bound for the Gromov-Hausdorff distance, which is tighter than the bottleneck distance between the $0$-th persistence diagrams of Rips filtrations (Remark \ref{thm:tighter} and Theorem \ref{thm:better}).

 For $k\in \Z_+$, by post-composing the simplicial homology functor $\Hrm_k:\simp\rightarrow \vect$ (with coefficients in the field $\F$) to the Rips filtration $\ripsss(X,d_X)$ of a metric space $(X,d_X)$, we obtain the persistence module \[\Hrm_k\circ \ripsss(X,d_X):\R\rightarrow \vect.\]
 
 Let $\dgm_k(\ripsss(X,d_X))$ be the $k$-th persistence diagram of the Rips filtration  \newline $\ripsss(X,d_X)$.  Also, let $\bott$ be the bottleneck distance (Definition \ref{def:bottleneck}). Recall that $\dintl$ coincides with $\dgh$ on the class of constant DMSs (Remark \ref{rem:gromov generalization}).

\begin{remark}\label{rem:reduction to standard}Consider any two constant DMSs $\gamma_X\equiv(X,d_X)$ and $\gamma_Y\equiv(Y,d_Y)$. Then, for any $k\in \Z_+$, inequality (\ref{eq:main2}) 
reduces to:
\begin{equation}\label{eq:main2 reduction}
  \dint^{\mathbf{Vec}}(\Hrm_k\circ \ripsss(X,d_X),\Hrm_k\circ \ripsss(Y,d_Y))\leq 2\ \dgh((X,d_X),(Y,d_Y)),  
\end{equation}
or equivalently to
\begin{equation}\label{eq:bottleneck stability vs GH}
    \bott\left(\dgm_k\left(\ripsss(X,d_X)\right), \dgm_k\left(\ripsss(Y,d_Y)\right)\right)\leq 2\cdot\dgh\left((X,d_X),(Y,d_Y)\right),
\end{equation}
which are known in \cite{dghrips,chazal2014persistence}. In other words, the LHS and the RHS of inequality (\ref{eq:main2}) are respectively identical to the LHS and the RHS of inequalities (\ref{eq:main2 reduction}) or (\ref{eq:bottleneck stability vs GH}). 
\end{remark}

We define the \emph{rank invariant} of a finite metric space as follows:

\begin{definition}[The rank invariant of a metric space]\label{def:the rank invariant of a metric space}Let $(X,d_X)$ be any finite metric space and let $k\in \Z_+$. We define the map $\rk_k(X,d_X):\R^2\rightarrow \Z_+\cup\{\infty\}$, called the \emph{$k$-th rank invariant of $(X,d_X)$}, as follows: For $\ba=(\delta,
\delta')\in \R^2$,
\[\rk_k(X,d_X)(\ba)=\begin{cases} \mathrm{rank}\left(\Hrm_k\left(\rips(X,d_X)\hookrightarrow \mathcal{R}_{\delta'}(X,d_X) \right) \right),&\delta\leq\delta', \\ \infty,&\mbox{otherwise.}
\end{cases}\]
(cf. Definition \ref{def:the rank invariant of a DMS})
\end{definition}
In Definition \ref{def:the rank invariant of a metric space}, note that we can regard $\rk_k(X,d_X)$ as a functor \newline $\Rop\times \R\rightarrow  (\Z_+\cup\{\infty\})^{\mathrm{op}}$. Therefore, we can compare the rank invariants of any two finite metric metric spaces via the interleaving distance $\dint$. 

\begin{remark}\label{prop:rank k vs GH}Consider any two constant DMSs $\gamma_X\equiv(X,d_X)$ and $\gamma_Y\equiv(Y,d_Y)$. Then, for any $k\in \Z_+$, inequality (\ref{eq:rank k stability}) 
reduces to:
\begin{equation}\label{eq:rank k vs GH}
  \dint(\rk_k(X,d_X),\rk_k(Y,d_Y))\leq 2\ \dgh((X,d_X),(Y,d_Y)).  
\end{equation}
\end{remark}

\begin{remark}\label{rem:comparison} We also remark that the LHS of (\ref{eq:bottleneck stability vs GH}) is greater than equal to that of (\ref{eq:rank k vs GH}) by Corollary \ref{cor:rank vs barcode}: \begin{align*}
    \dint(\rk_k(X,d_X),\rk_k(Y,d_Y))&\leq \bott\left(\dgm_k\left(\ripsss(X,d_X)\right), \dgm_k\left(\ripsss(Y,d_Y)\right)\right)\\&\leq 2\cdot \dgh((X,d_X),(Y,d_Y)).
\end{align*}.
\end{remark}

\begin{definition}[The Betti-0 function of a finite metric space]\label{def:betti-0}
Let $(X,d_X)$ be any finite metric space. We define the \emph{Betti-$0$ function} $\beta_0^{(X,d_X)}:\R_+\rightarrow \Z_+$ of $(X,d_X)$ by sending each $\delta\in \R_+$ to the dimension of $\Hrm_0(\rips(X,d_X))$ (cf. Definition \ref{def:betti0}).
\end{definition}

Since $\beta_0^{(X,d_X)}$ is non-increasing function and $\R_+$ is an upper set of $\R$, we can compare any two Betti-$0$ functions via $\dint$.

\begin{remark}[Stability of the Betti-$0$ function]\label{thm:tighter}
Consider any two constant DMSs $\gamma_X\equiv(X,d_X)$ and $\gamma_Y\equiv(Y,d_Y)$. Then,  the inequality in (\ref{eq:betti-0 stability})  
reduces to: 
\begin{equation}\label{eq:betti0 vs GH}
    \dint\left(\beta_0^{(X,d_X)},\beta_0^{(Y,d_Y)} \right)\leq 2\  \dgh\left((X,d_X),(X,d_Y)\right).
\end{equation}
\end{remark}

In particular,  as a lower bound for $2\cdot\dgh$, the LHS of inequality (\ref{eq:betti0 vs GH}) is always as effective as the LHS of inequality (\ref{eq:bottleneck stability vs GH}) for $k=0$: 
 
 \begin{theorem}\label{thm:better}For any finite metric spaces $(X,d_X)$ and $(Y,d_Y)$,
    \[\bott\left(\dgm_0\left(\ripsss(X,d_X)\right), \dgm_0\left(\ripsss(Y,d_Y)\right)\right)\leq \dint\left(\beta_0^{(X,d_X)},\beta_0^{(Y,d_Y)} \right).\]
\end{theorem}

The proof is provided in Section \ref{sec:new lower bound}. Example \ref{ex:demonstration} below illustrates Theorem \ref{thm:better}.

\begin{example}\label{ex:demonstration} Let $X=\{x_1,x_2\}$. For any $\eps\in [0,\infty)$,  we define the two metrics $d_X$ and $d_X^\eps$ on $X$ as \[d_X(x_1,x_2)=1,\ \ \mbox{and}\  \ d_X^\eps(x_1,x_2)=1+\eps.\] 
By definition of $\dgh$ (Definitions \ref{def:the GH}) and $\dint$(Section \ref{subsec:interleaving}), one can check the following:
\begin{enumerate}[label=(\roman*)]
	\item $2\  \dgh\left((X,d_X),(X,d_X^\eps)\right)=\eps$.\label{item:the GH}
	\item $\beta_0^{(X,d_X)}(\delta)=\begin{cases}2,& \delta<[0,1)\\1,& \delta\in [1,+\infty)\end{cases}$ and $\beta_0^{(X,d_X^\eps)}(\delta)=\begin{cases}2,& \delta<[0,1+\eps)\\1,& \delta\in [1+\eps,+\infty)\end{cases}.$ Also, \[\dint\left(\beta_0^{(X,d_X)},\beta_0^{(Y,d_Y)} \right)=\eps.\] \label{item:betti-0}
\end{enumerate}

\begin{enumerate}[resume,label=(\roman*)]
	\item $\dgm_0\left(\ripsss(X,d_X)\right)=\{(0,+\infty),(0,1)\}$, and $\dgm_0\left(\ripsss(X,d_X^\eps)\right)=\{(0,+\infty),(0,1+\eps)\}.$ Also, \[\bott\left(\dgm_0\left(\ripsss(X,d_X)\right), \dgm_0\left(\ripsss(X,d_X^\eps)\right)\right)=\min \left(\eps,\frac{1+\eps}{2}\right).\]\label{item:0-bottleneck}
	\item For $k\geq 1$, both $\dgm_k\left(\ripsss(X,d_X)\right)$ and  $\dgm_k\left(\ripsss(Y,d_Y)\right)$ are the empty set, and thus 
	\[\bott\left(\dgm_k\left(\ripsss(X,d_X)\right), \dgm_k\left(\ripsss(X,d_X^\eps)\right)\right)=0.\]\label{item:k-bottleneck}
\end{enumerate}
Items \ref{item:0-bottleneck} and \ref{item:k-bottleneck} indicate that the best lower bound for $2\  \dgh\left((X,d_X),(X,d_X^\eps)\right)$ obtained by invoking inequality (\ref{eq:bottleneck stability vs GH}) is $\min \left(\eps,\frac{1+\eps}{2}\right)$. On the other hand, from items \ref{item:the GH} and \ref{item:betti-0}, we have \[\eps=2\  \dgh\left((X,d_X),(X,d_X^\eps)\right)=\dint\left(\beta_0^{(X,d_X)},\beta_0^{(Y,d_Y)} \right),\] which is, when $\eps>1$, strictly larger than $\min \left(\eps,\frac{1+\eps}{2}\right)$. This example demonstrates inequality (\ref{eq:betti0 vs GH}) is a complement to the bottleneck stablility of Rips filtration, inequality (\ref{eq:bottleneck stability vs GH}). Also, items \ref{item:the GH} and \ref{item:betti-0} show the tightness of inequality (\ref{eq:betti0 vs GH}). 
\end{example}

\begin{example}\label{ex:demonstration2}
Define two DMSs $\gamma_X$ and $\gamma_X'$ to be the \emph{constant} DMSs which are, for every time $t\in\R$, isometric respectively to the metric spaces $(X,d_X)$ and $(X,d_X^\eps)$ in Example \ref{ex:demonstration}. Then, invoking Remarks \ref{prop:rank k vs GH} and \ref{thm:tighter}, one can compute:  \[\dintsix\left(\rk_0(\gamma_X),\rk_0(\gamma_X') \right)=\dinttwo\left(\rk_0(X,d_X),\rk_0(X,d_X^\eps)\right)=\min\left(\frac{1+\eps}{2},\eps\right),\]
\[ \dintthree\left(\beta_0^{\gamma_X},\beta_0^{\gamma_X'}\right)=\dintone\left(\beta_0^{(X,d_X)},\beta_0^{(X,d_X^\eps)}\right)= \eps.
\]
See below for computational details. When $\eps>1$, this example demonstrates that the RHS of inequality (\ref{eq:inequality}) can be strictly larger.
\end{example}

\begin{proof}[Details about Example \ref{ex:demonstration2}]\label{proof:demonstration2} One can compute $\rk_0(X,d_X),\rk_0(X,d_X^\eps):\R^2\rightarrow (\Z_+\cup\{\infty\})^{\mathrm{op}}$ (Definition \ref{def:the rank invariant of a metric space}) as illustrated in Figure \ref{fig:ranks}.

\begin{figure}
    \centering
    \includegraphics[width=0.7\textwidth]{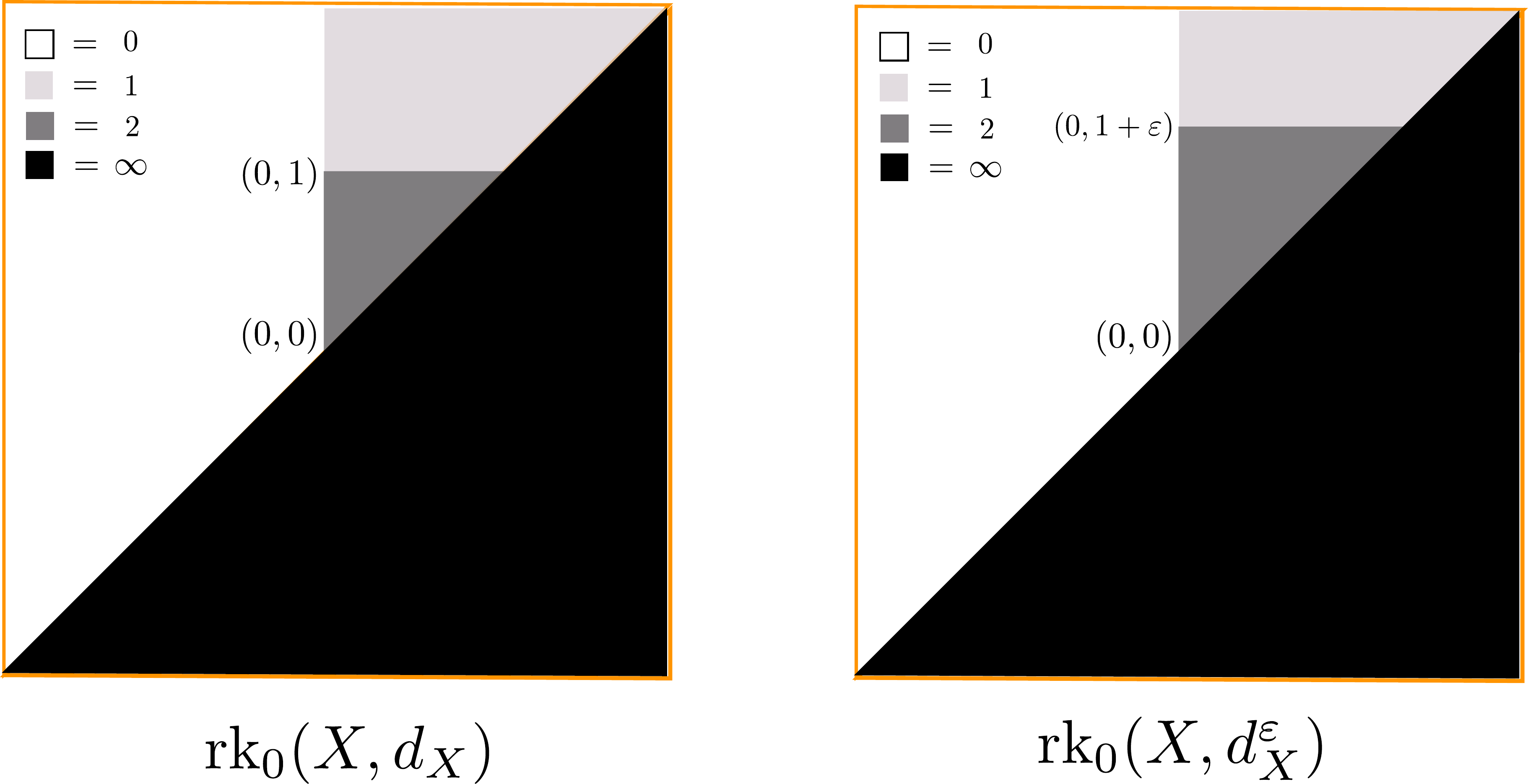}
    \caption{The $0$-th rank invariants of $(X,d_X)$ and $(X,d_X^\eps)$ in Example \ref{ex:demonstration}.}
    \label{fig:ranks}
\end{figure}
From this plot, one can check that
\[\dinttwo\left(\rk_0(X,d_X),\rk_0(X,d_X^\eps)\right)=\begin{cases}\eps,&\eps\leq [0,1]\\ \frac{1+\eps}{2},& \eps\in(1,\infty), \end{cases}\] which amounts to \[\dinttwo\left(\rk_0(X,d_X),\rk_0(X,d_X^\eps)\right)=\min\left(\frac{1+\eps}{2},\eps\right).\]
We already computed $\beta_0^{(X,d_X)}$ and $\beta_0^{(X,d_X^\eps)}$ in Example \ref{ex:demonstration}. 
Observe that the value   
\[\min\left\{\alpha\in[0,\infty):\forall \delta \in [0,\infty),\ \beta_0^{(X,d_X)}(\delta+\alpha)\leq \beta_0^{(X,d_X^{\eps})}(\delta),\ \ \beta_0^{(X,d_X^\eps)}(\delta+\alpha)\leq \beta_0^{(X,d_X^{\eps})}(\delta)\right\}\]
is equal to $\eps$. This implies that $\dintone\left(\beta_0^{(X,d_X)},\beta_0^{(X,d_X^\eps)}\right)= \eps.$
\end{proof}

\section{Computing the interleaving distance between integer-valued functions}\label{sec:algorithm}
In this section we propose an algorithm for computing the interleaving distance between integer-valued functors based on ordinary binary search.

For $n\in \N$, let $[n]:=\{1,\ldots,n\}$. Also, for each $d\in \N$, let $[n]^d\subset \Z^d$ be the subposet of $\Z^d$. Assume that $\ba=(a_1,\ldots,a_d)\in [n]^d$. If there exists $i\in \{1,\ldots,d\}$ such that $a_i=n$, we refer to $\ba$ as an \emph{upper boundary point of $[n]^d$}.  

Let $F:[n]^d\rightarrow \Z_+$ be any function. Then, $F$ can be regarded as a array of non-negative integers.  For each $k\in \{0,\ldots,n-1\}$, the restriction $F|_{[n-k]^d}$ of $F$ is the \emph{lower-left block} of $F$. Symmetrically, we define the \emph{upper-right block} $F|^{[n-k]^d}:[n-k]^d\rightarrow \Z_+$ of $F$ as follows: \[\left(F|^{[n-k]^d}\right)_\ba=F_{\ba+k(1,\ldots,1)}\ \mbox{for}\ \ba\in [n-k]^d.\] In words,  $F|^{[n-k]^d}$ is the restriction of the array $F$ to its upper-right corner of size $(n-k)^d$ with a re-indexing (in the obvious way).

Given $F,G:[n]^d\rightarrow \Z_+$, we write $F\geq G$ if $F_\ba\geq G_\ba$ for all $\ba\in [n]^d$. Let $F,G:[n]^d \rightarrow \Z_+$ be any two order-reversing functions with $0=F_{\ba}=G_{\ba}$ for each upper boundary point $\ba\in[n]^d$. For each $k\in\{0,\ldots,n-1\}$, we define the \textbf{$k$-test} for the pair $(F,G)$:\vspace{4mm}
\begin{algorithm}
\caption{$k$-test for $F,G:[n]^d\rightarrow \Z_+$.}\label{alg:delta test}
\begin{algorithmic}[H]
\If  {$F|_{[n-k]^d}\geq G|^{[n-k]^d}$ \textbf{and} $G|_{[n-k]^d}\geq F|^{[n-k]^d}$} \Return {Yes.} 
\Else \  \Return{No.}
\EndIf
\end{algorithmic}\label{alg:k-test}
\vspace{4mm}
\end{algorithm}

\begin{remark}\label{rem:k-test}Let $F,G:[n]^d \rightarrow \Z_+$ be any two order-reversing functions with $0=F_{\ba}=G_{\ba}$ for each upper boundary point $\ba\in[n]^d$. Fix $k\in\{0,\ldots,n-1\}$. Then,
\begin{enumerate}[label=(\roman*)]
    \item suppose that the $k$-test for $(F,G)$ returns "Yes". Then, for any $k'\in\{k,\ldots,n-1\}$ the $k'$-test for $(F,G)$ returns also "Yes",\label{item:k-test1}
    \item the $(n-1)$-test for $(F,G)$ always returns "Yes". \label{item:k-test2}
\end{enumerate}
\end{remark}

\begin{example} We consider two examples.
\begin{enumerate}[label=(\Alph*)]
    \item $(d=1)$ Consider $F,G:[4]\rightarrow \Z_+$ defined as follows:
\[F:=(F_1,F_2,F_3,F_4)=(5,3,1,0),\hspace{5mm}G:=(G_1,G_2,G_3,G_4)=(4,3,2,0).\]
Since $F\not\geq G$ nor $G\not\geq F$, the $0$-test for $(F,G)$ returns "No". However, since \[F|_{[3]}=(5,3,1)\geq(3,2,0)=G|^{[3]}, \ \mbox{and}\  G|_{[3]}=(4,3,2)\geq(3,1,0)=F|^{[3]},\] the $1$-test for $(F,G)$ returns "Yes". Also, one can check that for any $k\in\{2,3\}$, the $k$-test returns "Yes" (cf. Remark \ref{rem:k-test} \ref{item:k-test1}).
    \item $(d=2)$ Consider $F,G:[3]^2\rightarrow \Z_+$ defined as follows:
    \[F:=\begin{array}{|c|c|c|}
  \hline F_{(1,3)}&F_{(2,3)}&F_{(3,3)}  \\ \hline F_{(1,2)}&F_{(2,2)}&F_{(3,2)} \\ \hline F_{(1,1)}&F_{(1,2)}&F_{(1,3)} \\ \hline 
\end{array}=\begin{array}{|c|c|c|}
\hline  0&0&0  \\ \hline 3&3&0  \\ \hline 4&3&0 \\ \hline
\end{array},\hspace{5mm}G:=\begin{array}{|c|c|c|}
  \hline G_{(1,3)}&G_{(2,3)}&G_{(3,3)}  \\ \hline G_{(1,2)}&G_{(2,2)}&G_{(3,2)} \\ \hline G_{(1,1)}&G_{(1,2)}&G_{(1,3)} \\ \hline 
\end{array}=\begin{array}{|c|c|c|}
\hline  0&0&0  \\ \hline 2&1&0  \\ \hline 2&2&0 \\ \hline
\end{array}.\]
\end{enumerate}
Since $G\not\geq F$, the $0$-test for $(F,G)$ returns "No". Also, since 
\[G|_{[2]^2}=\begin{array}{|c|c|}
 \hline 2&1 \\ \hline 2&2 \\ \hline 
\end{array} \not\geq \begin{array}{|c|c|}
 \hline 0&0 \\ \hline 3&0 \\ \hline \end{array}= F|^{[2]^2},\] the $1$-test returns "No". Since $4\geq 0$ and $2\geq 0$, one can see that the $2$-test returns "Yes".
\end{example}

Recall the poset category $\Zop$: for any $p,q\in \Z_+$, there exists the unique arrow $p\rightarrow q$ if and only if $p\geq q$.  A \emph{function} $\overline{F}:\N^d\rightarrow \Z_+$ can be regarded as a \emph{functor} $\overline{F}:\N^d\rightarrow \Zop$ if and only if $\overline{F}:\N^d\rightarrow \Z_+$ is order-reversing.

By the definition of interleaving distance, we straightforwardly have:

\begin{proposition}\label{prop:computing}For $n,d\in \N$, let $\overline{F},\overline{G}:\N^d \rightarrow \Zop$ be any two functors with $0=F_{\ba}=G_{\ba}$ for each upper boundary point $\ba\in[n]^d$. Consider the restrictions $F:=\overline{F}|_{[n]^d}$ and $G:=\overline{G}|_{[n]^d}$. Then,

\[\dint\left(\overline{F},\overline{G}\right)=\min\left\{k\in\{0,1,\ldots,n-1\}:\ \mbox{the $k$-test for $(F,G)$ returns "Yes"}\right\}.\]
\end{proposition}

\paragraph{Computational complexity of computing the rank invariant.} Let $\vect$ be the category of vector spaces over a fixed field $\F$ with linear maps. Let $M:[n]^d\rightarrow \vect$ be a (finite) multidimensional module. Let $\mathrm{total}(M):=\sum_{\ba\in [n]^d} \dim(M_\ba)$. In order to compute the rank invariant $\rk(M):[n]^d\rightarrow \Z_+$, one needs $O(\mathrm{total}(M)^\omega)$ operations \cite[Appendix C]{bjerkevik2017computational}, where $\omega$ is the matrix multiplication exponent.

\paragraph{Proposed algorithm for computing $\dint$ and its computational complexity.}
Let $F,G:[n]^d \rightarrow \Z_+$ be any two order-reversing functions.  Based on Proposition \ref{prop:computing}, in order to find the mimimal $k\in\{0,\ldots,n-1\}$ for which the $k$-test for $(F,G)$ (Algorithm \ref{alg:k-test}) returns "Yes", we carry out \textbf{binary search}. 

Let us fix $k\in\{0,\ldots,n-1\}$. For carrying out the $k$-test for $(F,G)$, we compare pairs of integers from the arrays of $F$ and $G$. \emph{Assume that pairs of integers are compared one by one.} Then, notice that, depending on $F$ and $G$, the number of comparisons which are necessary to complete the $k$-test can vary from $1$ to $2(n-k)^d$. Under the assumption that the number of required comparisons is a random variable uniformly distributed in  $\{1,\ldots, 2(n-d)^k\}$ one can conclude that $\frac{1+2(n-d)^k}{2}\approx (n-d)^k$ comparisons are needed on average. Under the preceeding assumptions, by  results from \cite[Section 4]{knight1988search}, we directly have:

\begin{theorem}\label{thm:cost2}
The expected cost of computing $\dintd(F,G)$ is at least $O(n^d\log n)$. Furthermore, the algorithm based on ordinary binary search has this expected cost.
\end{theorem}

By Theorem \ref{thm:cost2}, the expected costs of computing the LHSs of inequalities in Theorems \ref{thm:rank k stability} and \ref{thm:betti-0 stability}, and Remarks \ref{prop:rank k vs GH} and \ref{thm:tighter} are $O(n^d\log n)$ where $d=6,3,2$ and $1$, respectively in order. 

In Section \ref{sec:other metrics} we compare $\dintd$ with the \emph{matching distance} \cite{cerri2013betti,cerri2011new,landi2018rank}, and with the \emph{dimension distance} \cite[Section 4]{dey2018computing}.

\section{Details about stability theorems}\label{sec:proof of main2}
The goal of this section is to prove all theorems in Section \ref{sec:overview} whose proof was not given therein.

\subsection{Interleaving stability of rank invariants and dimension functions}\label{sec:details about the interleaving is bounded by another interleaving}

\paragraph{The rank invariant and its stability.} For any persistence module $M:\R^d\rightarrow \vect$, the rank invariant of $M$ is defined as follows  \cite{carlsson2009theory}:

\begin{definition}[The rank invariant]\label{def:rank invariant} For any $M:\R^d\rightarrow \vect$, the map $\rk(M):\R^{2d}\rightarrow \Z_+\cup\{\infty\}$ defined as
\[\rk(M)(\ba,\bb):=\begin{cases}\rk(\varphi_M(\ba,\bb)),&\ba\leq\bb \ \in \R^d \\ \infty,&\mbox{otherwise.} \end{cases}\]
is called the \emph{rank invariant} of $M$.
\end{definition}
Given any $M:\R^d\rightarrow \vect$, note that for any $\ba'\leq \ba\leq \bb\leq \bb'$ in $\R^d$, 
\[\varphi_M(\ba',\bb')=\varphi_M(\bb,\bb')\circ\varphi_M(\ba,\bb)\circ \varphi_M(\ba',\ba).\]
Hence, we have that $\rk(M)(\ba',\bb')\leq \rk(M)(\ba,\bb)$. This means that $\rk(M)$ is a \emph{functor} between its domain and codomain when regarded 
\begin{enumerate}[label=(\roman*)]
    \item the domain $\R^{2d}$ as the product poset $(\R^d)^{\mathrm{op}}\times\R^d$ and, 
    \item the codomain $\Z_+\cup\{\infty\}$ as the poset $(\Z_+\cup\{\infty\})^{\mathrm{op}}$. 
\end{enumerate}

 We have \emph{stability} of the rank invariant:

\begin{theorem}[Stability of the rank invariant {\cite[Theorem 8.2]{patel2018generalized},\cite[Theorem 22]{puuska2017erosion}}]\label{prop:stability of the rank invariant} For any $M,N:\R^d\rightarrow\vect$,
\begin{equation}\label{eq:stability of the rank invariant}
    \dinttwod(\rk(M),\rk(N))\leq \dint^{\vect}(M,N).
\end{equation}
\end{theorem}
Note that Theorem \ref{prop:stability of the rank invariant} together with Theorem \ref{thm:main2} result in Theorem \ref{thm:rank k stability}. Even though the proof of Theorem \ref{prop:stability of the rank invariant} is given in  {\cite[Theorem 8.2]{patel2018generalized},\cite[Theorem 22]{puuska2017erosion}} in more general setting, we provide a brief version of the proof here. 

\begin{proof}
Since we regard $\rk(M)$ as a functor from $(\R^d,\geq)\times (\R^d,\leq)$ to $(\Z_+\cup\{\infty\})^\mathrm{op}$, for any $\eps\in [0,\infty)$, the $\eps$-shift $\rk(M)(\eps):(\R^d,\geq)\times (\R^d,\leq)\rightarrow(\Z_+\cup\{\infty\})^\mathrm{op}$ of $\rk(M)$ is defined as \[\rk(M)(\eps)_{(\ba,\bb)}=\rk(M)_{(\ba-\vec{\eps},\bb+\vec{\eps})}.\] Similarly, the $\eps$-shift of $\rk(N)$ is defined. 

Suppose that for some $\eps\in [0,\infty)$, the pair  $(f,g)$ is an $\eps$-interleaving pair for $M,N:\R^d\rightarrow \vect$ (Definition \ref{def:interleaving_general}). We show $\rk(N)(\eps)\leq\rk(M)$. Pick any $(\ba,\bb)\in \R^d \times \R^d$. If $\ba\not\leq \bb$ in $\R^d$, then $\rk(M)(\ba,\bb)=\infty$, and thus we trivially have $\rk(N)_{(\ba-\vec{\eps},\bb+\vec{\eps})}\leq \rk(M)_{(\ba,\bb)}$.  If $\ba\leq \bb$ in $\R^d$, then $\ba-\vec{\eps}\leq \bb+\vec{\eps}$, and since 
\[\varphi_N(\ba-\vec{\eps},\bb+\vec{\eps})=f_{\bb}\circ\varphi_M(\ba,\bb)\circ g_{\ba},\]
we have $\rk(N)_{(\ba-\vec{\eps},\bb+\vec{\eps})}\leq \rk(M)_{(\ba,\bb)}$. By symmetry, we also have $\rk(M)(\eps)\leq \rk(N)$, completing the proof.
\end{proof}

\begin{remark}\label{rem:erosion} In order to compare the rank invariants, the author of \cite{puuska2017erosion} makes use of a generalization of the \emph{erosion distance} in \cite{patel2018generalized}, which is denoted by $\dero$ (see Section \ref{sec:other metrics}). It can be deduced that for the LHS of inequality (\ref{eq:stability of the rank invariant}) coincides with  $\dero(\rk(M),\rk(N))$.
\end{remark}

Given $\delta>0$,  deciding whether $\dint^{\vect}(M,N)\leq \delta$  is in general \emph{NP-hard} \cite{bjerkevik2018computing,bjerkevik2017computational}.

In Theorem \ref{prop:stability of the rank invariant}, substituting the comparison of $M$ and $N$ with that of $\rk(M)$ and $\rk(N)$ results in doubling of the underlying dimension of the interleaving distance. This increase of dimension is a price one must pay for substituting the target category $\vect$ with the poset category $(\Z_+\cup\{\infty\})^\mathrm{op}$. Despite the increase in the underlying dimension, as we show in Section \ref{sec:algorithm}, it turns out that computing $\dint$ is easier than computing $\dint^\vect$. 

For any \emph{interval decomposable} modules $M,N:\R^d\rightarrow\vect$, let $\B(M)$ and $\B(N)$ be the barcode of $M$ and $N$, respectively. Then, by \cite[Proposition 2.13]{botnan2018algebraic},
\[\dint^\vect(M,N)\leq \bott\left (\B(M),\B(N)\right).\] 

Hence, together with Theorem \ref{prop:stability of the rank invariant} ,we straightforwardly have:

\begin{corollary}\label{cor:rank vs barcode} For any interval decomposable $M,N:\R^d\rightarrow\vect$, let $\B(M)$ and $\B(N)$ be the barcode of $M$ and $N$, respectively. Then,
\[\dinttwod(\rk(M),\rk(N))\leq \bott\left (\B(M),\B(N)\right).\]
\end{corollary}

\paragraph{Monotonicity and stability of dimension functions for surjective modules.}

\begin{definition}[Surjective persistence modules]\label{def:surjective persistence}Let $\C$ be either $\sets$ or $\vect$ and let $M:\R^d\rightarrow \C$ be any persistence module. We call $M$ \emph{surjective} if $\varphi_M(\ba,\bb):M_{\ba}\rightarrow M_{\bb}$ is surjective for all $\ba\leq \bb$ in $\R^d.$
    
\end{definition}

\begin{example}[The $0$-th homology of the Rips filtration]\label{ex:0th homology} Let $(X,d_X)$ be a metric space. By applying the 0-th (simplicial) homology functor to the \emph{Rips filtration} of $(X,d_X)$, we obtain surjective persistence module $\R\rightarrow \vect$. 
\end{example}

\begin{definition}[Dimension function]\label{def:dimension function} Let $\C$ be either $\sets$ or $\vect$ and let $M:\R^d\rightarrow \C$ be any persistence module. The \emph{dimension function} $\dm(M):\R^d\rightarrow \Z_+$ of $M$ is defined by sending each $\ba\in \R^d$ to the cardinality of $M_\ba$ (when $\C=\sets$) or the dimension of the vector spaces $M_\ba$ (when $\C=\vect$). 
\end{definition}

\begin{remark}\label{rem:decreasing function is a fuction}In Definition \ref{def:dimension function}, if $M$ is a  surjective persistence module, then we can regard $\dm(M)$ as a functor $\R^d\rightarrow \Zop$.
\end{remark}

\begin{proposition}[Interleaving stability of the dimension function]\label{prop:dimension stability via interleaving}Let $\C$ be either $\sets$ or $\vect$ and let $M,N:\R^d\rightarrow \C$ be any two \emph{surjective} persistence modules. Then,

\begin{enumerate}[label=(\roman*)]
  \item $\dintd\left(\dm(M),\dm(N)\right)\leq 2\cdot \dintd^{\C}(M,N).$\label{item:dimension stability1}
\item $\dintd\left(\dm(M),\dm(N)\right)\leq \bdintd^{\C}(M,N),$ \label{item:dimension stability2}
\end{enumerate}
\end{proposition}

\begin{proof} Let us assume that $\C=\sets$. The proof for the case $\C=\vect$ is similar. We show \ref{item:dimension stability1}. Suppose that $(f,g)$ is an $\eps$-interleaving pair between $M$ and $N$. Pick any $\ba\in \R^d$. 
We have $\varphi_N(\ba,\ba+2\vec{\eps})=g_{\ba+\vec{\eps}}\circ f_{\ba}$. Since $\varphi_N(\ba,\ba+2\vec{\eps})$ is surjective, we also have that $g_{\ba+\vec{\eps}}$ is surjective. Since $\varphi_M(\ba,\ba+\eps)$ is also surjective, the composition $g_{\ba+\vec{\eps}}\circ\varphi_M(\ba,\ba+\vec{\eps}):M_\ba\rightarrow N_{\ba+\vec{\eps}}$ is surjective. This implies that $\dm(M)_{\ba}\geq \dm(N)_{\ba+2\vec{\eps}}$. By symmetry, we also have that $\dm(N)_{\ba}\geq \dm(M)_{\ba+2\vec{\eps}}$ for each $\ba\in \R^d$. Therefore, $\dint\left(\dm(M),\dm(N)\right)\leq 2\eps$, as desired.

We prove Item \ref{item:dimension stability2}. Suppose that there exists a \emph{full} $\eps$-interleaving pair between $M$ and $N$. Then, this directly implies that for all $\ba \in \R^d$, $\dm(M)_{\ba}\geq \dm(N)_{\ba+\vec{\eps}}$ and $\dm(N)_{\ba}\geq \dm(M)_{\ba+\vec{\eps}}$.
\end{proof}

\begin{proposition}\label{prop:dim vs rk}Let $\C$ be either $\sets$ or $\vect$ and let $M,N:\R^d\rightarrow \C$ be any two \emph{surjective} persistence modules. Then, 
\[\dinttwod\left(\rk(M),\rk(N)\right)\leq \dintd\left(\dm(M),\dm(N)\right).\]
\end{proposition}

\begin{proof} 
    Suppose that for some $\eps\in [0,\infty)$, $\dint\left(\dm(M),\dm(N)\right)<\eps$.  It suffices to prove that for all $\ba,\bb\in \R^d$ with $\ba\leq \bb$, and for all $\eps'>\eps$ in $[0,\infty)$,    
    \[\rk(N)(\ba-\vec{\eps'},\bb+\vec{\eps'})\leq \rk(M)(\ba,\bb).\]
Invoking that $M$ and $N$ are surjective, notice that $\rk(N)(\ba-\vec{\eps'},\bb+\vec{\eps'})= \dm(N)(\bb+\vec{\eps'})$ and $\rk(M)(\ba,\bb)=\dm(M)(\bb).$ By assumption, we readily have that $\dm(N)(\bb+\vec{\eps'})\leq\dm(M)(\bb)$, completing the proof.
\end{proof}

Proposition \ref{prop:better-dynamic} is a corollary of Proposition \ref{prop:dim vs rk}.

\subsection{Proof of Theorem \ref{thm:main2}}\label{sec:details about PH}

Before showing Theorem \ref{thm:main2}, we begin with the remarks below.

\begin{remark}[Simplicial maps between Rips complexes]\label{rem:rips}For any (semi-)metric spaces\footnote{ We call $(X,d_X)$ a \emph{semi-metric} space if the function $d_X:X\times X\rightarrow \R_+$ satisfies: (1) for all $x\in X$, $d_X(x,x)=0$, and (2) for all $x,x'\in X$, $d_X(x,x')=d_X(x',x)$.} $(X,d_X)$ and $(Y,d_Y)$, and for some $\delta,\delta'\geq 0$, consider the Rips complexes $K=\rips(X,d_X)$ and $L=\mathcal{R}_{\delta'}(Y,d_Y)$ . By the definition of Rips complex, in order to claim that any map $f:X\rightarrow Y$ is a simplicial map, it suffices to show that whenever $x,x'\in X$ with $d_X(x,x')\leq \delta$, it holds that $d_Y(f(x),f(x'))\leq \delta'.$
\end{remark}
For $\bu=[u,u']\in\U$ and $\eps\in[0,\infty)$, let $\bu^\eps:=[u-\eps,u'+\eps]$.
\begin{remark}\label{rem:inequality}  
Let $\gammax$ and $\gammay$ be any two DMSs and let $\tripod$ be a $\eps$-tripod between $\gamma_X$ and $\gamma_Y$. Then it is not difficult to check that for any closed interval $I$ of $\R$, 
\begin{equation}\label{eq:tripod-general}
    \bigvee_{I^\eps}d_X\leq_R \bigvee_I d_Y+2\eps\ \ \mbox{and}\ \ \bigvee_{I^\eps}d_Y\leq_R \bigvee_I d_X+2\eps,
\end{equation}
	which is slightly more general than the condition in (\ref{eq:distor}).	
\end{remark}

\begin{proof}[Proof of Theorem \ref{thm:main2}]\label{proof:main2} 
If $\dintl(\gamma_X,\gamma_Y)=\infty$, there is nothing to prove. Suppose that  $\dintl(\gamma_X,\gamma_Y)<\eps$ for some $\eps \in (0,\infty)$.
Let $\mathcal{S}:=\ripss(\gamma_X)$ and $\mathcal{T}:=\ripss(\gamma_Y)$ (Definition \ref{def:spatiotemporal Rips}). We regard $\U\times\R_+$ as the subposet of $\Rop\times \R\times \R$ (Figure \ref{fig:3d filtration}).  Let $\bv:=\eps(-1,1,2)\in \R^3$. It suffices to show that there are natural transformations $\Phi:\mathcal{S}\Rightarrow \mathcal{T}(\bv)$ and $\Psi:\mathcal{T}\Rightarrow \mathcal{S}(\bv)$ (between the two $\U\times\R_+$-indexed, $\simp$-valued functors) such that for each $(\bu,\delta)\in \U\times \R_+,$ the following diagrams commute up to contiguity:
	\[\begin{tikzcd}
	\mathcal{S}_{(\bu,\delta)} \arrow[rd,swap, "\varphi_{(\bu,\delta)}"] \arrow[rr, "\mathcal{S}\left((\bu\mbox{,}\delta)\leq (\bu^{2\eps}\mbox{,}\delta+4\eps)\right)"] & &\mathcal{S}_{(\bu^{2\eps},\delta+4\eps)}\\
	& \mathcal{T}_{(\bu^{\eps},\delta+2\eps)} \arrow[ru,swap, "\psi_{(\bu^{\eps},\delta+2\eps)}"] &
	\end{tikzcd}
	\begin{tikzcd}
	& \mathcal{S}_{(\bu^{\eps},\delta+2\eps)} \arrow[rd, "\varphi_{(\bu^{\eps},\delta+2\eps)}"] &\\
	\mathcal{T}_{(\bu,\delta)}  \arrow[ru, "\psi_{(\bu,\delta)}"] \arrow[rr,swap, "\mathcal{T}\left((\bu\mbox{,}\delta)\leq (\bu^{2\eps}\mbox{,}\delta+4\eps)\right)"] & &\mathcal{T}_{(\bu^{2\eps},\delta+4\eps)}. 
	\end{tikzcd}\]
	Indeed, by functoriality of homology, the existence of such pair $(\Phi,\Psi)$ of natural transformations guarantees the $\bv$-interleaving between two $(\U\times \R_+)$-indexed modules $\Hrm_{k}\circ \mathcal{S}$ and $\Hrm_{k}\circ \mathcal{T}.$
	
	 Suppose that $\tripod$ is an $\eps$-tripod between $\gamma_X$ and $\gamma_Y$ (Definition \ref{def:distortion}), which exists by the assumption $\dintl(\gamma_X,\gamma_Y)<\eps$. Since $\varphi_X$ and $\varphi_Y$ are surjective, we can take two maps $\phi:X\rightarrow Y$ and $\psi:Y\rightarrow X$ such that
	 \begin{equation}\label{eq:phi and psi}
	     \{(x,\phi(x)):x\in X\}\cup \{(\psi(y),y):y\in Y \}\subset\{(x,y)\in X\times Y: \exists z\in Z, \ x=\varphi_X(z),\ \mbox{and}\ y=\varphi_Y(z)\}.
	 \end{equation}
First, let us check that for any $(\bu,\delta)\in\U\times \R_+,$ $\phi$ is a simplicial map from $\mathcal{S}(\bu,\delta)$ to $\mathcal{T}(\bu^{\eps},\delta+2\eps).$ Fix any $(\bu,\delta)\in \U\times \R_+$, and assume that an 1-simplex $\{x,x'\}\subset X$ is contained in the simplicial complex $\mathcal{S}(\bu,\delta).$ Denoting $\bu=[u,u']$, this means that $\left(\bigvee_{[u,u']}d_X\right)(x,x')\leq \delta.$ By Remark \ref{rem:rips}, it suffices to show that $\left(\bigvee_{[u,u']^\eps}d_Y\right)(\phi(x),\phi(x'))\leq \delta+2\eps.$ This is immediate from the fact that $R$ is an $\eps$-tripod, and the assumption $\{(x,\phi(x)):x\in X\}\subset \varphi_Y\circ\varphi_X^{-1}.$ 

Furthermore, $\phi$ serves as an $\bv$-morphism $\mathcal{S}\Rightarrow \mathcal{T}(\bv)$. Indeed, for any $(\bu,\delta)\leq (J,\delta')$ in $\U\times\R_+$, \[\mathcal{T}\left((\bu^{\eps},\delta+2\eps)\leq({J}^{\eps},\delta'+2\eps)  \right)\circ\phi=\mathrm{id}_Y\circ \phi=\phi=\phi\circ\mathrm{id}_X=\phi\circ\mathcal{S}\left((\bu,\delta)\leq(J,\delta')\right).\] By symmetry, $\psi:Y\Rightarrow X$ also serves as an $\bv$-morphism $\mathcal{T}\Rightarrow\mathcal{S}(\bv)$. 
	
	Next, we show that $(\phi,\psi)$ is an $\eps$-interleaving pair. By symmetry we only prove that for any $(\bu,\delta)\in \U\times \R_+$, $\psi_{(\bu^{\eps},\delta+\eps)}\circ \phi_{(\bu,\delta)}$ is contiguous to $\mathcal{S}\left((\bu,\delta)\leq (\bu^{2\eps},\delta+2\eps)\right)$, which is the identity map on the vertex set $X$. Let $\sigma\subset X$ be a simplex in $\mathcal{S}(\bu,\delta).$ We wish to show that there is a simplex in $\mathcal{S}(\bu^{2\eps},\delta+2\eps)$ that contains both $\sigma$ and the image  $\mathrm{im}(\sigma)$ of $\sigma$ by $\psi_{(\bu^{\eps},\delta+\eps)}\circ \phi_{(\bu,\delta)}$. To this end, we prove that the union $\sigma\cup\mathrm{im}(\sigma)$ has the diameter that is less than or equal to $\delta+2\eps$ in the (semi-)metric space $(X,\bigvee_{[u,u']^{2\eps}}d_X).$ Invoking Remark \ref{rem:inequality}, we consider the following three different cases of choosing any two elements in $\sigma\cup\mathrm{im}(\sigma)$:
	
	\begin{enumerate}[label=(\roman*)]
		\item Take any $x,x'\in \sigma$. Since $\sigma$ is a simplex in the Rips complex  $\mathcal{S}(\bu,\delta)=\newline \rips(X,\bigvee_{[u,u']}d_X),$ we have \[\left(\bigvee_{[u,u']^{2\eps}}d_X\right)(x,x')\leq \left(\bigvee_{[u,u']}d_X\right)(x,x')\leq \delta <\delta+2\eps.\]
	\end{enumerate}
Let $\tilde{R}:=\{(x,\phi(x)):x\in X\}\cup \{(\psi(y),y):y\in Y \}$ (see the inclusion in (\ref{eq:phi and psi})).
    \begin{enumerate}[resume,label=(\roman*)]
		\item Take $x\in \sigma$ and $x'\in\mathrm{im}(\sigma)$. Then $x'=\psi\circ\phi(x'')$ for some $x''\in \sigma$. Since $(x,\phi(x))$, $(x',\phi(x'')), (x'',\phi(x''))\in \tilde{R}$,
		\[\left(\bigvee_{[u,u']^{2\eps}}d_X\right)(x,x')\leq \left(\bigvee_{[u,u']^{\eps}}d_Y\right)(\phi(x),\phi(x''))+\eps\leq \left(\bigvee_{[u,u']}d_X\right)(x,x'')+2\eps\leq \delta+2\eps.\]
		
		\item Take any $x,x'\in \mathrm{im}(\sigma)$. Then there are $x'',x'''\in \sigma$ which are sent to $x,x'$ via $\psi\circ\phi$, respectively. Since $(x,\phi(x'')),(x',\phi(x''')),(x'',\phi(x'')), (x''',\phi(x'''))\in \tilde{R}$,
		\[\left(\bigvee_{[u,u']^{2\eps}}d_X\right)(x,x')\leq \left(\bigvee_{[u,u']^\eps}d_Y\right)(\phi(x''),\phi(x'''))+\eps\leq \left(\bigvee_{[u,u']}d_X\right)(x'',x''')+2\eps\leq \delta+2\eps.\]
	\end{enumerate}
\end{proof}


\subsection{Proof of  Proposition \ref{prop:decreasing property}}
\begin{lemma}[Convexity of admissible vectors]\label{lem:convexity}Suppose that $\ba,\bb\in \R^6_{\times}$ are admissible with $\ba\leq \bb$. Then, any $\bc\in \R_{\times}^6$ such that $\ba\leq\bc\leq\bb$ is also admissible.
\end{lemma}

\begin{proof} Let $\ba:=(a_i)_{i=1}^6$ and $\bb:=(b_i)_{i=1}^6$ and $\bc=(c_i)_{i=1}^6$. From the assumptions that  $\ba\leq\bc\leq\bb$ and that $\ba,\bb$ are admissible, one can see that 
\[b_4\leq c_4\leq a_4 \leq a_1\leq c_1 \leq b_1 \leq b_2\leq c_2 \leq a_2 \leq  a_5 \leq c_5\leq b_5,\ \mbox{and}\]
\[0\leq b_3\leq c_3\leq a_3\leq a_6\leq c_6\leq b_6.\]
Therefore, $\bc$ is admissible.
\end{proof}

\begin{proof}[Proof of Proposition \ref{prop:decreasing property}]\label{proof:decreasing property}  Pick $\ba,\bb \in \R_\times ^6$ such that $\ba\leq \bb$.  We consider the following cases:

\begin{enumerate}[label=(\roman*)]
    \item Both $\ba$ and $\bb$ are admissible.\label{item:both admi}
    \item $\ba$ is admissible and $\bb$ is non-admissible.\label{item:admi and non}
    \item $\ba$ is non-admissible and $\bb$ is admissible.\label{item:non and admi}
    \item Both $\ba$ and $\bb$ are non-admissible. \label{item:both non}
\end{enumerate}
In case \ref{item:both admi}, let $\ba=(a_1,a_2,a_3,a_4,a_5,a_6)$ and $\bb=(b_1.b_2,b_3,b_4,b_5,b_6)$. Then we have the inclusions
\[\mathcal{R}_{b_3}\left(X,\bigvee_{[b_1,b_2]}d_X\right)\stackrel{i_1}{\hookrightarrow}\mathcal{R}_{a_3}\left(X,\bigvee_{[a_1,a_2]}d_X\right)\stackrel{i_2}{\hookrightarrow}\mathcal{R}_{a_6}\left(X,\bigvee_{[a_4,a_5]}d_X\right)\stackrel{i_3}{\hookrightarrow}\mathcal{R}_{b_6}\left(X,\bigvee_{[b_4,b_5]}d_X\right).\]
By applying $\Hrm_k$ to the above inclusions, we obtain the diagram of vector spaces and linear maps \[V_1\stackrel{\Hrm_k(i_1)}{\longrightarrow}V_2\stackrel{\Hrm_k(i_2)}{\longrightarrow}V_3\stackrel{\Hrm_k(i_3)}{\longrightarrow}V_4.\]
Notice that $\rk_k(\ba)$ is the rank of $\Hrm_k(i_2)$, whereas $\rk_k(\bb)$ is the rank of $\Hrm_k(i_3)\circ\Hrm_k(i_2)\circ \Hrm_k(i_1)$. This implies that $\rk_k(\ba)\geq \rk_k(\bb)$. In case \ref{item:admi and non}, $\bb$ cannot be trivially non-admissible by definition. Therefore, $\rk_k(\gamma_X)(\bb)=0$.
In case \ref{item:non and admi}, by Lemma \ref{lem:convexity}, $\ba$ must be trivially non-admissible and hence $\rk_k(\gamma_X)(\ba)=\infty$. In case \ref{item:both non}, by the definition of trivially non-admissible, it is not possible that $\ba$ is non-trivially non-admissible with $\bb$ being trivially non-admissible. Therefore, we always have \newline $\rk_k(\gamma_X)(\ba)\geq \rk_k(\gamma_X)(\bb)$.
\end{proof}

\subsection{Spatiotemporal Dendrogram of a DMS and Proof of Theorem \ref{thm:betti-0 stability}}\label{sec:details about SC}

\paragraph{Overview of the proof.} The Betti-$0$ function of a DMS $\gamma_X$ can be obtained by the two steps: First, adapting the ideas of the SLHC method (Section \ref{sec:SLHC})
, we induce the \emph{spatiotemporal SLHC dendrogram} $\theta(\gamma_X)$ of $\gamma_X$. Then, the dimension function $\dm\left(\theta(\gamma_X)\right)$ (Definition \ref{def:dimension function}) of $\theta(\gamma_X)$ coincides with the Betti-$0$ function of $\gamma_X$ given in Definition \ref{def:betti0}. Therefore, by proving that each of the successive associations $\gamma_X\mapsto \theta(\gamma_X) \mapsto \dm\left(\theta(\gamma_X)\right)$ is stable, we can show Theorem \ref{thm:betti-0 stability}.

\paragraph{Partition category and dendrograms.} Let $X$ be a non-empty finite set. Given any two partitions $P,Q$ of $X$, we write $P\leq Q$ if $P$ refines $Q$, i.e. for all $B\in P$, there exists a (unique) $C\in Q$ such that $B\subset C$. In this case, the surjective map $P\twoheadrightarrow Q$ sending each $B\in P$ to the unique block $C\in Q$ such that $B\subset C$ is called the \emph{natural map} from $P$ to $Q$.

\begin{definition}[$\Part(X)$ and its structure]\label{def:part} Let $X$ be a non-empty finite set.\newline  By $\Part(X)$, we mean the subcategory of $\sets$ described as follows:
\begin{enumerate}[label=(\roman*)]
    \item Objects: All partitions of $X$.
    \item Morphisms: For any two partitions $P,Q$ of $X$ with $P\leq Q$, the unique morphism $P\twoheadrightarrow Q$ is the natural map.
\end{enumerate}
\end{definition}
We remark that any partition $P$ of $X$ has the corresponding equivalence relation $\sim$ on $X$. Namely, $P=X/\sim$, where $x\sim x'$ if and only if $x,x'$ belong to the same block of $P$.

\begin{definition}[Dendrogram]\label{def:general dendrogram} Let $X$ be a non-empty finite set and let $\Pb$ be any poset. We will call any functor $\Pb\rightarrow \Part(X)$ a \emph{$\Pb$-indexed dendrogram over $X$} or simply a \emph{dendrogram}.
\end{definition}

\paragraph{The spatiotemporal SLHC dendrogram of a DMS.}   We aim at encoding multiscale clustering features of a DMS into a \emph{single} dendrogram (Definition \ref{def:spatiotemporal dendrogram}). Since we take into account both temporal and spatial parameters, this dendrogram will have a \emph{multidimensional} indexing poset, in contrast to its counterpart for a static metric space (Definition \ref{def:SLHC dendrogram}). We prove that this dendrogram is stable under perturbation of the input DMS (Theorem \ref{thm:stability of spatiotemporal dendrogram}). 

Let $\gammax$ be a DMS.  For $\bu\in \U$ and $\delta\in \R_+$, we define the equivalence relation $\sim_{X,\delta}^{\bu}$ on $X$ as follows:
	\[x\sim_{X,\delta}^{\bu} x'\ \Leftrightarrow\  \exists x=x_0,x_1,\ldots,x_n=x'\ \mbox{in $X$ s.t.} \left(\bigvee_{\bu}d_X\right)(x_i,x_{i+1})\leq \delta.\] 
Observe that, for any pair $(\bu,\delta)\leq (J,\delta')$ in $\U\times \R_+$, the relation $\sim_{X,\delta}^\bu$ is contained in $\sim_{X,\delta'}^{J}$ and hence	
\begin{equation}\label{eq:monotonic--dendrogram for dmss}
    \left(X/\sim_{X,\delta}^{\bu}\right) \leq \left(X/\sim_{X,\delta'}^{J}\right).
\end{equation}

By this monotonicity in (\ref{eq:monotonic--dendrogram for dmss}), we can extend the notion of SLHC dendrogram for static metric spaces (Definition \ref{def:SLHC dendrogram}) to the \emph{spatiotemporal SLHC dendrogram} of a DMS:
	
\begin{definition}[The spatiotemporal SLHC dendrogram of a DMS]\label{def:spatiotemporal dendrogram}Given any DMS $\gammax$, we define the \emph{spatiotemporal SLHC dendrogram} $\theta(\gamma_X):\U\times\R_+\rightarrow \Part(X)$ of $\gamma_X$ as follows:
\begin{enumerate}[label=(\roman*)]
    \item To each $(\bu,\delta)\in \U\times \R_+$, assign the partition $X/\sim_{X,\delta}^\bu$ of $X$.
    \item To each pair $(\bu,\delta)\leq(J,\delta')$ in $\U\times \R_+$, assign the natural map (Definition \ref{def:part})  \[X/\sim_{X,\delta}^\bu \twoheadrightarrow  X/\sim_{X,\delta'}^{J}.\] 
\end{enumerate}
\end{definition}

In order to prove Theorem \ref{thm:betti-0 stability}, we need:

\begin{theorem}[Stability of the spatiotemporal SLHC dendrogram]\label{thm:stability of spatiotemporal dendrogram}
Then,
\[\bdint^{\Sets}(\theta(\gamma_X),\theta_Y(\gamma_Y))\leq 2\cdot \dintl(\gamma_X,\gamma_Y).\]
\end{theorem}

The proof of Theorem \ref{thm:betti-0 stability} will be straightforward by re-interpreting Definition \ref{def:betti0}:

\begin{definition}[Another interpretation of Definition {\ref{def:betti0}}]\label{def:betti0-rigorous}Let $\gammax$ be a DMS. We define the Betti-0 function $\beta_0^{\gamma_X}:\U\times\R_+\rightarrow \Z_+$ of $\gamma_X$ as the \emph{dimension function} of the spatiotemporal dendrogram $\theta(\gamma_X):\U\times \R_+\rightarrow \Part(X)$ of $\gamma_X$. In other words, $\beta_0^{\gamma_X}$ sends each  $(\bu,\delta)\in \U\times \R_+$ to the number of blocks in the partition $\theta(\gamma_X)(\bu,\delta).$
\end{definition}

\begin{proof}[Proof of Theorem \ref{thm:betti-0 stability}] Invoking that $\beta_0^{\gamma_X}$ and $\beta_0^{\gamma_Y}$ are the dimension functions of $\theta(\gamma_X)$ and $\theta(\gamma_Y)$, respectively, the proof straightforwardly follows from  Proposition \ref{prop:dimension stability via interleaving} and Theorem \ref{thm:stability of spatiotemporal dendrogram}. 
\end{proof}

\begin{proof}[Proof of Theorem \ref{thm:stability of spatiotemporal dendrogram}]Let $M:=\theta(\gamma_X):\U\times\R_+\rightarrow \Part(X)(\hookrightarrow \sets)$ and $N:=\theta(\gamma_Y):\U\times\R_+\rightarrow \Part(Y)(\hookrightarrow\sets)$. For each $(I,\delta)\in \U\times\R_+$, consider the equivalence relation $\sim^{X}_{I,\delta}$ on $X$ defined, for any $x,x'\in X$, as $x\sim^{X}_{I,\delta} x'$ if and only if there is a sequence $x=x_0,x_1,\ldots,x_l=x'$ in $X$ such that  $\bigvee_{I}d_X(x_i,x_{i+1})\leq \delta$  for each $i=0,\ldots,l-1$. Similarly, define the equivalence relation $\sim^{Y}_{I,\delta}$ on $Y$. Note that, by definition of $M$ and $N$,

\[M_{(I,\delta)}=X\big/\sim^{X}_{I,\delta}\ \ \mbox{and}\hspace{3mm} N_{(I,\delta)}=Y\big/\sim^{Y}_{I,\delta}.\]

For $x\in X$, let $[x]_{(I,\delta)}^X$ be the block containing $x$ in the partition $M_{(I,\delta)}$. Then, for any $(I,\delta),(J,\delta')\in \U\times\R_+$ with $(I,\delta)\leq(J,\delta')$, the internal morphism $\varphi_M((I,\delta),(J,\delta'))$ of $M$ sends $[x]_{(I,\delta)}^X$ to $[x]_{(J,\delta')}^X$ for each $x\in X$. We can describe the internal morphisms of $N$ in the same way.

Suppose that $2\ \dintl\left(\gamma_X,\gamma_Y\right)<\eps$ for some $\eps\in(0,\infty)$. Then, there exists an  $(\eps/2)$-tripod  $\tripod$ between $\gamma_X$ and $\gamma_Y$ (Definitions  \ref{def:distortion} and \ref{def:lambda dist}).

Since two maps $\varphi_X:Z\rightarrow X$ and $\varphi_Y:Z\rightarrow Y$ are surjective, we can take two maps $f:X\rightarrow Y$ and $g:Y\rightarrow X$ such that 
\begin{equation}\label{eq:the map f}
	\{(x,f(x)):x\in X\}\cup \{(g(y),y):y\in Y\}\subset \{(x,y)\in X\times Y: \exists z\in Z, \ x=\varphi_X(z),\ \mbox{and}\ y=\varphi_Y(z)\}.
\end{equation} 
We will show that $f,g$ induce a full $\eps$-interleaving pair between $M$ and $N$. For any $I=[u,u']\in\U$ and any $\alpha\in [0,\infty)$, let $I^{\alpha}:=[u-\alpha,u'+\alpha]$. 
For each $(I,\delta)\in \U\times\R_+$, we define $\bar{f}_{(I,\delta)}:M_{(I,\delta)}\rightarrow N_{(I^\eps,\delta+\eps)}$ as
\[[x]_{(I,\delta)}^X \mapsto \left[f(x)\right]_{(I^\eps,\delta+\eps)}^Y,\hspace{5mm}x\in X.\]
Similarly, we define $\bar{g}_{(I,\delta)}:N_{(I,\delta)}\rightarrow M_{(I^\eps,\delta+\eps)}$. It suffices to show that for each $(I,\delta)\in \U\times\R_+$,

\begin{enumerate}[label=(\roman*)]
    \item  $\bar{f}_{(I,\delta)}$  (resp. $\bar{g}_{(I,\delta)}$) is a well-defined set map from $M_{(I,\delta)}$ to $N_{(I^\eps,\delta+\eps)}$ (resp. from $N_{(I,\delta)}$ to $M_{(I^\eps,\delta+\eps)}$),\label{item:well-defined}
    \item $\bar{f}_{(I,\delta)}:M_{(I,\delta)}\rightarrow N_{(I^\eps,\delta+\eps)}$ and $\bar{g}_{(I,\delta)}:N_{(I,\delta)}\rightarrow M_{(I^\eps,\delta+\eps)}$ are surjective.\label{item:surjective}
    \item when $(I,\delta)\leq(J,\delta')$ in $\U\times\R_+$, \[\varphi_N((I^\eps,\delta+\eps),(J^\eps,\delta'+\eps))\circ\bar{f}_{(I,\delta)}=\bar{f}_{(J,\delta')}\circ\varphi_M((I,\delta),(J,\delta')),\] 
    \[ \varphi_M((I^\eps,\delta+\eps),(J^\eps,\delta'+\eps))\circ\bar{g}_{(I,\delta)}=\bar{g}_{(J,\delta')}\circ\varphi_N((I,\delta),(J,\delta')).\]\label{item:naturality}
    \item  $\bar{g}_{(I^\eps,\delta+\eps)}\circ \bar{f}_{(I,\delta)}=\varphi_M((I,\delta),(I^{2\eps},\delta+2\eps))$, and\newline $\bar{f}_{(I^\eps,\delta+\eps)}\circ \bar{g}_{(I,\delta)}=\varphi_N((I,\delta),(I^{2\eps},\delta+2\eps)).$\label{item:interleaving condition}
\end{enumerate}
 
We prove \ref{item:well-defined}. Fix $(I,\delta)\in \U\times \R_+$. Suppose that $x'\in [x]_{(I,\delta)}^X$. It suffices to show that $f(x')\in [f(x)]_{(I^\eps,\delta+\eps)}^Y$. By assumption, there exist $x=x_0,\ldots,x_l=x'$ in $X$ such that $\bigvee_{I}d_X(x_i,x_{i+1})\leq \delta$, $i=1,\ldots,l-1$. Then, invoking $R$ is an $(\eps/2)$-tripod between $\gamma_X$ and $\gamma_Y$ (see (\ref{eq:distor})), together with assumption (\ref{eq:the map f}) and Remark \ref{rem:inequality}, \[\bigvee_{I^{\eps}}d_Y(f(x_i),f(x_{i+1}))\leq \bigvee_{I^{(\eps/2)}}d_Y(f(x_i),f(x_{i+1}))\leq \delta+\eps
,\ \mbox{for}\ i=1,\ldots,l-1.\] This directly implies that $f(x')\in [f(x)]_{(I^\eps,\delta+\eps)}^Y$. In a similar way, it can be proved that $\bar{g}_{(I,\delta)}$ is well-defined.

Now we show \ref{item:surjective}. Fix $(I,\delta)\in \U\times \R_+$. We only prove that $\bar{f}_{(I,\delta)}:M_{(I,\delta)}\rightarrow N_{(I^\eps,\delta+\eps)}$ is surejctive. Pick any $[y]^Y_{(I^\eps,\delta+\eps)}\in N_{(I^\eps,\delta+\eps)}.$ Since $\varphi_Y:Z\rightarrow Y$ is surjective, there exists $z\in Z$ such that $\varphi_Y(z)=y.$ Let $x:=\varphi_X(z)$. Then, invoking $R$ is an $(\eps/2)$-tripod between $\gamma_X$ and $\gamma_Y$, together with assumption (\ref{eq:the map f}) and Remark \ref{rem:inequality}, \[\bigvee_{I^{\eps}}d_Y(y,f(x))\leq \bigvee_{I^{\eps/2}}d_Y(y,f(x))\leq \bigvee_{I} d_X(x,x)+\eps=0+\eps \leq \delta+\eps.\] 
This implies that
 $[f(x)]^Y_{(I^\eps,\delta+\eps)}=[y]^Y_{(I^\eps,\delta+\eps)}$. Also, by definition of $\bar{f}_{(I,\delta)}$, $[x]^X_{(I,\delta)}$ is sent to $[y]^Y_{(I^\eps,\delta+\eps)}$ via $\bar{f}_{(I,\delta)}$. Since $[y]^Y_{(I^\eps,\delta+\eps)}\in N_{(I^{\eps},\delta+\eps)}$ was arbitrary chosen, we have shown the surjectivity of $\bar{f}_{(I,\delta)}$.
 
 Next we prove \ref{item:naturality}. Fix $(I,\delta)\leq(J,\delta')$ in $\U\times\R_+$. We only show \[\varphi_N((I^\eps,\delta+\eps),(J^\eps,\delta'+\eps))\circ\bar{f}_{(I,\delta)}=\bar{f}_{(J,\delta')}\circ\varphi_M((I,\delta),(J,\delta')).\] By the definition of maps $\varphi_M(\cdot,\cdot),\varphi_N(\cdot,\cdot),\bar{f}_{(\cdot,\cdot)}$ and $\bar{g}_{(\cdot,\cdot)}$, for any $[x]_{(I,\delta)}^X \in M_{(I,\delta)}$, 
 \begin{align*}
     \varphi_N((I^\eps,\delta+\eps),(J^\eps,\delta'+\eps))\circ\bar{f}_{(I,\delta)} \left([x]_{(I,\delta)}^X \right)&=\varphi_N((I^\eps,\delta+\eps),(J^\eps,\delta'+\eps))\left(\left[f(x)\right]^Y_{(I^\eps,\delta+\eps)}\right)\\&=\left[f(x)\right]^Y_{(J^\eps,\delta'+\eps)},
 \end{align*}
 \[\bar{f}_{(J,\delta')}\circ\varphi_M((I,\delta),(J,\delta'))\left([x]_{(I,\delta)}^X\right)=\bar{f}_{(J,\delta')}\left([x]_{(J,\delta')}^X\right)=\left[f(x)\right]^Y_{(J^\eps,\delta'+\eps)}.\]
 
 Finally, we prove \ref{item:interleaving condition}.  Fix $(I,\delta)\in \U\times \R_+$. We only show \[\bar{g}_{(I^\eps,\delta+\eps)}\circ \bar{f}_{(I,\delta)}=\varphi_M((I,\delta),(I^{2\eps},\delta+2\eps)).\] Take any $[x]_{(I,\delta)}^X \in M_{(I,\delta)}$. Then, by $\bar{g}_{(I^\eps,\delta+\eps)}\circ \bar{f}_{(I,\delta)}$, the block $[x]_{(I,\delta)}^X$ is sent to $[g\circ f(x)]_{(I^{2\eps},\delta+2\eps)}^X$. By invoking that $R$ is an $(\eps/2)$-tripod between $\gamma_X$ and $\gamma_Y$ and (\ref{eq:the map f}) and Remark \ref{rem:inequality}, we also have 
 \[\bigvee_{I^{2\eps}}d_X\left(x,g\circ f(x)\right)\leq \bigvee_{I^{(\eps/2)}}d_X\left(x,g\circ f(x)\right)\leq  \bigvee_{I}d_Y(f(x),f(x))+ \eps =0+\eps \leq \delta+2\eps.\]This implies that $[x]_{\delta+2\eps}^X=[g\circ f(x)]_{\delta+2\eps}^X$, completing the proof.
\end{proof}

For $t\in \R$, consider $\bt\in \U$.
\begin{remark}[Comprehensiveness of Definition \ref{def:spatiotemporal dendrogram}]\label{rem:basic0}We remark the following (see Figure \ref{fig:comprehensiveness}):
\begin{enumerate}[label=(\roman*)]
    \item Consider the constant DMS $\gamma_X\equiv (X,d_X)$ as in Example \ref{ex:constant}. Then, the spatiotemporal SLHC dendrogram of $\gamma_X$ is amount to the SLHC dendrogram (Definition \ref{def:SLHC dendrogram}) of $(X,d_X)$: for all $\bu\in \U$ and $\delta\in \R_+$, \[\theta(\gamma_X)_{(\bu,\delta)}=\theta(X,d_X)_\delta.\] \label{item:basic0-1}
    \item Let $\gammax$ be a DMS. For each $t\in \R$, we have the SLHC dendrogram $\theta(X,d_X(t)):\R_+\rightarrow \Part(X)$ of the metric space $(X,d_X(t))$ (Definition \ref{def:SLHC dendrogram}). All those dendrograms are incorporated by $\theta(\gamma_X)$ in the following sense:
    \[\theta_X(\gamma_X)_{(\bt,\delta)}=\theta(X,d_X(t))_\delta, \ \ t\in \R, \ \delta\in \R_+.\]\label{item:basic0-2}
\end{enumerate}
\end{remark}
\begin{figure}
    \centering
    \includegraphics[width=0.4\textwidth]{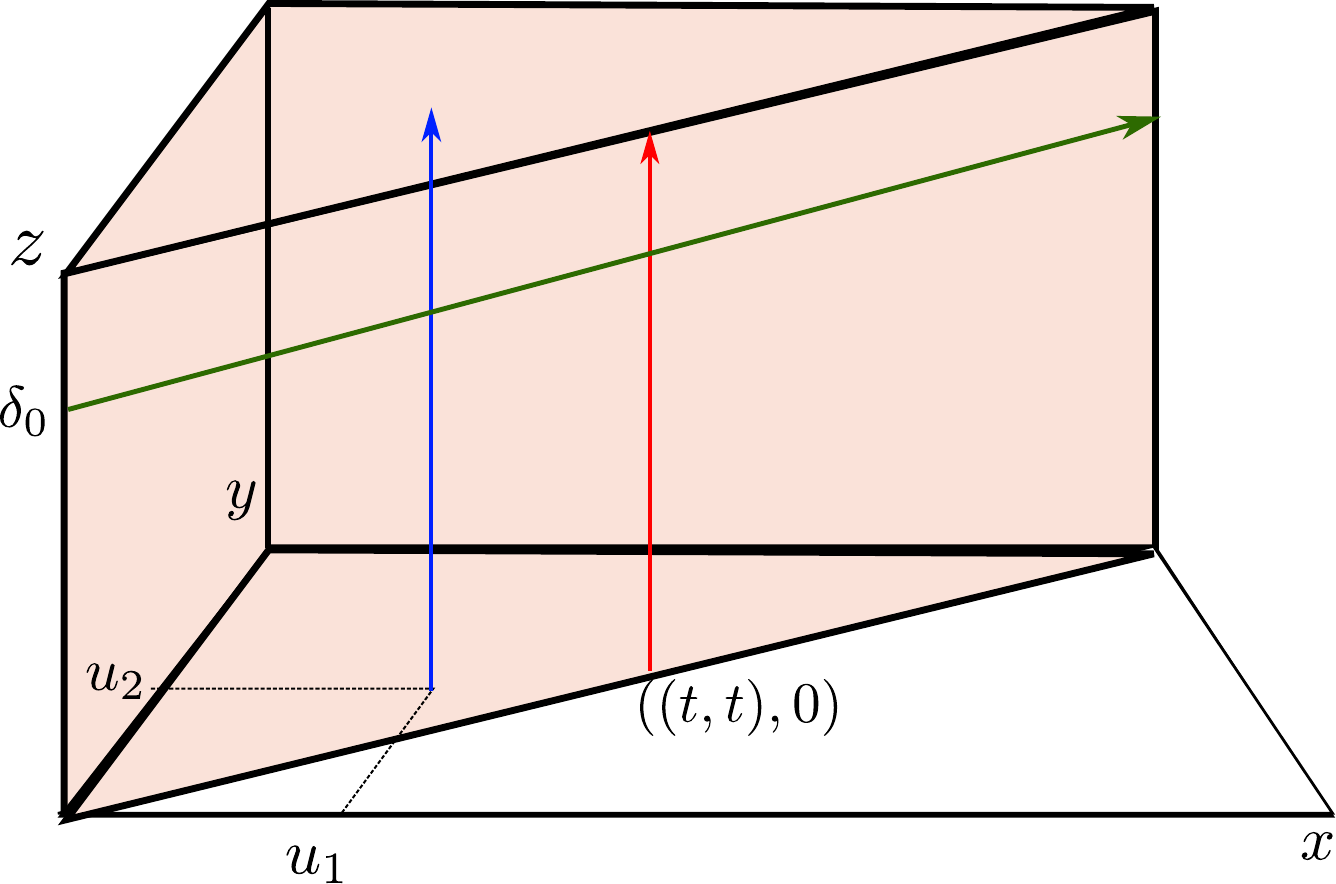}
    \caption{Consider a DMS $\gammax$. (1) If  $\gamma_X\equiv(X,d_X)$, then the SLHC dendrogram $\theta(X,d_X)$ is encoded along any vertical ray, such as blue or red rays in the figure (Remark \ref{rem:basic0} \ref{item:basic0-1}). (2) For each $t\in \R$, the SLHC dendrogram $\theta(X,d_X(t))$ of $(X,d_X(t))$ is recorded along the red ray (Remark \ref{rem:basic0}\ref{item:basic0-2}) (3) Along the greed horizontal line at height $\delta_0$ over the diagonal plane $y=x$, the formigram induced from $\gamma_X$ with respect to the connectivity parameter $\delta_0$ is encoded.}
    \label{fig:comprehensiveness}
\end{figure}
\begin{remark}[Connection to {\cite{kim2017stable}}]\label{rem:formigram}
Let $\gammax$ be a DMS and fix $\delta_0\in \R_+$. The map $\theta_X^\delta:\R\rightarrow \Part(X)$ defined as \[\theta_X^{\delta_0}(t)=X/\sim_{X,\delta_0}^{\bt}\ \mbox{for all $t\in\R$}\] is the \emph{formigram} induced from $\gamma_X$ with respect to the connectivity parameter $\delta$ \cite{kim2017stable}. 
\end{remark}

\subsection{Proof of Theorem \ref{thm:better}}\label{sec:new lower bound}

\begin{proof}[Proof of Theorem \ref{thm:better}] We utilize $\lmulti\cdot \rmulti$ instead of $\{\cdot\}$ to denote \emph{multisets.}
Let $m:=\abs{X}$, $n:=\abs{Y}$, and without loss of generality assume that $m\leq n$. Then, for some  $a_1\leq \ldots \leq a_{m-1}$, and $b_1\leq \ldots \leq b_{n-1}$ in $\R_+$, we have \[\mathcal{A}:=\dgm_0\left(\ripsss(X,d_X)\right)\setminus \lmulti(0,+\infty)\rmulti=\lmulti(0,a_i)\rmulti_{i=1}^{m-1},\]
\[  \mathcal{B}:=\dgm_0\left(\ripsss(Y,d_Y)\right)\setminus \lmulti(0,+\infty)\rmulti=\lmulti(0,b_j)\rmulti_{j=1}^{n-1}.\] 
Then, 	\[\bott\left(\dgm_0\left(\ripsss(X,d_X)\right),\dgm_0\left(\ripsss(Y,d_Y)\right) \right)= \bott(\mathcal{A},\mathcal{B}).\]
Let $A=\lmulti a_1',\ldots,a_{n-1}'\rmulti$ and $B=\lmulti b_1,\ldots,b_{n-1} \rmulti$, where $A$ consists of $n-m$ zeros at the beginning, followed by the sequence $a_1,a_2\ldots,a_{m-1}.$ Then, notice that 
	\[ \bott(\mathcal{A},\mathcal{B})\leq\max_{i=1}^{n-1}\abs{a_i'-b_i}.\]
Therefore, it suffices to show that
\begin{equation}\label{eq:betti and bottle}
    \max_{i=1}^{n-1}\abs{a_i'-b_i}\leq \dint\left(\beta_0^{(X,d_X)},\beta_0^{(Y,d_Y)}\right).
\end{equation}
Let $\eps:=\dint\left(\beta_0^{(X,d_X)},\beta_0^{(Y,d_Y)}\right),$ i.e. 
\begin{equation}\label{eq:interleaving}
\mbox{for all $\delta\in \R_+$},\ 
\beta_0^{(X,d_X)}(\delta+\eps)\leq \beta_0^{(Y,d_Y)}(\delta), \mbox{and}\ \ \beta_0^{(Y,d_Y)}(\delta+\eps)\leq \beta_0^{(X,d_X)}(\delta).
\end{equation}
Observe the following:
\begin{enumerate}[label=(\roman*)]
	\item $\beta_0^{(X,d_X)},\beta_0^{(Y,d_Y)}$ are monotonically decreasing as maps from $\R_+$ to $\Z_+$,\label{item:obs1}
	\item For $0\leq \delta<a_{1}=a_{n-m+1}',$ we have $\beta_0^{(X,d_X)}(\delta)=m$,\label{item:obs2} 
	\item For integers $k=1,\ldots, m-1$, we have $ a_{n-k}'=\min\left\{\delta\in \R_+:\beta_0^{(X,d_X)}(\delta)=k\right\},$ \label{item:obs3}
	\item For integers $k=1, \ldots, n-1$, we have $ b_{n-k}=\min\left\{\delta\in \R_+:\beta_0^{(Y,d_Y)}(\delta)=k\right\}.$ \label{item:obs4}
\end{enumerate}	
In order to show inequality (\ref{eq:betti and bottle}), first we show that $\abs{a_i'-b_i}\leq \eps$  for $1\leq i\leq n-m$. By construction we have $a_1'=a_2'=\ldots=a_{n-m}'=0$, and thus it suffices to show that $b_{i}\leq \eps$ for $1\leq i\leq n-m$. By the assumption in (\ref{eq:interleaving}) and item \ref{item:obs2}, we have 
\[\beta_0^{(Y,d_Y)}(\eps) \leq \beta_0^{(X,d_X)}(0)=m.\]
Also, by items \ref{item:obs1} and \ref{item:obs4}, we have $b_{n-m}\leq \eps.$ Since $b_1\leq b_2\leq \ldots \leq b_{n-m-1}\leq b_{n-m}$, we have shown that $b_i\leq \eps$ for  $1\leq i\leq n-m$, as desired. 

Now we show that $\abs{a_i'-b_i}\leq \eps$ for $i=n-m+1,n-m+2,\ldots,  n-1$. By re-indexing it suffices to prove that $\abs{a_{n-k}'-b_{n-k}}\leq \eps$ for $k=1,\ldots, m-1$.  Notice that, for $k=1,\ldots, m-1$, by the assumption in (\ref{eq:interleaving}) and item \ref{item:obs4}, we have 
\[\beta_0^{(X,d_X)}(b_{n-k}+\eps)\leq \beta_0^{(Y,d_Y)}(b_{n-k})=k.\]
Then by items \ref{item:obs1} and \ref{item:obs3}, we have that $a'_{n-k}\leq b_{n-k}+\eps$. Similarly, one can prove that for $k=1,\ldots, m-1$, it holds that $b_{n-k}\leq a_{n-k}'+\eps$. Therefore, we have $\abs{a_{n-k}'-b_{n-k}}\leq \eps$ for $k=1,\ldots, m-1$, as desired.
\end{proof}

\section{Discussion}\label{sec:conclusion}

The primary contribution of this paper is to construct multiparameter persistent homology groups from dynamic metric data. Not only are these persistent homology groups stable to perturbations of the input, but also this stability result turns out to be a generalization of a fundamental stability theorem in topological data analysis. A second practical contribution of our paper is to propose a polynomial time algorithm that can be carried out for quantifying the behavioral difference between two dynamic metric data sets.

\newpage

\appendix

\section{Discretization of DMSs}\label{sec:discretization of a DMS}

In order to compute the lower bound for the distance $\dintl$ given in Theorems \ref{thm:rank k stability} and \ref{thm:betti-0 stability}  in practice, we need to \emph{discretize} DMSs, i.e. turn DMSs into a locally constant DMSs. This discretization depends on the resolution parameter $\alpha\in (0,\infty)$, described as below. We will show that, if $\alpha$ is small and DMSs $\gamma_X$ and $\gamma_Y$ satisfy a mild assumption, then the lower bounds for $\dintl(\gamma_X,\gamma_Y)$ given in Theorems \ref{thm:rank k stability} and \ref{thm:betti-0 stability} can be well-approximated using the $\alpha$-\emph{discretized} DMSs associated to $\gamma_X$ and $\gamma_Y$.

We call any map $i:\Z^d\rightarrow \R^d$ \emph{grid-like} if $i$ is an \emph{strictly} injective poset morphism, i.e. 

\begin{enumerate}[label=(\roman*)]
    \item for any pair $\ba=(a_1,\ldots,a_d)<\bb=(b_1,\ldots,b_d)$ with $a_i<b_i$, $i=1,\ldots,d$ in $\Z^d$, for $i(\ba)=(a_1',\ldots,a_d')$ and $i(\bb)=(b_1',\ldots,b_d')$, we have $a_i'<b_i'$, $i=1,\ldots,d$.
    \item For all $\bc=(c_1,\ldots,c_d) \in \R^d$, there are $\ba,\bb \in \Z^d$ such that $i(\ba)\leq \bc \leq i(\bb)$.
\end{enumerate}

Given a grid-like $i:\Z^d \rightarrow \R^d$, for any $\ba\in \R^d$, define $\lfloor \ba \rfloor_i$ to be the maximum element in the image of $\Z^d$ by $i$ which does not exceed $\ba$.

\begin{definition}[Discrete persistence modules]\label{def:discrete}We call a persistence module $M:\R^d\rightarrow \C$ \emph{discrete} if there exists a grid-like map $i:\Z^d\rightarrow \R^d$ such that for each $\ba\in \R^d$, the morphism $\varphi_M(\lfloor\ba\rfloor_i,\ba):M_{\lfloor\ba\rfloor_i}\rightarrow M_{\ba}$ is an isomorphism.  

\end{definition}

Let $\alpha\in (0,\infty)$. For any $t\in \R$, let $\lfloor t\rfloor_\alpha\in \alpha\Z$ be the greatest element in $\alpha\Z$ which does not exceed $t$. Given any DMS $\gammax$, we define the \emph{$\alpha$-discretization} of $\gamma_X$:

\begin{definition}[Discretization of a DMS] Let $\gammax$ be any DMS and let $\alpha\in (0,\infty)$. The \emph{$\alpha$-discretization} of $\gamma_X$ is the $\R$-parametrized family of finite (pseudo-)metric spaces $\gamma_X^\alpha:=\left\{\left(X,d_X^{\alpha\Z}(t)\right):t\in \R \right\}$, where \[d_X^{\alpha\Z}(t):=d_X(\lfloor t \rfloor_\alpha):X\times X\rightarrow \R_+.\]
\end{definition}

Notice that the discretization $\gamma_X^\alpha$ of $\gamma_X$ does \emph{not} necessarily satisfy Definition \ref{def:dms} \ref{item:dynamic2} and  \ref{item:dynamic3} and hence $\gamma_X^\alpha$ does \emph{not} deserve to be called a DMS. However, for convenience, we will call $\gamma_X^\alpha$ the \emph{$\alpha$-discretized DMS} of $\gamma_X$ or simply the \emph{discretized} DMS.

We can regard $\dintl$ as an extended pseudometric on a collection containing both all DMSs \emph{and} all discretized DMSs: Indeed, items \ref{item:dynamic2} and  \ref{item:dynamic3} in Definition \ref{def:dms} are not necessary to claim that $\dintl$ satisfies the triangle inequality (see the proof of  \cite[Theorem 9.14]{kim2017stable} in \cite[Section 11.4.2]{kim2017stable}).

A DMS $\gammax$ is said to be \emph{$l$-Lipschitz} if  $d_X(\cdot)(x,x'):\R\rightarrow \R_+$ is $l$-Lipschitz for every $x,x'\in X$. Assuming that $\gamma_X$ is $l$-Lipschitz, the smaller the resolution parameter $\alpha$ is, the closer the discretized DMS $\gamma_X^\alpha$ to $\gamma_X$ is: 

\begin{proposition}\label{prop:discretization error}Let $\gammax$ be any $l$-Lipschitz DMS. Then,
\[\dintl\left(\gamma_X,\gamma_X^{\alpha}\right)\leq l\alpha.\]
\end{proposition}

Note that for the discretized DMS $\gamma_X^\alpha$, we can define the rank invariant and the Betti-$0$ function of $\gamma_X^\alpha$ in the same way as in Definitions \ref{def:rank invariant_brief} and \ref{def:betti0}, respectively. Furthermore, in a bounded time interval $I\subset \R$, it is not difficult to check that both the Betti-$0$ function $\beta_0^{\gamma_X^\alpha}$ and the rank invariant $\rk_k(\gamma_X^\alpha),\ k\in\Z_+$ are discrete (Definition \ref{def:discrete}). Therefore, one can straightforwardly utilize the results in Section \ref{sec:algorithm} for computing $\dint$. 

\begin{proposition}[Approximating $\dintl$ from below with discretized DMSs]\label{prop:approximation} 
Let $\gammax$ and $\gammay$ be any two $l$-Lipschitz DMSs. 

\[\dint\left(\beta_0^{\gamma_X^\alpha},\beta_0^{\gamma_Y^\alpha}\right)-4l\alpha\hspace{3mm}\leq\hspace{3mm} 2\cdot\dintl(\gamma_X,\gamma_Y)\hspace{3mm}\mbox{and,}\]
\[\dint\left(\rk_k(\gamma_X^\alpha),\rk_k(\gamma_Y^\alpha)\right)-4l\alpha\hspace{3mm}\leq\hspace{3mm} 2\cdot\dintl(\gamma_X,\gamma_Y),\hspace{3mm} k\in\Z_+.\]
\end{proposition}

 \begin{proof}[Proof of Proposition \ref{prop:discretization error}]
 For ease of notation, we prove the statement assuming that $\alpha=1$, without loss of generality. Consider the tripod $R:X \xtwoheadleftarrow{\mathrm{id}_X} X \xtwoheadrightarrow{\mathrm{id}_X} X$ (Definition \ref{def:tripod}). We prove that $R$ is a $l$-tripod between $\gamma_X$ and $\gamma_X^{\alpha\Z}$ (Definition \ref{def:distortion}). Fix $t\in \R$.  Since $\lfloor t\rfloor \in [t-1,t+1]=[t]^1$, it is clear that $\bigvee_{[t]^1}d_X\leq_R d_X^{\alpha}(t)$ and hence $\bigvee_{[t]^1}d_X\leq_R d_X^{\alpha\Z}(t)+2l.$ 
It remains to show that $\bigvee_{[t]^1}d_X^{\alpha\Z}\leq_R d_X(t)+2l.$ Observe that, for any $x,x'\in X$, $\left(\bigvee_{[t]^1}d_X^{\alpha\Z}\right)(x,x')$ is the minimum among $d_X(\lfloor t\rfloor-1)(x,x'),\ d_X(\lfloor t\rfloor)(x,x')$ and $d_X(\lfloor t\rfloor+1)(x,x').$ Also, observe that all of $\lfloor t\rfloor -1, \lfloor t\rfloor, \lfloor t\rfloor+1$ belong to the closed interval $[t]^2=[t-2,t+2]$. Therefore, invoking that $\gamma_X$ is $l$-Lipschitz, for any $x,x'\in X$, \[\left(\bigvee_{[t]^1}d_X^{\alpha\Z}\right)(x,x')\leq d_X(t)(x,x')+2l.\]
This implies that $\bigvee_{[t]^1}d_X^{\Z}\leq_R d_X(t)+2l$, as desired.
\end{proof}

\begin{proof}[Proof of Proposition \ref{prop:approximation}]We have
\begin{align*}
    \dintl(\gamma_X^\alpha,\gamma_Y^\alpha)&\leq \dintl(\gamma_X^\alpha,\gamma_X)+\dintl(\gamma_X,\gamma_Y)+\dintl(\gamma_Y,\gamma_Y^\alpha)&\mbox{by the triangle inequality,}\\
    &\leq 2l\alpha+\dintl(\gamma_X,\gamma_Y)&\mbox{by Proposition \ref{prop:discretization error}.}
\end{align*}
Also, by Theorem \ref{thm:betti-0 stability},  we obtain $\dint\left(\beta_0^{\gamma_X^\alpha},\beta_0^{\gamma_Y^\alpha}\right)\leq 2\cdot\dintl(\gamma_X^\alpha,\gamma_Y^\alpha)$, and in turn the first inequality in the statement. The second inequality can be proved in a similar way.
\end{proof}

\section{Relationship between the rank invariant and CROCKER-plot}\label{sec:Crocker}

We relate the rank invariant of a DMS  to the CROCKER plot of \cite{topaz}:
\begin{definition}[The CROCKER plots of a DMS \cite{topaz}]\label{def:crocker}Let $\gammax$ be a DMS. For $k\in\Z_+$, the $k$-th CROCKER plot $C_k(\gamma_X)$ of $\gamma_X$ is a map $\R\times\R_+\rightarrow \Z_+$ sending $(t,\delta)\in\R\times\R_+$ to the dimension of the vector space $\Hrm_k\left(\mathcal{R}_\delta(X,d_X(t))\right)$.
\end{definition}

Let $\gammax$ be any DMS. Note that for any time $t_0\in\R$ and scale $\delta_0\in\R_+$, the value of $\rk_k(\gamma_X)$ associated to the \emph{repeated} pair $([t_0,t_0],\delta_0),([t_0,t_0],\delta_0)\in \U\times\R_+$ is identical to the dimension of the vector space $\Hrm_k\left(\mathcal{R}_{\delta_0}(X,d_X(t_0))\right)$, i.e. $C_k(\gamma_X)(t_0,\delta_0)$. This implies that $\rk_k(\gamma_X)$ is an \emph{enriched version} of the $k$-th CROCKER plot $C_k(\gamma_X)$ of $\gamma_X$.\footnote{To illustrate this, the $0$-th CROCKER plot $C_0(\gamma_X)$ is obtained by restricting $\beta_0^{\gamma_X}$ to the front diagonal vertical plane $\{[t,t]: t\in\R\}\times \R_+\subset \U\times \R_+$, which is colored brown in the middle picture of Figure \ref{fig:betti_zero}.} Therefore, Theorem \ref{thm:rank k stability} can be interpreted somehow as establishing the stability of the CROCKER plots of a DMS.

Recall Definition \ref{def:betti0}, the Betti-$0$ function of a DMS. 

\begin{remark}[Comparison between the Betti-$0$ function and the $0$-th CROCKER plot]\label{rem:croker-0}

Consider the DMSs $\gamma_X$ and $\gamma_Y$ in Figure \ref{fig:intro}. Since  the two metric spaces $\gamma_X(t)$ and $\gamma_Y(t)$ are isometric at \emph{each} time $t\in\R$, the two CROCKER plots $C_0(\gamma_X)$ and $C_0(\gamma_Y)$ are identical. On the other hand, the Betti-$0$ function $\beta_0^{\gamma_X}$ is distinct from $\beta_0^{\gamma_Y}$ as illustrated in Figure \ref{fig:betti_zero}. This implies that, in comparison with the $0$-th CROCKER plot, the Betti-$0$ function is more sensitive invariant of a DMS.
\end{remark}

\section{Other relevant metrics}\label{sec:other metrics}

\paragraph{Bottleneck distance}
Let us define:
\begin{itemize}
    \item 
$\ER:=\R\cup\{+\infty,-\infty\}$,
\item $\UU:=\{(u_1,u_2)\in \R^2: u_1\leq u_2\}$, which is the upper-half plane above the line $y=x$ in $\R^2$.  
\item  $\EU:=\{(u_1,u_2)\in \ER^2: u_1\leq u_2\}$, which is the upper-half plane above the line $y=x$ in the extended plane $\ER^2$. 
\end{itemize}

 For $\buu=(u_1,u_2),\ \bv=(v_1,v_2)\in \EU$, let \[\norm{\buu-\bv}_\infty:=\max\left(\abs{u_1-v_1}, \abs{u_2-v_2}\right).\]
 
 Let $X_1$ and $X_2$ be multisets of points. Let $\alpha:X_1\nrightarrow X_2$ be a matching, i.e. a partial injection. By $\dom(\alpha)$ and $\im(\alpha)$, we denote the points in $X_1$ and $X_2$ respectively, which are matched by $\alpha$.
\begin{definition}[The bottleneck distance {\cite{cohen2007stability}}]\label{def:bottleneck} Let $X_1,X_2$ be multisets of points in $\EU$. Let $\alpha:X_1\nrightarrow X_2$ be a matching. We call $\alpha$ an \emph{$\eps$-matching} if 
\begin{enumerate}[label=(\roman*)]
    \item for all $\buu\in \dom(\alpha)$, $\norm{\buu-\alpha(\buu)}_\infty\leq \eps$,
    \item for all $\buu=(u_1,u_2)\in X_1\setminus \dom(\alpha)$, $u_2-u_1 \leq 2\eps$,
    \item for all $\bv=(v_1,v_2)\in X_2 \setminus \im(\alpha)$, $v_2-v_1 \leq 2\eps$.
\end{enumerate}
Their bottleneck distance $\bott(X_1,X_2)$ is defined  as the infimum of $\eps\in[0,\infty)$ for which there exists an $\eps$-matching $\alpha:X_1\nrightarrow X_2$.
\end{definition}

\paragraph{Erosion distance.} Recently, Patel generalized the notion of persistence diagrams and proposed a new metric, the \emph{erosion distance}, for comparing generalized persistence diagrams \cite{patel2018generalized}. We review a particular case of the erosion distance. Let $\Pb$ and $\Q$ be any two posets.  Given any two maps $f,g:\Pb\rightarrow \Q$, we write $f\leq g$ if $f(p)\leq g(p)$ for all $p\in \Pb$.

Let $\UU:=\{(x,y)\in\R^2:x\leq y\}$ equipped with the partial order inherited from $\Rop\times \R$. For any $\eps\in [0,\infty),$ let $\vec{\eps}:=(-\eps,\eps)\in \UU$.  Given any map $Y:\UU\rightarrow \Z_+$ and $\eps\in [0,\infty)$, define another map $\nabla_\eps Y:\UU\rightarrow \Z_+$ as $\nabla_\eps Y(\bu):=Y(\bu+\vec{\eps}).$ If $Y$ is order-reversing, it is clear that $\nabla_\eps Y \leq Y.$

\begin{definition}[Erosion distance {\cite{patel2018generalized}}]\label{def:erosion}Let $Y_1,Y_2:\UU \rightarrow \Z_+$ be any two  order-reversing maps. The erosion distance between $Y_1$ and $Y_2$ is defined as 
\[\dero(Y_1,Y_2):=\inf\left\{\eps\in[0,\infty): \nabla_\eps Y_i\leq Y_j,\ \mbox{for}\ i,j\in\{1,2\}\right\},\] with the convention that $\dero(Y_1,Y_2)=\infty$ when there is no $\eps\in [0,\infty)$ satisfying the condition in the above set. 
\end{definition}

Note that since $\UU$ is a subposet of $\Rop\times \R$, we can regard $\dero$ is a particular case of $\dinttwo$ from Section \ref{subsec:interleaving}. The erosion distance is further generalized in \cite{puuska2017erosion}.

\paragraph{Matching distance \cite{cerri2013betti,landi2018rank}.} 
In brief, the matching distance $\dmatch$ compares rank invariants via one-dimensional reduction along lines. Namely, for any $M,N:\R^d\rightarrow \vect$, the matching distance between $\rk(M)$ and $\rk(N)$ is defined as 
\begin{equation}\label{eq:matching distance}
    \dmatch(\rk(M),\rk(N)):=\sup_{L:u=s\vec{m}+b}m^\ast \bott(\B(M|_L),\B(N|_L)),
\end{equation}
where $L$ varies in the set of all the lines parameterized by $u=s\vec{m}+b$, with $m^\ast:=\min_i m_i>0$, $\max_i m_i=1$, $\sum_{i}^n b_i=0$. Specifically, $\dmatch$ is upper bounded by $\dint^\vect$ \cite{landi2018rank}. We briefly discuss about the algorithms for $\dmatch$ and their computational cost:
\begin{itemize}
    \item For $d=1$, the RHS of equation (\ref{eq:matching distance}) reduces to the bottleneck distance between the barcodes of $M$ and $N$. The bottleneck distance can be computed in time $O(n^{1.5}\log n)$ where $n$ is the total cardinality of the two barcodes \cite{kerber2017geometry}. See also \cite{cerri2014comparing}.
    \item For $d=2$, $\dmatch$ can be computed exactly in time $O(n^{11})$ where $n$ is the size of \emph{finite presentations} of $M$ and $N$ \cite{kerber2018exact}.
    \item For $d\geq2$, algorithms for approximating  $\dmatch$ within any threshold $\eps>0$ are proposed in \cite{biasotti2011new,cerri2011new}. In particular, for the case $d\geq 3$ which is of our interest, the running time for the proposed algorithm is proportional to $\left(\frac{d}{\eps}\right)^d$ in the worst case \cite[Section 3.1]{cerri2011new}.
\end{itemize}

\paragraph{Dimension distance \cite[Section 4]{dey2018computing}.} Let $M,N:\R^d\rightarrow \vect$ be any two persistence modules. If $M,N$ are \emph{nice}\footnote{A persistence module $M:\R^d\rightarrow \vect$ is nice if there exists a value $\eps_0\in \R_+$ so that for every $\eps<\eps_0$, each internal morphism $\varphi_M(\ba,\ba+\vec{\eps})$ is either injective or surjective (or both).}, then the \emph{dimension distance $d_0$} between $\dm(M)$ and $\dm(N)$ serves as a lower bound for $\dint^\vect(M,N)$ \cite[Theorem 39]{dey2018computing}. A strength of $d_0$ is the computational efficiency. Let $M',N':[n]^d\rightarrow \vect$ be any two finite persistence modules. The entire computation for $d_0(\dm(M'),\dm(N'))$ takes only $O(n^2\log n)$ \cite[Section 4.2]{dey2018computing}.

If a persistence module $M$ is obtained by applying the $0$-th homology functor to the spatiotemporal Rips filtration of a DMS $\gamma_X$ (Definition \ref{def:spatiotemporal Rips}), then every internal morphim $\varphi_M(\cdot,\cdot)$ is surjective, and hence $M$ is nice. Specifically, $\dm(M)$ coincides with the Betti-$0$ function $\beta_0^{\gamma_X}$ (Definition \ref{def:betti0}). Therefore, one can utilize $d_0$ for comparing Betti-$0$ functions of DMSs and for obtaining a lower bound of $\dintl$ (by virtue of Theorem \ref{thm:main2}).

On the other hand, for $k\geq 1$, a persistence module $M$ obtained by applying the $k$-th homology functor to the spatiotemporal Rips filtration of a DMS does not necessarily satisfy the ``nice" condition. This prevents us from freely utilizing $d_0$ in order to obtain a lower bound for $\dintl$.

\section{Stability of the single linkage hierarchical clustering method}\label{sec:static metric spaces}

We review the single linkage hierarchical clustering (SLHC) method and its stability under the Gromov-Hausdorff distance. We begin by reviewing the Gromov-Hausdorff distance. 

\subsection{The Gromov-Hausdorff distance}\label{subsec:GH}
The Gromov-Hausdorff distance $\dgh$  (Definition \ref{def:the GH}) measures how far two metric spaces are from being isometric. 
 
Let $(X,d_X)$ and $(Y,d_Y)$ be any two metric spaces and let $\tripod$ be a tripod between $X$ and $Y$. Then, the \emph{distortion} of $R$ is defined as \[\displaystyle\dis(R):=\sup_{\substack{z,z'\in Z}}\abs{d_X\left(\varphi_X(z),\varphi_X(z')\right)-d_Y\left(\varphi_Y(z),\varphi_Y(z')\right)}.\]  
\begin{definition}[Gromov-Hausdorff distance {\cite[Section 7.3.3]{burago}}] \label{def:the GH} Let $(X,d_X)$ and $(Y,d_Y)$ be any two metric spaces. Then, 
	\[\dgh\left((X,d_X),(Y,d_Y)\right)=\frac{1}{2}\inf_R\ \dis(R),\]
	where the infimum is taken over all tripods $R$ between $X$ and $Y$. In particular, any tripod $R$ between $X$ and $Y$ with $\dis(R)\leq \eps$ is said to be an \emph{$\eps$-tripod} between $(X,d_X)$ and $(Y,d_Y)$.
\end{definition}

The computation cost of $\dgh$ leads to NP-hard problem, even for metric spaces of simple structure  \cite{agarwal2015computing,schmiedl2017computational}. Therefore, one of practical approaches for estimating $\dgh$ is to search for tractable lower bounds.


\subsection{Single linkage hierarchical clustering (SLHC) method}\label{sec:SLHC}

Let $(X,d_X)$ be a finite metric space. For each $\delta\in \R_+$, we define the equivalence relation $\sim_\delta$ on $X$ as
\[x\sim_\delta x'\ \mbox{if and only if}\ \exists x=x_0,\ldots,x_n \ \mbox{in}\ X, \mbox{s.t.}\ d_X(x_i,x_{i+1})\leq \delta.\]
Observe that for any $\delta\leq \delta'$ in $\R_+$, the inclusion $\sim_\delta \ \subset \ \sim_{\delta'}$ holds, leading to $(X/\sim_\delta) \leq (X/\sim_{\delta'})$ in $\Part(X)$ (Definition \ref{def:part}).
\begin{definition}[The dendrogram from the SLHC]\label{def:SLHC dendrogram}Let $(X,d_X)$ be a finite metric space. The dendrogram $\theta(X,d_X):\R_+\rightarrow \Part(X)$ defined by sending $\delta\in \R_+$ to $X/\sim_\delta$ is called the \emph{SLHC dendrogram} of $(X,d_X)$.
\end{definition}

\paragraph{The ultrametric induced by the single linkage hierarchical clustering method \cite{clustum}.} An ultrametric space $(X,u_X)$ is a metric space satisfying the \emph{strong triangle inequality}: for all $x,x',x''\in X$, $u_X(x,x')\leq \max\left\{u_X(x,x''),u_X(x'',x')  \right\}$.

Let $(X,d_X)$ be a finite metric space and consider its SLHC dendrogram $\theta(X,d_X):\R_+\rightarrow \Part(X)$.
For any $x,x'\in X$, define
\[u_X(x,x'):=\min\{\delta\in [0,\infty):\  \mbox{$x,x'$ belong to the same block of $X/\sim_\delta$}\}.\]
It is not difficult to check that $u_X:X\times X\rightarrow \R_+$ is a ultrametric and that $u_X(x,x')\leq d_(x,x')$, for all $x,x'\in X$.

\begin{definition}[The ultrametrics induced by the single linkage hierarchical clustering \cite{clustum}]\label{def:SLHC} Given any finite metric space $(X,d_X)$, the ultrametric space $(X,u_X)$ defined as above is said to be \emph{the ultrametric space induced by the SLHC on $(X,d_X)$} and we write $(X,u_X)=\sing(X,d_X).$
\end{definition}
The assignment $(X,d_X)\mapsto \sing(X,d_X)$ is known to be 1-Lipschitz with respect to the Gromov-Hausdorff distance: 

\begin{theorem}[Stability of the SLHC  {\cite{clustum}}]\label{thm:stability of SLHC 1} For any two finite metric spaces $(X,d_X)$ and $(Y,d_Y)$, let $(X,u_X)$ and $(Y,u_Y)$ be the ultrametric spaces induced from  $(X,d_X)$ and $(Y,d_Y)$ by the SLHC method. Then,
\begin{equation}\label{eq:stability of SLHC 1}
    \dgh((X,u_X),(Y,u_Y))\leq \dgh((X,d_X),(Y,d_Y)).
\end{equation}
\end{theorem}

\begin{remark}\label{rem:computational cost} The term $\dgh((X,u_X),(Y,u_Y))$ in (\ref{eq:stability of SLHC 1}) cannot be approximated within any
factor less than 3 in polynomial time, unless P = NP \cite[Theorem 3]{kim2018CCCG}. Therefore, in a practical viewpoint, it is desirable to find another lower bound for $\dgh$.
\end{remark}


The Gromov-Hausdorff distance can be bounded from below by the bottleneck distance between persistence diagrams associated to Rips filtrations: see inequality (\ref{eq:bottleneck stability vs GH}).
Computing the LHS of inequality (\ref{eq:bottleneck stability vs GH}) can be carried out in polynomial time \cite{kerber2017geometry}.

\begin{remark}\label{rem:relationship}
Observe that both of the LHSs of the  inequalities in (\ref{eq:stability of SLHC 1}) and (\ref{eq:bottleneck stability vs GH}) with $k=0$ measure the difference between clustering features of $(X,d_X)$ and $(Y,d_Y)$. In fact, for any two finite metric spaces $(X,d_X)$ and $(Y,d_Y)$, the persistence modules $\Hrm_0\left(\ripsss(X,d_X)\right)$ and $\Hrm_0\left(\ripsss(Y,d_Y)\right)$ are isomorphic to $\Hrm_0\left(\ripsss(X,u_X)\right)$ and $\Hrm_0\left(\ripsss(Y,u_Y)\right)$, respectively. Therefore, 
	\[\bott\left(\dgm_0\left(\ripsss(X,d_X)\right), \dgm_0\left(\ripsss(Y,d_Y)\right)\right)\leq \dgh\left((X,u_X),(Y,u_Y)\right) \leq 2\cdot\dgh\left((X,d_X),(Y,d_Y)\right).\]	
\end{remark}


\bibliographystyle{abbrv}
\bibliography{biblio}

\begin{thebibliography}{10}

\bibitem{agarwal2015computing}
P.~K. Agarwal, K.~Fox, A.~Nath, A.~Sidiropoulos, and Y.~Wang.
\newblock Computing the gromov-hausdorff distance for metric trees.
\newblock In {\em International Symposium on Algorithms and Computation}, pages
  529--540. Springer, 2015.

\bibitem{babichev2018robust}
A.~Babichev, D.~Morozov, and Y.~Dabaghian.
\newblock Robust spatial memory maps encoded by networks with transient
  connections.
\newblock {\em PLoS computational biology}, 14(9):e1006433, 2018.

\bibitem{bauer2017persistence}
U.~Bauer, H.~Edelsbrunner, G.~Jablonski, and M.~Mrozek.
\newblock Persistence in sampled dynamical systems faster.
\newblock {\em arXiv preprint arXiv:1709.04068}, 2017.

\bibitem{bendich2013homology}
P.~Bendich, H.~Edelsbrunner, D.~Morozov, A.~Patel, et~al.
\newblock Homology and robustness of level and interlevel sets.
\newblock {\em Homology, Homotopy and Applications}, 15(1):51--72, 2013.

\bibitem{benkert2008reporting}
M.~Benkert, J.~Gudmundsson, F.~H{\"u}bner, and T.~Wolle.
\newblock Reporting flock patterns.
\newblock {\em Computational Geometry}, 41(3):111--125, 2008.

\bibitem{biasotti2011new}
S.~Biasotti, A.~Cerri, P.~Frosini, and D.~Giorgi.
\newblock A new algorithm for computing the 2-dimensional matching distance
  between size functions.
\newblock {\em Pattern Recognition Letters}, 32(14):1735--1746, 2011.

\bibitem{bjerkevik2018computing}
H.~B. Bjerkevik.
\newblock Computing the interleaving distance is {NP}-hard.
\newblock {\em arXiv preprint arXiv:1811.09165}, 2018.

\bibitem{bjerkevik2017computational}
H.~B. Bjerkevik and M.~B. Botnan.
\newblock Computational complexity of the interleaving distance.
\newblock {\em arXiv preprint arXiv:1712.04281}, 2017.

\bibitem{botnan2018algebraic}
M.~Botnan and M.~Lesnick.
\newblock Algebraic stability of zigzag persistence modules.
\newblock {\em Algebraic \& Geometric Topology}, 18(6):3133--3204, 2018.

\bibitem{bubenik2014categorification}
P.~Bubenik and J.~A. Scott.
\newblock Categorification of persistent homology.
\newblock {\em Discrete \& Computational Geometry}, 51(3):600--627, 2014.

\bibitem{buchin2013trajectory}
K.~Buchin, M.~Buchin, M.~J. van Kreveld, B.~Speckmann, and F.~Staals.
\newblock Trajectory grouping structure.
\newblock {\em JoCG}, 6(1):75--98, 2015.

\bibitem{burago}
D.~Burago, Y.~Burago, and S.~Ivanov.
\newblock {\em A course in metric geometry}, volume~33.
\newblock American Mathematical Soc., 2001.

\bibitem{Carl09}
G.~Carlsson.
\newblock Topology and data.
\newblock {\em Bull. Amer. Math. Soc.}, 46:255--308, 2009.

\bibitem{zigzag}
G.~Carlsson and V.~De~Silva.
\newblock Zigzag persistence.
\newblock {\em Foundations of computational mathematics}, 10(4):367--405, 2010.

\bibitem{carlsson2009zigzag}
G.~Carlsson, V.~De~Silva, and D.~Morozov.
\newblock Zigzag persistent homology and real-valued functions.
\newblock In {\em Proceedings of the twenty-fifth annual symposium on
  Computational geometry}, pages 247--256. ACM, 2009.

\bibitem{clustum}
G.~Carlsson and F.~M{\'e}moli.
\newblock Characterization, stability and convergence of hierarchical
  clustering methods.
\newblock {\em Journal of Machine Learning Research}, 11:1425--1470, 2010.

\bibitem{carlsson2009theory}
G.~Carlsson and A.~Zomorodian.
\newblock The theory of multidimensional persistence.
\newblock {\em Discrete \& Computational Geometry}, 42(1):71--93, 2009.

\bibitem{cerri2014comparing}
A.~Cerri, B.~Di~Fabio, G.~Jab{\l}o{\'n}ski, and F.~Medri.
\newblock Comparing shapes through multi-scale approximations of the matching
  distance.
\newblock {\em Computer Vision and Image Understanding}, 121:43--56, 2014.

\bibitem{cerri2013betti}
A.~Cerri, B.~D. Fabio, M.~Ferri, P.~Frosini, and C.~Landi.
\newblock Betti numbers in multidimensional persistent homology are stable
  functions.
\newblock {\em Mathematical Methods in the Applied Sciences},
  36(12):1543--1557, 2013.

\bibitem{cerri2011new}
A.~Cerri and P.~Frosini.
\newblock A new approximation algorithm for the matching distance in
  multidimensional persistence.
\newblock 2011.

\bibitem{CCG09}
F.~Chazal, D.~Cohen-Steiner, M.~Glisse, L.~J. Guibas, and S.~Oudot.
\newblock Proximity of persistence modules and their diagrams.
\newblock In {\em Proceeding of twenty-fifth ACM Symposium on Computational
  Geommetry}, pages 237--246, 2009.

\bibitem{dghrips}
F.~Chazal, D.~Cohen-Steiner, L.~J. Guibas, F.~M{\'e}moli, and S.~Y. Oudot.
\newblock Gromov-{H}ausdorff stable signatures for shapes using persistence.
\newblock In {\em Proc. of SGP}, 2009.

\bibitem{chazal2014persistence}
F.~Chazal, V.~De~Silva, and S.~Oudot.
\newblock Persistence stability for geometric complexes.
\newblock {\em Geometriae Dedicata}, 173(1):193--214, 2014.

\bibitem{cohen2007stability}
D.~Cohen-Steiner, H.~Edelsbrunner, and J.~Harer.
\newblock Stability of persistence diagrams.
\newblock {\em Discrete \& Computational Geometry}, 37(1):103--120, 2007.

\bibitem{CEM06}
D.~Cohen-Steiner, H.~Edelsbrunner, and D.~Morozov.
\newblock Vines and vineyards by updating persistence in linear time.
\newblock In {\em Proceedings of the twenty-second annual symposium on
  Computational geometry}, pages 119--126. ACM, 2006.

\bibitem{de2016categorified}
V.~De~Silva, E.~Munch, and A.~Patel.
\newblock Categorified {R}eeb graphs.
\newblock {\em Discrete \& Computational Geometry}, 55(4):854--906, 2016.

\bibitem{dey2018persistent}
T.~K. Dey, M.~Juda, T.~Kapela, J.~Kubica, M.~Lipinski, and M.~Mrozek.
\newblock Persistent homology of {M}orse decompositions in combinatorial
  dynamics.
\newblock {\em arXiv preprint arXiv:1801.06590}, 2018.

\bibitem{dey2018computing}
T.~K. Dey and C.~Xin.
\newblock Computing bottleneck distance for $2 $-d interval decomposable
  modules.
\newblock {\em arXiv preprint arXiv:1803.02869}, 2018.

\bibitem{edelsbrunner2008persistent}
H.~Edelsbrunner and J.~Harer.
\newblock Persistent homology-a survey.
\newblock {\em Contemporary mathematics}, 453:257--282, 2008.

\bibitem{comptopo-herbert}
H.~Edelsbrunner and J.~Harer.
\newblock {\em Computational Topology - an Introduction}.
\newblock American Mathematical Society, 2010.

\bibitem{edelsbrunner2008time}
H.~Edelsbrunner, J.~Harer, A.~Mascarenhas, V.~Pascucci, and J.~Snoeyink.
\newblock Time-varying reeb graphs for continuous space--time data.
\newblock {\em Computational Geometry}, 41(3):149--166, 2008.

\bibitem{edelsbrunner2015persistent}
H.~Edelsbrunner, G.~Jab{\l}o{\'n}ski, and M.~Mrozek.
\newblock The persistent homology of a self-map.
\newblock {\em Foundations of Computational Mathematics}, 15(5):1213--1244,
  2015.

\bibitem{ghrist2008barcodes}
R.~Ghrist.
\newblock Barcodes: the persistent topology of data.
\newblock {\em Bulletin of the American Mathematical Society}, 45(1):61--75,
  2008.

\bibitem{giusti2016two}
C.~Giusti, R.~Ghrist, and D.~S. Bassett.
\newblock Two’s company, three (or more) is a simplex.
\newblock {\em Journal of computational neuroscience}, 41(1):1--14, 2016.

\bibitem{giusti2015clique}
C.~Giusti, E.~Pastalkova, C.~Curto, and V.~Itskov.
\newblock Clique topology reveals intrinsic geometric structure in neural
  correlations.
\newblock {\em Proceedings of the National Academy of Sciences},
  112(44):13455--13460, 2015.

\bibitem{gudmundsson2006computing}
J.~Gudmundsson and M.~van Kreveld.
\newblock Computing longest duration flocks in trajectory data.
\newblock In {\em Proceedings of the 14th annual ACM international symposium on
  Advances in geographic information systems}, pages 35--42. ACM, 2006.

\bibitem{gudmundsson2007efficient}
J.~Gudmundsson, M.~van Kreveld, and B.~Speckmann.
\newblock Efficient detection of patterns in 2d trajectories of moving points.
\newblock {\em Geoinformatica}, 11(2):195--215, 2007.

\bibitem{hajij2018visual}
M.~Hajij, B.~Wang, C.~Scheidegger, and P.~Rosen.
\newblock Visual detection of structural changes in time-varying graphs using
  persistent homology.
\newblock In {\em Pacific Visualization Symposium (PacificVis), 2018 IEEE},
  pages 125--134. IEEE, 2018.

\bibitem{huang2008modeling}
Y.~Huang, C.~Chen, and P.~Dong.
\newblock Modeling herds and their evolvements from trajectory data.
\newblock In {\em International Conference on Geographic Information Science},
  pages 90--105. Springer, 2008.

\bibitem{hwang2005mining}
S.-Y. Hwang, Y.-H. Liu, J.-K. Chiu, and E.-P. Lim.
\newblock Mining mobile group patterns: A trajectory-based approach.
\newblock In {\em PAKDD}, volume 3518, pages 713--718. Springer, 2005.

\bibitem{jeung2008discovery}
H.~Jeung, M.~L. Yiu, X.~Zhou, C.~S. Jensen, and H.~T. Shen.
\newblock Discovery of convoys in trajectory databases.
\newblock {\em Proceedings of the VLDB Endowment}, 1(1):1068--1080, 2008.

\bibitem{kahle2013limit}
M.~Kahle, E.~Meckes, et~al.
\newblock Limit the theorems for {B}etti numbers of random simplicial
  complexes.
\newblock {\em Homology, Homotopy and Applications}, 15(1):343--374, 2013.

\bibitem{kalnis2005discovering}
P.~Kalnis, N.~Mamoulis, and S.~Bakiras.
\newblock On discovering moving clusters in spatio-temporal data.
\newblock In {\em SSTD}, volume 3633, pages 364--381. Springer, 2005.

\bibitem{kerber2018exact}
M.~Kerber, M.~Lesnick, and S.~Oudot.
\newblock Exact computation of the matching distance on 2-parameter persistence
  modules.
\newblock In {\em Proceedings of the thirty-fifth International Symposium on
  Computational Geometry}, pages 46:1–--46:15, 2019.

\bibitem{kerber2017geometry}
M.~Kerber, D.~Morozov, and A.~Nigmetov.
\newblock Geometry helps to compare persistence diagrams.
\newblock {\em Journal of Experimental Algorithmics (JEA)}, 22:1--4, 2017.

\bibitem{kim2018CCCG}
W.~Kim and F.~M{\'e}moli.
\newblock Formigrams: Clustering summaries of dynamic data.
\newblock In {\em Proceedings of 30th Canadian Conference on Computational
  Geometry (CCCG18)}.

\bibitem{kim2017stable}
W.~Kim and F.~M\'emoli.
\newblock Stable signatures for dynamic graphs and dynamic metric spaces via
  zigzag persistence.
\newblock {\em arXiv preprint arXiv:1712.04064}, 2017.

\bibitem{zane}
W.~Kim, F.~M\'emoli, and Z.~Smith.
\newblock
  \href{https://research.math.osu.edu/networks/formigrams}{https://research.math.osu.edu/networks/formigrams}.

\bibitem{knight1988search}
W.~J. Knight.
\newblock Search in an ordered array having variable probe cost.
\newblock {\em SIAM Journal on Computing}, 17(6):1203--1214, 1988.

\bibitem{kostitsyna2015trajectory}
I.~Kostitsyna, M.~J. van Kreveld, M.~L{\"{o}}ffler, B.~Speckmann, and
  F.~Staals.
\newblock Trajectory grouping structure under geodesic distance.
\newblock In {\em 31st International Symposium on Computational Geometry, SoCG
  2015, June 22-25, 2015, Eindhoven, The Netherlands}, pages 674--688, 2015.

\bibitem{landi2018rank}
C.~Landi.
\newblock The rank invariant stability via interleavings.
\newblock In {\em Research in Computational Topology}, pages 1--10. Springer,
  2018.

\bibitem{lesnick}
M.~Lesnick.
\newblock The theory of the interleaving distance on multidimensional
  persistence modules.
\newblock {\em Found. Comput. Math.}, 15(3):613--650, June 2015.

\bibitem{li2010swarm}
Z.~Li, B.~Ding, J.~Han, and R.~Kays.
\newblock Swarm: Mining relaxed temporal moving object clusters.
\newblock {\em Proceedings of the VLDB Endowment}, 3(1-2):723--734, 2010.

\bibitem{mac2013categories}
S.~Mac~Lane.
\newblock {\em Categories for the working mathematician}.
\newblock Springer Science \& Business Media, 2013.

\bibitem{munch2013applications}
E.~Munch.
\newblock {\em Applications of persistent homology to time varying systems}.
\newblock PhD thesis, 2013.

\bibitem{oesterling2015computing}
P.~Oesterling, C.~Heine, G.~H. Weber, D.~Morozov, and G.~Scheuermann.
\newblock Computing and visualizing time-varying merge trees for
  high-dimensional data.
\newblock In {\em Topological Methods in Data Analysis and Visualization},
  pages 87--101. Springer, 2015.

\bibitem{parrish1997animal}
J.~K. Parrish and W.~M. Hamner.
\newblock {\em Animal groups in three dimensions: how species aggregate}.
\newblock Cambridge University Press, 1997.

\bibitem{patel2018generalized}
A.~Patel.
\newblock Generalized persistence diagrams.
\newblock {\em Journal of Applied and Computational Topology}, pages 1--23,
  2018.

\bibitem{puuska2017erosion}
V.~Puuska.
\newblock Erosion distance for generalized persistence modules.
\newblock {\em arXiv preprint arXiv:1710.01577}, 2017.

\bibitem{schmiedl14shape}
F.~Schmiedl.
\newblock {\em Shape Matching and Mesh Segmentation}.
\newblock PhD thesis, Technische Universit\"{a}t M\"{u}nchen, 2014.

\bibitem{schmiedl2017computational}
F.~Schmiedl.
\newblock Computational aspects of the {G}romov--{H}ausdorff distance and its
  application in non-rigid shape matching.
\newblock {\em Discrete \& Computational Geometry}, 57(4):854--880, 2017.

\bibitem{scolamiero2017multidimensional}
M.~Scolamiero, W.~Chach{\'o}lski, A.~Lundman, R.~Ramanujam, and S.~{\"O}berg.
\newblock Multidimensional persistence and noise.
\newblock {\em Foundations of Computational Mathematics}, 17(6):1367--1406,
  2017.

\bibitem{sumpter-collective}
D.~J. Sumpter.
\newblock {\em Collective animal behavior}.
\newblock Princeton University Press, 2010.

\bibitem{topaz}
C.~M. Topaz, L.~Ziegelmeier, and T.~Halverson.
\newblock Topological data analysis of biological aggregation models.
\newblock {\em PloS one}, 10(5):e0126383, 2015.

\bibitem{ulmer2018assessing}
M.~Ulmer, L.~Ziegelmeier, and C.~M. Topaz.
\newblock Assessing biological models using topological data analysis.
\newblock {\em arXiv preprint arXiv:1811.04827}, 2018.

\bibitem{van2016grouping}
A.~van Goethem, M.~J. van Kreveld, M.~L{\"{o}}ffler, B.~Speckmann, and
  F.~Staals.
\newblock Grouping time-varying data for interactive exploration.
\newblock In {\em 32nd International Symposium on Computational Geometry, SoCG
  2016, June 14-18, 2016, Boston, MA, {USA}}, pages 61:1--61:16, 2016.

\bibitem{van2015central}
M.~J. van Kreveld, M.~L{\"{o}}ffler, and F.~Staals.
\newblock Central trajectories.
\newblock {\em Journal of Computational Geometry}, 8(1):366--386, 2017.

\bibitem{van2016refined}
M.~J. van Kreveld, M.~L{\"{o}}ffler, F.~Staals, and L.~Wiratma.
\newblock A refined definition for groups of moving entities and its
  computation.
\newblock In {\em 27th International Symposium on Algorithms and Computation,
  {ISAAC} 2016, December 12-14, 2016, Sydney, Australia}, pages 48:1--48:12,
  2016.

\bibitem{vieira2009line}
M.~R. Vieira, P.~Bakalov, and V.~J. Tsotras.
\newblock On-line discovery of flock patterns in spatio-temporal data.
\newblock In {\em Proceedings of the 17th ACM SIGSPATIAL international
  conference on advances in geographic information systems}, pages 286--295.
  ACM, 2009.

\bibitem{wang2008efficient}
Y.~Wang, E.-P. Lim, and S.-Y. Hwang.
\newblock Efficient algorithms for mining maximal valid groups.
\newblock {\em The VLDB Journal—The International Journal on Very Large Data
  Bases}, 17(3):515--535, 2008.

\end{thebibliography}
\end{document}